\documentclass[10]{article}
\usepackage{amsfonts,amssymb,amsmath,amsthm,cite}
\usepackage{graphicx}

\usepackage[applemac]{inputenc}
\usepackage{amsmath,amssymb,amsthm, hyperref, euscript}
\usepackage[matrix,arrow,curve]{xy}
\usepackage{graphicx}
\usepackage{tabularx}
\usepackage{float}
\usepackage{hyperref}
\usepackage{tikz}
\usepackage{slashed}
\usepackage{mathrsfs}

\usetikzlibrary{matrix}

\usepackage[T1]{fontenc}
\usepackage{amsfonts,cite}
\usepackage{graphicx}

\usepackage[applemac]{inputenc}

\usepackage[sc]{mathpazo}
\DeclareFontFamily{T1}{pzc}{}
\DeclareFontShape{T1}{pzc}{m}{it}{1.8 <-> pzcmi8t}{}
\DeclareMathAlphabet{\mathpzc}{T1}{pzc}{m}{it}

\theoremstyle{plain}
\newtheorem{prop}{Proposition}[section]

\newtheorem{lem}[prop]{Lemma}
\newtheorem{cor}[prop]{Corollary}
\newtheorem{thm}[prop]{Theorem}
\newtheorem{theorem}[prop]{Theorem}
\newtheorem{lemma}[prop]{Lemma}
\newtheorem{proposition}[prop]{Proposition}
\newtheorem{corollary}[prop]{Corollary}

\theoremstyle{definition}
\newtheorem{defn}[prop]{Definition}
\newtheorem{empt}[prop]{}
\newtheorem{rem}[prop]{Remark}

\theoremstyle{definition}
\newtheorem{definition}[prop]{Definition}
\newtheorem{example}[prop]{Example}
\newtheorem{remark}[prop]{Remark}
\numberwithin{equation}{section}

\newcommand{\vertiii}[1]{{\left\vert\kern-0.25ex\left\vert\kern-0.25ex\left\vert #1
    \right\vert\kern-0.25ex\right\vert\kern-0.25ex\right\vert}}
\newcommand{\Ga}{\Gamma}                     
\newcommand{\Coo}{C^\infty}                  

\usepackage[sc]{mathpazo}
\newbox\ncintdbox \newbox\ncinttbox 
\setbox0=\hbox{$-$} \setbox2=\hbox{$\displaystyle\int$}
\setbox\ncintdbox=\hbox{\rlap{\hbox
    to \wd2{\hskip-.125em \box2\relax\hfil}}\box0\kern.1em}
\setbox0=\hbox{$\vcenter{\hrule width 4pt}$}
\setbox2=\hbox{$\textstyle\int$} \setbox\ncinttbox=\hbox{\rlap{\hbox
    to \wd2{\hskip-.175em \box2\relax\hfil}}\box0\kern.1em}

\newcommand{\Id}{\mathrm{Id}}                

\newcommand{\A}{\mathcal{A}}                 

\newcommand{\C}{\mathbb{C}}                  
\renewcommand{\H}{\mathcal{H}}               
\newcommand{\half}{\tfrac{1}{2}}             
\newcommand{\hookto}{\hookrightarrow}        
\newcommand{\La}{\Lambda}                    
\newcommand{\la}{\lambda}                    

\newcommand{\N}{\mathbb{N}}                  
\newcommand{\om}{\omega}                     
\newcommand{\eps}{\varepsilon}                    
\newcommand{\Q}{\mathbb{Q}}                  
\newcommand{\R}{\mathbb{R}}                  
\newcommand{\set}[1]{\{\,#1\,\}}             

\renewcommand{\SS}{\mathcal{S}}              
\DeclareMathOperator{\supp}{\mathfrak{supp}}         
\newcommand{\T}{\mathbb{T}}                  
\renewcommand{\th}{\theta}                   
\DeclareMathOperator{\tr}{tr}                
\newcommand{\Z}{\mathbb{Z}}                  
\renewcommand{\.}{\cdot}                     

\newcommand{\sS}{\mathcal{S}}       

\newcommand{\Om}{\Omega}       

\newcommand{\al}{\alpha}          
\newcommand{\bt}{\beta}           
\newcommand{\ga}{\gamma}          
\newcommand{\Th}{\Theta}          
\renewcommand{\th}{\theta}        

\def\<#1|#2>{\langle#1\stroke#2\rangle} 
\def\?#1|#2?{\{#1\stroke#2\}}        

\def\<#1,#2>{\langle#1,#2\rangle}            
\def\ee_#1{e_{{\scriptscriptstyle#1}}}       
\def\wick:#1:{\mathopen:#1\mathclose:}       

\newbox\ncintdbox \newbox\ncinttbox 
\setbox0=\hbox{$-$}
\setbox2=\hbox{$\displaystyle\int$}
\setbox\ncintdbox=\hbox{\rlap{\hbox
           to \wd2{\box2\relax\hfil}}\box0\kern.1em}
\setbox0=\hbox{$\vcenter{\hrule width 4pt}$}
\setbox2=\hbox{$\textstyle\int$}
\setbox\ncinttbox=\hbox{\rlap{\hbox
           to \wd2{\hskip-.05em\box2\relax\hfil}}\box0\kern.1em}

\newcommand{\stroke}{\mathbin|}   

\newcommand{\Aut}{\mathrm{Aut}}       

\newcommand{\be}{\begin{equation}}
\renewcommand{\ee}{\end{equation}}
\newcommand{\bea}{\begin{eqnarray}}
\newcommand{\eea}{\end{eqnarray}}
\newcommand{\bean}{\begin{eqnarray*}}
	\newcommand{\eean}{\end{eqnarray*}}
\newcommand{\brray}{\begin{array}}
	\newcommand{\erray}{\end{array}}

\title{Quantization of noncompact coverings}
\begin{document}
\maketitle  \setlength{\parindent}{0pt}
\begin{center}
\author{
{\textbf{Petr R. Ivankov*}\\
e-mail: * monster.ivankov@gmail.com \\
}
}
\end{center}

\vspace{1 in}

\noindent

\paragraph{}

The concept of  quantization consists in replacing commutative
quantities by noncommutative ones. In mathematical language
an algebra of continuous functions on a locally compact topological
space is replaced with a noncommutative $C^*$-algebra. Some classical topological notions have noncommutative generalizations. This article is concerned with a generalization of coverings. 
\tableofcontents
\section{Motivation. Preliminaries}
\paragraph*{}
Gelfand-Na\u{\i}mark theorem \cite{arveson:c_alg_invt} states the correspondence between  locally compact Hausdorff topological spaces and commutative $C^*$-algebras.

\begin{theorem}\label{gelfand-naimark}\cite{arveson:c_alg_invt} (Gelfand-Na\u{\i}mark). 
	Let $A$ be a commutative $C^*$-algebra and let $\mathcal{X}$ be the spectrum of $A$. There is the natural $*$-isomorphism $\gamma:A \to C_0(\mathcal{X})$.
\end{theorem}

So any (noncommutative) $C^*$-algebra may be regarded as a generalized (noncommutative)  locally compact Hausdorff topological space. Following theorem yields a pure algebraic description of finite-fold coverings of compact spaces.
\begin{theorem}\label{pavlov_troisky_thm}\cite{pavlov_troisky:cov}
	Suppose $\mathcal X$ and $\mathcal Y$ are compact Hausdorff connected spaces and $p :\mathcal  Y \to \mathcal X$
	is a continuous surjection. If $C(\mathcal Y )$ is a projective finitely generated Hilbert module over
	$C(\mathcal X)$ with respect to the action
	\begin{equation*}
	(f\xi)(y) = f(y)\xi(p(y)), ~ f \in  C(\mathcal Y ), ~ \xi \in  C(\mathcal X),
	\end{equation*}
	then $p$ is a finite-fold  covering.
\end{theorem} 
This article contains  pure algebraic generalizations of following topological objects:
\begin{itemize}
	\item Coverings of noncompact spaces,
	\item Infinite coverings. 
\end{itemize}

This article assumes elementary knowledge of following subjects:
\begin{enumerate}
	\item Set theory \cite{halmos:set},
	\item Category theory  \cite{spanier:at},
	\item Algebraic topology  \cite{spanier:at},
	\item $C^*$-algebras, $C^*$-Hilbert modules \cite{blackadar:ko,pedersen:ca_aut}.
	
\end{enumerate}

The words "set", "family" and "collection" are synonyms.

\break
Following table contains  special symbols.
\newline
\begin{tabular}{|c|c|}
\hline
Symbol & Meaning\\
\hline
&\\
$\hat{A}$ & Spectrum of a  $C^*$- algebra $A$  with the hull-kernel topology \\
 & (or Jacobson topology)\\
$A_+$  & Cone of positive elements of $C^*$- algebra, i.e. $A_+ = \left\{a\in A \ | \ a \ge 0\right\}$\\
$A^G$  & Algebra of $G$ - invariants, i.e. $A^G = \left\{a\in A \ | \ ga=a, \forall g\in G\right\}$\\
$\mathrm{Aut}(A)$ & Group of * - automorphisms of $C^*$- algebra $A$\\
$A''$  & Enveloping von Neumann algebra  of $A$\\

$B(\H)$ & Algebra of bounded operators on a Hilbert space $\H$\\
$\mathbb{C}$ (resp. $\mathbb{R}$)  & Field of complex (resp. real) numbers \\
$C(\mathcal{X})$ & $C^*$- algebra of continuous complex valued \\
 & functions on a compact  space $\mathcal{X}$\\
$C_0(\mathcal{X})$ & $C^*$- algebra of continuous complex valued functions on a locally \\
 &   compact  topological space $\mathcal{X}$ equal to $0$ at infinity\\
$C_c(\mathcal{X})$ & Algebra of continuous complex valued functions on a \\
 &  topological  space $\mathcal{X}$ with compact support\\
 $C_b(\mathcal{X})$ & $C^*$- algebra of bounded  continuous complex valued \\
  & functions on a locally compact topological space $\mathcal{X}$ \\
$G\left( \widetilde{\mathcal{X}}~ |~ \mathcal{X}\right) $ & Group of covering transformations of covering  $\widetilde{\mathcal{X}} \to \mathcal{X}$ \cite{spanier:at}  \\
$\H$ & Hilbert space \\
$\mathcal{K}= \mathcal{K}\left(\H \right) $ & $C^*$- algebra of compact operators on the separable Hilbert space $\H$  \\
$K(A)$ & Pedersen ideal of $C^*$-algebra $A$\\
$\varinjlim$ & Direct limit \\
$\varprojlim$ & Inverse limit \\
$M(A)$  & A multiplier algebra of $C^*$-algebra $A$\\
$\mathbb{M}_n(A)$  & The $n \times n$ matrix algebra over $C^*$-algebra $A$\\
$\mathbb{N}$  & A set of positive integer numbers\\
$\mathbb{N}^0$  & A set of nonnegative integer numbers\\



$U(A) \subset A $ & Group of unitary operators of algebra $A$\\
$\mathbb{Z}$ & Ring of integers \\

$\mathbb{Z}_n$ & Ring of integers modulo $n$ \\
$\overline{k} \in \mathbb{Z}_n$ & An element in $\mathbb{Z}_n$ represented by $k \in \mathbb{Z}$  \\
$X \backslash A$ & Difference of sets  $X \backslash A= \{x \in X \ | \ x\notin A\}$\\
$|X|$ & Cardinal number of a finite set $X$\\ 
$\left[x\right]$ & The range projection of element $x$ of a von Neumann algebra.\\ 
$f|_{A'}$& Restriction of a map $f: A\to B$ to $A'\subset A$, i.e. $f|_{A'}: A' \to B$\\ 
\hline
\end{tabular}

\break



  


\subsection{Prototype. Inverse limits of coverings in topology}\label{inf_to}
\subsubsection{Topological construction}
\paragraph*{}

This subsection is concerned with a topological construction of the inverse limit in the category of coverings. 
\begin{definition}\label{comm_cov_pr_defn}\cite{spanier:at}
	Let $\widetilde{\pi}: \widetilde{\mathcal{X}} \to \mathcal{X}$ be a continuous map. An open subset $\mathcal{U} \subset \mathcal{X}$ is said to be {\it evenly covered } by $\widetilde{\pi}$ if $\widetilde{\pi}^{-1}(\mathcal U)$ is the disconnected union of open subsets of $\widetilde{\mathcal{X}}$ each of which is mapped homeomorphically onto $\mathcal{U}$ by $\widetilde{\pi}$. A continuous map $\widetilde{\pi}: \widetilde{\mathcal{X}} \to \mathcal{X}$ is called a {\it covering projection} if each point $x \in \mathcal{X}$ has an open neighborhood evenly covered by $\widetilde{\pi}$. $\widetilde{\mathcal{X}}$ is called the {
		\it covering space} and $\mathcal{X}$ the {\it base space} of the covering.
\end{definition}
\begin{definition}\cite{spanier:at}
	A fibration $p: \mathcal{\widetilde{X}} \to \mathcal{X}$ with unique path lifting is said to be  {\it regular} if, given any closed path $\omega$ in $\mathcal{X}$, either every lifting of $\omega$ is closed or none is closed.
\end{definition}
\begin{definition}\cite{spanier:at}
	A topological space $\mathcal X$ is said to be \textit{locally path-connected} if the path components of open sets are open.
\end{definition}
Denote by $\pi_1$ the functor of fundamental group \cite{spanier:at}.
\begin{theorem}\label{locally_path_lem}\cite{spanier:at}
	Let $p: \widetilde{\mathcal X} \to \mathcal X$ be a fibration with unique path lifting and assume that a nonempty $\widetilde{\mathcal X}$ is a locally path-connected space. Then $p$ is regular if and only if for some $\widetilde{x}_0 \in  \widetilde{\mathcal X}$, $\pi_1\left(p\right)\pi_1\left(\widetilde{\mathcal X}, \widetilde{x}_0\right)$ is a normal subgroup of $\pi_1\left(\mathcal X, p\left(\widetilde{x}_0\right)\right)$.
\end{theorem}
\begin{definition}\label{cov_proj_cov_grp}\cite{spanier:at}
	Let $p: \mathcal{\widetilde{X}} \to \mathcal{X}$ be a covering projection.  A \textit{self-equivalence} is a homeomorphism $f:\mathcal{\widetilde{X}}\to\mathcal{\widetilde{X}}$ such that $p \circ f = p$. This group of such homeomorphisms is said to be the {\it group of covering transformations} of $p$ or the {\it covering group}. Denote by $G\left( \mathcal{\widetilde{X}}~|~\mathcal{X}\right)$ this group.
\end{definition}

\begin{proposition}\label{reg_cov_prop}\cite{spanier:at}
	If $p: \mathcal{\widetilde{X}} \to \mathcal{X}$ is a regular covering projection and $\mathcal{\widetilde{X}}$ is connected and locally path connected, then $\mathcal{X}$ is homeomorphic to space of orbits of $G\left( \mathcal{\widetilde{X}}~|~\mathcal{X}\right)$, i.e. $\mathcal{X} \approx \mathcal{\widetilde{X}}/G\left( \mathcal{\widetilde{X}}~|~\mathcal{X}\right) $. So $p$ is a principal bundle.
\end{proposition}
\begin{corollary}\label{top_cov_from_pi1_cor}\cite{spanier:at}
	Let $p: \widetilde{\mathcal X} \to \mathcal X$ be a fibration with a unique path lifting. If $ \widetilde{\mathcal X}$ is connected and locally path-connected and $\widetilde{x}_0 \in \widetilde{\mathcal X}$ then $p$ is regular if and only if $G\left(\widetilde{\mathcal X}~|~{\mathcal X} \right)$ transitively acts on each fiber of $p$, in which case 
	$$
	\psi: G\left(\widetilde{\mathcal X}~|~{\mathcal X} \right) \approx \pi_1\left(\mathcal X, p\left( \widetilde{x}_0\right)  \right) / \pi_1\left( p\right)\pi_1\left(\widetilde{\mathcal X}, \widetilde{x}_0 \right).  
	$$
\end{corollary}
\begin{remark}
	Above results are copied from \cite{spanier:at}. Below  the \textit{covering projection} word is replaced with \textit{covering}.
\end{remark}
\begin{definition}\label{top_comp_defn}\cite{munkres:topology}
A	\textit{ compactification} of a space $\mathcal X$ is a compact Hausdorff space $\mathcal Y$ containing $\mathcal X$ as a subspace and the closure $\overline{\mathcal X}$ of $\mathcal X$ is $\mathcal Y$, i.e $\overline{\mathcal X} = \mathcal Y$. 
\end{definition}
The algebraic construction requires following definition
\begin{definition}\label{top_cov_comp_defn}
	A   covering $\pi: \widetilde{   \mathcal X } \to \mathcal X$ is said to be a \textit{covering with compactification} if there are compactifications ${   \mathcal X } \hookto {   \mathcal Y }$ and $\widetilde{   \mathcal X } \hookto \widetilde{   \mathcal Y }$ such that:
	\begin{itemize}
		\item There is a covering $\overline{\pi}:\widetilde{   \mathcal Y }\to {   \mathcal Y }$,
		\item The covering $\pi$ is the restriction of $\overline{\pi}$, i.e. $\pi = \overline{\pi}|_{\widetilde{   \mathcal X }}$.
	\end{itemize}
\end{definition}
\begin{example}
	
	Let $g: S^1 \to S^1$ be an $n$-fold covering of a circle. Let $\mathcal X = \widetilde{\mathcal X} = S^1 \times \left[0,1\right)$. 
	The map
	\begin{equation*}
	\begin{split}
	\pi: \widetilde{   \mathcal X } \to \mathcal X,\\
	\pi = g \times \Id_{\left[0,1\right)}
	\end{split}
	\end{equation*}
	is an $n$-fold covering. If $\mathcal Y = \widetilde{\mathcal Y} = S^1 \times \left[0,1\right]$ then a compactification $\left[0,1\right) \hookto \left[0,1\right]$ induces  compactifications  $\mathcal X \hookto\mathcal Y$, $\widetilde{   \mathcal X } \hookto \widetilde{   \mathcal Y }$. The map
	\begin{equation*}
	\begin{split}
	\overline{\pi}: \widetilde{   \mathcal Y } \to \mathcal Y,\\
	\overline{\pi} = g \times \Id_{\left[0,1\right]}
	\end{split}
	\end{equation*}
	is a covering  such that $\overline{\pi}|_{\widetilde{   \mathcal X }}=\pi$. So if $n > 1$ then $\pi$ is  a nontrivial covering with  compactification.\\
\end{example}
\begin{example}
	Let $\mathcal X  = \C \backslash \{0\}$ be a complex plane with punctured 0, which is parametrized by the complex variable $z$. Let  $\mathcal X \hookto\mathcal Y$ be any compactification. If both $\left\{z'_n \in \mathcal X\right\}_{n \in \N}$, 	$\left\{z''_n \in \mathcal X\right\}_{n \in \N}$ are Cauchy sequences such that $\lim_{n \to \infty}\left|z'_n\right|=\lim_{n \to \infty}\left|z'_n\right| = 0$ then form $\lim_{n \to \infty}\left|z'_n-z''_n\right|= 0$ it turns out 
	\begin{equation}\label{x_0_eqn}
	x_0 = \lim_{n \to \infty} z'_n = \lim_{n \to \infty} z''_n \in \mathcal Y.
	\end{equation}
	
	If $\widetilde{   \mathcal X } = \mathcal X$ then for any $n \in \N$ there is a finite-fold covering 
	\begin{equation*}
	\begin{split}
	\pi: \widetilde{   \mathcal X } \to \mathcal X,\\
	z \mapsto z^n.
	\end{split}
	\end{equation*}
	If both  $\mathcal X \hookto\mathcal Y$, $\widetilde{   \mathcal X } \hookto \widetilde{   \mathcal Y }$ are compactifications, and $\overline{\pi}: \widetilde{   \mathcal Y } \to \mathcal Y$ is a covering such that $\overline{\pi}|_{\widetilde{   \mathcal X }} = \pi$ then from \eqref{x_0_eqn} it turns out $\overline{\pi}^{-1}\left(x_0 \right)= \left\{\widetilde{x}_0\right\}$ where $\widetilde{x}_0$ is the unique point such that following conditions hold:
	\begin{equation*}
	\begin{split}
	\widetilde{x}_0 = \lim_{n \to \infty} \widetilde{z}_n \in \widetilde{   \mathcal Y },\\
	\lim_{n \to \infty}\left|\widetilde{z}_n\right|= 0.
	\end{split}
	\end{equation*}
	It turns out $\left|\overline{\pi}^{-1}\left(x_0 \right) \right|=1$. However $\overline{\pi}$ is an $n$-fold covering and if $n >1$ then  $\left|\overline{\pi}^{-1}\left(x_0 \right) \right|=n>1$. It contradicts with $\left|\overline{\pi}^{-1}\left(x_0 \right) \right|=1$, and from the contradiction it turns out that for any $n > 1$ the map  $\pi$ is not a covering with  compactification.
\end{example}

\begin{definition}\label{top_sec_defn}
	The sequence of regular finite-fold coverings	
	\begin{equation*}
	\mathcal{X} = \mathcal{X}_0 \xleftarrow{}... \xleftarrow{} \mathcal{X}_n \xleftarrow{} ... 
	\end{equation*}
	is said to be a \textit{(topological)  finite covering sequence} if following conditions hold:
	\begin{itemize}
		\item   The space $\mathcal{X}_n$ is a second-countable \cite{munkres:topology}    locally compact connected Hausdorff space  for any $n \in \mathbb{N}^0$,
		\item If $k < l < m$ are any nonnegative integer numbers then there is the natural exact sequence
		$$
		\{e\}\to	G\left(\mathcal X_m~|~\mathcal X_l\right) \to 	G\left(\mathcal X_m~|~\mathcal X_k\right)\to 	G\left(\mathcal X_l~|~\mathcal X_k\right)\to \{e\}.
		$$ 
	\end{itemize} 
	For any finite covering sequence we will use a following notation
	\begin{equation*}
	\mathfrak{S} = \left\{\mathcal{X} = \mathcal{X}_0 \xleftarrow{}... \xleftarrow{} \mathcal{X}_n \xleftarrow{} ...\right\}= \left\{ \mathcal{X}_0 \xleftarrow{}... \xleftarrow{} \mathcal{X}_n \xleftarrow{} ...\right\},~~\mathfrak{S} \in \mathfrak{FinTop}.
	\end{equation*}
	
\end{definition}
\begin{example}
	Let $	\mathfrak{S} = \left\{	\mathcal{X} = \mathcal{X}_0 \xleftarrow{}... \xleftarrow{} \mathcal{X}_n \xleftarrow{} ... \right\}$ be a sequence of   locally compact connected Hausdorff spaces and finite-fold regular coverings such that $\mathcal X_n$ is  locally path-connected for any $n \in \mathbb{N}$. It follows from Lemma \ref{locally_path_lem} that if $p > q$ and $f_{pq}:\mathcal X_p \to \mathcal X_q$ then $\pi_1\left(f_{pq}\right)\pi_1\left(\mathcal X_p, x_0\right)$ is a normal subgroup of $\pi_1\left(\mathcal X_q, f_{pq}\left(x_0\right) \right)$. Otherwise from the Corollary \ref{top_cov_from_pi1_cor} it turns out 
	$$
	G\left(\mathcal X_p~|~\mathcal X_q\right) \approx \pi_1\left(\mathcal X_q, f_{pq}\left(x_0\right)\right) / \pi_1\left(f_{pq}\right)\pi_1\left(\mathcal X_p, x_0\right).
	$$
	If $k < l < m$ then a following sequence 
	\begin{equation*}
	\begin{split}
	\{e\}\to	\pi_1\left(\mathcal X_l, f_{ml}\left(x_0\right) \right)/ \pi_1\left(f_{ml}\right)\pi_1\left(\mathcal X_m, x_0\right) \to \\ \to
	\pi_1\left(\mathcal X_k, f_{mk}\left(x_0\right)\right) / \pi_1\left(f_{mk}\right)\pi_1\left(\mathcal X_m, x_0\right)\to \\
	\to \pi_1\left(\mathcal X_k, f_{mk}\left( x_0\right) \right)/ \pi_1\left(f_{lk}\right)\pi_1\left(\mathcal X_l, f_{ml}\left( x_0\right) \right)	\to \{e\}
	\end{split}
	\end{equation*}
	is exact. Above sequence is equivalent to the following  sequence
	$$
	\{e\}\to	G\left(\mathcal X_m~|~\mathcal X_l\right) \to 	G\left(\mathcal X_m~|~\mathcal X_k\right)\to 	G\left(\mathcal X_l~|~\mathcal X_k\right)\to \{e\}
	$$ 
	which is also exact. Thus $\mathfrak{S} \in \mathfrak{FinTop}$.
\end{example}
\begin{definition}\label{top_cov_trans_defn} Let  $\left\{\mathcal{X} = \mathcal{X}_0 \xleftarrow{}... \xleftarrow{} \mathcal{X}_n \xleftarrow{} ...\right\} \in \mathfrak{FinTop}$, and let
	$\widehat{\mathcal{X}} = \varprojlim \mathcal{X}_n$ be the inverse limit  in the category of topological spaces and continuous maps (cf. \cite{spanier:at}). If $\widehat{\pi}_0: \widehat{\mathcal{X}} \to \mathcal{X}_0$ is the natural continuous map then a homeomorphism $g$ of the space $\widehat{\mathcal{X}}$ is said to be a \textit{covering  transformation} if a following condition holds
	$$
	\widehat{\pi}_0 = \widehat{\pi}_0 \circ g.
	$$
	The group $\widehat{G}$ of such homeomorphisms is said to be the \textit{group of  covering  transformations} of $\mathfrak S$. Denote by $G\left(\widehat{\mathcal{X}}~|~\mathcal X \right)\stackrel{\text{def}}{=}\widehat{G}$. 
\end{definition}
\begin{lemma}\label{top_surj_group_lem}
	Let  $\left\{\mathcal{X} = \mathcal{X}_0 \xleftarrow{}... \xleftarrow{} \mathcal{X}_n \xleftarrow{} ...\right\} \in \mathfrak{FinTop}$, and let
	$\widehat{\mathcal{X}} = \varprojlim \mathcal{X}_n$ be the inverse limit  in the category of topological spaces and continuous maps. There is the natural group isomorphism $G\left(\widehat{\mathcal{X}}~|~\mathcal X \right) \cong \varprojlim G\left({\mathcal{X}}_n~|~\mathcal X \right)$. For any $n \in \mathbb{N}$ there is the natural surjective homomorphism $h_n:G\left(\widehat{\mathcal{X}}~|~\mathcal X \right) \to G\left(\mathcal{X}_n~|~\mathcal X \right)$ and $\bigcap_{n \in \mathbb{N}} \ker h_n$ is a trivial group.
\end{lemma}
\begin{proof}
	For any $n \in \mathbb{N}$ there is the natural continuous map $\widehat{\pi}_n:\widehat{\mathcal{X}} \to \mathcal{X}_n$. Let $x_0 \in \mathcal{X}_0$ and $\widehat{x}_0 \in \widehat{\mathcal{X}}$ be such that $\widehat{\pi}_0\left( \widehat{x}_0\right) = x_0$. Let $\widehat{x}' \in \widehat{\mathcal{X}}$ be such that $\widehat{\pi}_0\left( \widehat{x}'\right)=x_0$.  If  $x'_n = \widehat{\pi}_n\left(\widehat{x}' \right)$ and $x_{n} = \widehat{\pi}_n\left(\widehat{x}_0 \right)$ then $\pi_n\left(x_{n} \right)=\pi_n\left(x'_{n} \right)$, where $\pi_n : \mathcal X_n \to \mathcal X$ is the natural covering. Since $\pi_n$ is regular for any $n \in \N$ there is the unique $g_n \in G\left( \mathcal{X}_n~|~\mathcal{X}\right)$ such that $x'_n = g_n x_{n}$. In result there is a sequence $\left\{g_n \in G\left( \mathcal{X}_n~|~\mathcal{X}\right)\right\}_{n \in \N}$ which satisfies to the following condition
	\begin{equation*}
	g_m \circ	\pi^n_m = \pi^n_m \circ g_n
	\end{equation*}
	where $n > m$ and $\pi^n_m : \mathcal X_n \to \mathcal X_m$ is the natural covering. The sequence $\left\{g_n \right\}$ naturally defines an element $\widehat{g} \in \varprojlim G\left(\mathcal X_n~|\mathcal X \right)$. 
	Let us define an homeomorphism $\varphi_{\widehat{g}}: \widehat{\mathcal{X}} \to \widehat{\mathcal{X}}$ by a following construction. 	If $\widehat{x}''\in \widehat{\mathcal{X}}$ is any point  then there is a sequence $\left\{x''_n \in \mathcal X_n\right\}_{n \in \N}$ such that
	$$
	x''_n = \widehat{\pi}_n\left(\widehat{x}''\right) .
	$$
	On the other hand there is the sequence $\left\{x''^{\widehat{g}}_n \in \mathcal X_n\right\}_{n \in \N}$
	$$
	x''^{\widehat{g}}_n=g_nx''_n
	$$
	which for any $n > m$ satisfies to the following condition
	\begin{equation*}\label{top_xg_eqn}
	\pi^n_m\left(x''^{\widehat{g}}_n \right) = {x}''^{g}_m.
	\end{equation*}
	From the above equation and properties of inverse limits it follows that there is the unique $\widehat{x}''^{\widehat{g}} \in \widehat{\mathcal{X}}$ such that 
	$$
	\widehat{\pi}_n \left( \widehat{x}''^{\widehat{g}}\right) = x''^{\widehat{g}}_n; ~~ \forall n \in \N.
	$$
	The required homeomorphism $\varphi_{\widehat{g}}$ is given by
	$$
	\varphi_{\widehat{g}}\left( \widehat{x}''\right)  = \widehat{x}''^{\widehat{g}}.
	$$
	From $\widehat{\pi}\circ \varphi_{\widehat{g}} = \widehat{\pi}$ it follows that $\varphi_{\widehat{g}}$ corresponds to an element in $G\left(\widehat{  \mathcal X}~|~\mathcal X \right)$ which mapped onto $g_n$ for any $n \in \N$. Otherwise $\varphi_{\widehat{g}}$ naturally corresponds to the element $\widehat{g} \in \varprojlim G\left(\mathcal X_n~|\mathcal X \right)$, so one has  the natural group isomorphism $G\left(\widehat{\mathcal{X}}~|~\mathcal X \right) \cong \varprojlim G\left({\mathcal{X}}_n~|~\mathcal X \right)$.
	From the above construction it turns out that any homeomorphism $\widehat{g} \in G\left(\widehat{  \mathcal X}~|~\mathcal X \right)$ uniquely depends on $\widehat{x}'=\widehat{g}\widehat{x}_0\in \widehat{\pi}_0^{-1} \left( x_0\right)$.  It follows that there is the 1-1 map $\varphi:\widehat{\pi}_0^{-1}\left(x_0 \right)\xrightarrow{\approx} G\left(\widehat{\mathcal{X}}~|~\mathcal X \right)$. Since the covering $\pi_n : \mathcal{X}_n\to\mathcal X$ is regular there is the 1-1 map	$\varphi_n:\pi_n^{-1}\left(x_0 \right)\xrightarrow{\approx} G\left(\mathcal{X}_n~|~\mathcal X \right)$. The natural surjective map
	$$
	\widehat{\pi}_0^{-1}\left(x_0 \right) \to  \pi_n^{-1}\left(x_0 \right)
	$$
	induces the surjective homomorphism $G\left(\widehat{\mathcal{X}}~|~\mathcal X \right) \to G\left(\mathcal{X}_n~|~\mathcal X \right)$. If $\widehat{g} \in \bigcap_{n \in \mathbb{N}} \ker h_n$ is not trivial then $\widehat{g} \widehat{x}_0 \neq  \widehat{x}_0$ and there is $n \in \mathbb{N}$ such that $\widehat{\pi}_n\left(\widehat{x}_0\right)\neq \widehat{\pi}_n\left(\widehat{g}\widehat{x}_0\right)= h_n\left(\widehat{g} \right) \widehat{\pi}_n\left(\widehat{x}_0\right)$, so $h_n\left(\widehat{g} \right)  \in G\left(\mathcal{X}_n~|~\mathcal X \right)$ is not trivial and $\widehat{g} \notin \ker h_n$. From this contradiction it follows that $\bigcap_{n \in \mathbb{N}} \ker h_n$ is a trivial group.
\end{proof}
\begin{definition}\label{top_?oh_defn}
	Let $\mathfrak{S} = \left\{ \mathcal{X}_0 \xleftarrow{}... \xleftarrow{} \mathcal{X}_n \xleftarrow{} ...\right\}$ be 
	a finite covering sequence. The pair $\left(\mathcal{Y},\left\{\pi^{\mathcal Y}_n\right\}_{n \in \mathbb{N}} \right) $ of a (discrete) set $\mathcal{Y}$ with and  
	surjective  maps $\pi^{\mathcal Y}_n:\mathcal{Y} \to \mathcal X_n$ is said to be a \textit{coherent system} if for any $n \in \mathbb{N}^0$ a following diagram  
	\newline
	\begin{tikzpicture}
	\matrix (m) [matrix of math nodes,row sep=3em,column sep=4em,minimum width=2em]
	{
		& \mathcal{Y}  &  \\
		\mathcal{X}_n &  &  \mathcal{X}_{n-1} \\};
	\path[-stealth]
	(m-1-2) edge node [left] {$\pi^{\mathcal Y}_n~$} (m-2-1)
	(m-1-2) edge node [right] {$~\pi^{\mathcal Y}_{n-1}$} (m-2-3)
	(m-2-1) edge node [above] {$\pi_n$}  (m-2-3);
	
	\end{tikzpicture}
	\newline
	is commutative.	
\end{definition}

\begin{definition}\label{comm_top_constr_defn}
	Let $\mathfrak{S} = \left\{ \mathcal{X}_0 \xleftarrow{}... \xleftarrow{} \mathcal{X}_n \xleftarrow{} ...\right\}$ be 
	a topological finite covering sequence. A coherent system $\left(\mathcal{Y},\left\{\pi^{\mathcal Y}_n\right\} \right)$ is said to
	be a \textit{connected covering} of $\mathfrak{S}$ if $\mathcal Y$ is a connected topological space and $\pi^{\mathcal Y}_n$ is a regular covering  for any $n \in \mathbb{N}$.  We will use following notation $\left(\mathcal{Y},\left\{\pi^{\mathcal Y}_n\right\} \right)\downarrow \mathfrak{S}$ or simply $\mathcal{Y} \downarrow \mathfrak{S}$.
\end{definition}

\begin{definition}\label{top_spec_defn}
	Let $\left(\mathcal{Y},\left\{\pi^{\mathcal Y}_n\right\} \right)$ be  a coherent system   of $\mathfrak{S}$ and $y \in \mathcal{Y}$. A subset  $\mathcal V \subset \mathcal{Y}$ is said to be \textit{special} if $\pi^{\mathcal Y}_0\left(\mathcal{V} \right)$ is evenly covered by $\mathcal{X}_1 \to \mathcal{X}_0$ and for any  $n \in \mathbb{N}^0$ following conditions hold:
	\begin{itemize}
		\item $\pi^{\mathcal Y}_n\left(\mathcal{V} \right) \subset \mathcal X_n$ is  an open connected set, 
		\item The restriction $\pi^{\mathcal Y}_n|_{\mathcal V}:\mathcal{V}\to \pi^{\mathcal Y}_n\left( {\mathcal V}\right) $ is a bijective map.
	\end{itemize}
	
\end{definition}
\begin{remark}
	For any $n \in \N^0$ the space $\mathcal X_n$ is second-countable, so from the  Theorem \ref{comm_sep_thm} for any point $x \in \mathcal X_n$ there is an open connected neighborhood $\mathcal U \subset \mathcal X_n$.
\end{remark}
\begin{remark}
	If $\left(\mathcal{Y},\left\{\pi^{\mathcal Y}_n\right\} \right)$ is  a covering  of $\mathfrak{S}$ then the set of special sets is a base of the topology of $\mathcal{Y}$.
\end{remark}

\begin{lemma}\label{top_equ_borel_set_lem}
	Let $\widehat{\mathcal{X}} = \varprojlim \mathcal{X}_n$ be the inverse limit of the sequence $\mathcal{X}_0 \xleftarrow{}... \xleftarrow{} \mathcal{X}_n \xleftarrow{} ...$ in the category of topological spaces and continuous maps. Any special set of $\widehat{\mathcal{X}}$ is a Borel subset of $\widehat{\mathcal{X}}$.
	
\end{lemma}\label{top_spec_borel_lem}
\begin{proof}
	If $\mathcal U_n\subset \mathcal X_n$ is an open set then $\widehat{\pi}_n^{-1} \left(\mathcal U_n \right) \subset \widehat{\mathcal X}$ is open. 
	If $\widehat{\mathcal U}$ is a special set then $\widehat{\mathcal U} = \bigcap_{n \in \mathbb{N}} \widehat{\pi}_n^{-1} \circ \widehat{\pi}_n\left(\widehat{\mathcal U}\right)$, i.e. $\widehat{\mathcal U}$ is a countable intersection of open sets. So $\widehat{\mathcal U}$ is a Borel subset.	
\end{proof}
\begin{definition}\label{comm_top_constr_morph_defn}
	Let us consider the situation of the Definition \ref{comm_top_constr_defn}. A \textit{morphism} from $\left(\mathcal{Y}',\left\{\pi^{\mathcal Y'}_n\right\}\right)\downarrow\mathfrak{S}$ to $\left(\mathcal{Y}'',\left\{\pi^{\mathcal Y''}_n\right\}\right)\downarrow\mathfrak{S}$ is a covering  $f: \mathcal{Y}' \to \mathcal{Y}''$ such that
	$$
	\pi_n^{\mathcal Y''} \circ f= \pi_n^{\mathcal Y'} 
	$$
	for any $n \in \N$.
	
\end{definition}
\begin{empt}\label{comm_top_constr}
	There is a category with objects and morphisms described by Definitions \ref{comm_top_constr_defn}, \ref{comm_top_constr_morph_defn}. Denote by $\downarrow \mathfrak S$ this category.
\end{empt}
\begin{lemma}\label{top_universal_covering_lem}
	There is the final object of the category $\downarrow \mathfrak S$  described in \ref{comm_top_constr}.
\end{lemma}
\begin{proof}
	Let $\widehat{\mathcal{X}} = \varprojlim \mathcal{X}_n$ be the inverse limit of the sequence $\mathcal{X}_0 \xleftarrow{}... \xleftarrow{} \mathcal{X}_n \xleftarrow{} ...$ in the category of topological spaces and continuous maps. Denote by $\overline{\mathcal X}$ a topological space such that
	\begin{itemize}
		\item $\overline{\mathcal X}$ coincides with $\widehat{\mathcal X}$ as a set, 
		\item A set of special sets of $\widehat{\mathcal X}$ is a base of the topology of $\overline{\mathcal X}$.
	\end{itemize}
	If $x_n \in \mathcal X_n$ is a point then there is $\overline{x}\in \overline{\mathcal X}=\widehat{\mathcal X}$ such that $x_n = \widehat{\pi}_n\left( \overline{x}\right)$ and there is a special subset $\widehat{\mathcal U}$ such that $\overline{x} \in \widehat{\mathcal U}$.  From the construction of special subsets it follows that:
	\begin{itemize}
		\item $\mathcal U_n = \widehat{\pi}_n\left( \widehat{\mathcal U}\right)$ is an open neighborhood of $x_n$;
		\item  $$\widehat{\pi}_n^{-1} \left(\mathcal U_n \right) = \bigsqcup_{g \in \ker\left(G\left(\widehat{\mathcal{X}}~|~\mathcal X \right) \to G\left(\mathcal{X}_n~|~\mathcal X \right) \right)  } g \widehat{\mathcal U};$$

		\item For any $g \in \ker\left(G\left(\widehat{\mathcal{X}}~|~\mathcal X \right) \to G\left(\mathcal{X}_n~|~\mathcal X \right) \right)$ the set $g \widehat{\mathcal U}$  mapped homeomorphically onto $\mathcal{U}_n$.
	\end{itemize}
	So the natural map 	$\pi^{\overline{\mathcal X}}_n:\overline{\mathcal X} \to \mathcal X_n$ is a covering. If $\widetilde{\mathcal{X}} \subset \overline{\mathcal{X}}$ is a nontrivial connected component then the map $\widetilde{\mathcal{X}} \to \mathcal X_n$ is a covering, hence  $\widetilde{\mathcal{X}}$ is an object of $\downarrow \mathfrak S$. Let $G \subset \widehat{G}$ be a maximal subgroup such that $G\widetilde{\mathcal{X}}=\widetilde{\mathcal{X}}$. The subgroup $G \subset \widehat{G}$ is normal. If $g \in  \widehat{G} \backslash G$ then $g \widetilde{\mathcal X} \bigcap \widetilde{\mathcal X} = \emptyset$, however  $g$ is a homeomorphism, i.e. $g: \widetilde{\mathcal X} \xrightarrow{\approx} g\widetilde{\mathcal X}$. If $\overline{x} \in \overline{\mathcal{X}}$ then there is $\widetilde{x} \in \widetilde{\mathcal X}$ such that $\overline{\pi}_0\left(\overline{x} \right)= \overline{\pi}_0\left(\widetilde{x} \right)$, hence there is $g \in \widehat{G}$ such that $\overline{x}=g\widetilde{x}$ and $\overline{x} \in g \widetilde{\mathcal X}$. It follows that
	\begin{equation}\label{top_disconnected_repr_eqn}
	\overline{\mathcal X}=  \bigsqcup_{g \in J    } g \widetilde{\mathcal X}                                                      
	\end{equation}
	where $J \subset \widehat{G} $ is a set of representatives of ${\widehat{G}/G}$.
	If $\left(\mathcal{Y},\left\{\pi^{\mathcal Y}_n\right\} \right)$ is  a connected covering  of $\mathfrak{S}$ then there is the natural continuous map $\mathcal{Y}\to \widehat{\mathcal{X}}$, because $\widehat{\mathcal{X}}$ is the inverse limit. Since the continuous map $\overline{\mathcal{X}}\to \widehat{\mathcal{X}}$ is bijective there is the natural map $\overline{\pi}:\mathcal{Y} \to \overline{\mathcal{X}}$. Let $\overline{x} \in \overline{\mathcal{X}}$  be such that  $\overline{x} \in \overline{\pi}\left( \mathcal Y\right)$, i.e.   $\exists y \in \mathcal Y$ which satisfies to a condition $\overline{x} = \overline{\pi}\left( y\right)$. Let $G^{\mathcal Y} \subset G\left(\mathcal Y~|~\mathcal X \right)$ be such that $\overline{\pi}\left( G^{\mathcal Y}y\right)  = \left\{\overline{x}\right\} $. If $\widehat{\mathcal U}$ is a special neighborhood of $\overline{x}$ then there is a connected neighborhood $\mathcal V$ of $y$ which is mapped homeomorphically onto $\widehat{\pi}_0\left(\widehat{\mathcal U}\right) \subset \mathcal X_0$. It follows that
	\begin{equation}\label{top_pi_u_eqn}
	\overline{\pi}^{-1}\left(\widehat{\mathcal U}\right)= \bigsqcup_{g \in G^{\mathcal Y}}g \mathcal V,
	\end{equation}
	i.e. $\widehat{\mathcal U}$ is evenly covered by $\overline{\pi}$. It turns out the map  $\overline{\pi}:\mathcal{Y} \to \overline{\mathcal{X}}$ is continuous. From \eqref{top_disconnected_repr_eqn} it turns out that there is $g \in \widehat{G}$, such that  $\overline{x} \in g \widetilde{\mathcal X}$. The space $\mathcal Y$ is connected  so it is mapped into $g \widetilde{\mathcal X}$, hence there is a continuous map $ \widetilde{\pi} = g^{-1} \circ \overline{\pi}: \mathcal Y \to \widetilde{\mathcal X}$. The set $\widetilde{\pi}\left(  \mathcal Y\right) \subset \widetilde{\mathcal X} $ contains a nontrivial open subset. Denote by $\mathring{  \mathcal X}_{\widetilde{\pi}}\subset \widetilde{\mathcal X}$ (resp. $\overline{  \mathcal X}_{\widetilde{\pi}}\subset \widetilde{\mathcal X}$) maximal open subset of $\widetilde{\pi}\left(  \mathcal Y\right) $ (resp. minimal closed superset of $\widetilde{\pi}\left(  \mathcal Y\right) $). The space $\widetilde{\mathcal X}$ is connected (i.e. open and closed), hence from $\widetilde{\pi}\left(  \mathcal Y\right) \neq \widetilde{\mathcal X}$ it turns out $ \overline{  \mathcal X}_{\widetilde{\pi}} \backslash\mathring{  \mathcal X}_{\widetilde{\pi}} \neq \emptyset$. Let $\widetilde{x} \in \overline{  \mathcal X}_{\widetilde{\pi}} \backslash\mathring{  \mathcal X}_{\widetilde{\pi}}$, and let $\widetilde{\mathcal U}$ be a special neighborhood of $\widetilde{x}$. There is $y \in \mathcal Y$ such that $y \in \widetilde{\pi}^{-1}\left(\widetilde{\mathcal U} \right)$ and there is a special connected neighborhood $\widetilde{\mathcal V} \subset \mathcal Y$ of $y$ such that $\widetilde{\pi}\left(\widetilde{\mathcal V}\right)$ is mapped homemorphically onto  $\pi^{\overline{\mathcal X}}_0\left( \widetilde{\mathcal U}\right) \subset \mathcal X_0$. Otherwise $\widetilde{\mathcal U}\subset \widetilde{\mathcal X} $ is mapped homeomorphically onto $\pi^{\overline{\mathcal X}}_0\left( \widetilde{\mathcal U}\right)$,  it follows that  $\widetilde{\mathcal V}$ is mapped onto $\widetilde{\mathcal U}$. The open set $\widetilde{\mathcal U}$ is such that:
	\begin{itemize}
		\item $\widetilde{\mathcal U}$ is a neighborhood of $\widetilde{x}$,
		\item $\widetilde{\mathcal U} \subset \widetilde{\pi}\left( \mathcal Y\right)$. 
	\end{itemize}
	From the above conditions it follows that $\widetilde{x}$ lies in open subset $\widetilde{\mathcal U} \subset \widetilde{\pi}\left(  \mathcal Y\right) $,   hence $\widetilde{x} \in \mathring{  \mathcal X}_{\widetilde{\pi}}$. This fact contradicts with $\widetilde{x}\notin\overline{  \mathcal X}_{\widetilde{\pi}} \backslash\mathring{  \mathcal X}_{\widetilde{\pi}}$ and from the contradiction it turns out $\widetilde{\pi}\left(  \mathcal Y\right) = \widetilde{\mathcal X} $, i.e. $\widetilde{\pi}$ is surjective. From \eqref{top_pi_u_eqn} it follows that $\widetilde{\pi}: \mathcal Y \to \widetilde{\mathcal X}$ is a covering. Thus $\widetilde{\mathcal X}$ is the final object of the category $\downarrow \mathfrak S$.
\end{proof}
\begin{definition}\label{top_topological_inv_lim_defn}
	The final object $\left(\widetilde{\mathcal{X}},\left\{\pi^{\widetilde{\mathcal X}}_n\right\} \right)$ of the category $\downarrow\mathfrak{S}$ is said to be the \textit{(topological) inverse limit} of $\downarrow\mathfrak{S}$.  The notation $\left(\widetilde{\mathcal{X}},\left\{\pi^{\widetilde{\mathcal X}}_n\right\} \right) = \varprojlim \downarrow \mathfrak{S}$ or simply $~\widetilde{\mathcal{X}} =  \varprojlim \downarrow\mathfrak{S}$ will be used. The space $\overline{\mathcal X}$ from the proof of the Lemma \ref{top_universal_covering_lem} is said to be the \textit{disconnected inverse limit} of $\mathfrak{S}$.
\end{definition}
\begin{lemma}\label{top_biject_lem}
	Suppose
	$\mathfrak{S} = \left\{\mathcal{X} = \mathcal{X}_0 \xleftarrow{}... \xleftarrow{} \mathcal{X}_n \xleftarrow{} ...\right\} \in \mathfrak{FinTop}$, and  $~\widehat{\mathcal X} = \varprojlim  \mathcal X_n$. If $\overline{\mathcal X}$ a topological space which coincides with  $\widehat{\mathcal X}$ as a set and the topology on $\overline{\mathcal X}$ is generated by special sets then there is the natural isomorphism $G\left(\overline{\mathcal X}~|~\mathcal{X} \right) \xrightarrow{\approx} G\left(\widehat{\mathcal{X}}~|~\mathcal{X} \right)$ induced by the map $\overline{\mathcal X} \to \widehat{\mathcal{X}}$.
	
\end{lemma}

\begin{proof}
	
	Since $\overline{\mathcal X}$ coincides with $\widehat{\mathcal X}$ as a set, and  the topology of $\overline{\mathcal X}$ is finer than the topology of $\widehat{\mathcal X}$ there is the natural injective map $G\left(\overline{\mathcal X}~|~\mathcal{X} \right)\hookto G\left(\widehat{\mathcal{X}}~|~\mathcal{X} \right)$. If $\widehat{g}\in G\left(\widehat{\mathcal{X}}~|~\mathcal{X} \right)$ and $\widehat{\mathcal U}$ is a special set, then for any $n \in \mathbb{N}$ following condition holds
	\begin{equation}\label{top_pi_g_u_eqn}
	\widehat{\pi}_n\left(\widehat{g} \widehat{\mathcal U} \right)= h_n\left( \widehat{g}\right)\circ\widehat{\pi}_n\left( \widehat{\mathcal U} \right)  
	\end{equation}
	where $\widehat{\pi}_n: \widehat{\mathcal X} \to \mathcal X_n$ is the natural map, and $h_n : G\left(\widehat{\mathcal{X}}~|~\mathcal{X} \right)\to G\left(\mathcal{X}_n~|~\mathcal{X} \right)$ is given by the Lemma \ref{top_surj_group_lem}. Clearly $h_n\left( \widehat{g}\right) $ is a homeomorphism of $\mathcal X_n$, so from \eqref{top_pi_g_u_eqn} it follows that $\widehat{\pi}_n\left(\widehat{g} \widehat{\mathcal U} \right)$ is an open subset of $\mathcal X_n$. Hence $\widehat{g} \widehat{\mathcal U}$ is special. So $\widehat{g}$ maps special sets onto special sets. Since topology of $\overline{\mathcal X}$ is generated by special sets the map $\widehat{g}$ is a homeomorphism of $\overline{\mathcal X}$, i.e. $\widehat{g} \in  G\left(\overline{\mathcal X}~|~\mathcal{X} \right)$.
	
\end{proof}

\subsubsection{Algebraic construction in brief}\label{comm_alg_constr_susub}
\paragraph*{}
The inverse limit of coverings $\widetilde{\mathcal X}$ is obtained from inverse limit of topological spaces $\widehat{\mathcal X}$ by a change of a topology. The topology of  $\widetilde{\mathcal X}$ is finer then topology of $\widehat{\mathcal X}$, it means that $C_0\left(\widehat{\mathcal X}\right)$ is a subalgebra of $C_b\left(\widetilde{\mathcal X}\right)$. The topology of $\widetilde{\mathcal X}$ is obtained from topology of $\widehat{\mathcal X}$ by addition of special subsets. Addition of new sets to a topology is equivalent to addition of new elements to $C_0\left(\widehat{\mathcal X}\right)$. To obtain $C_b\left(\widetilde{\mathcal X}\right)$ we will add to  $C_0\left(\widehat{\mathcal X} \right) $ special elements (cf. Definition \ref{special_el_defn}). If $\widetilde{\mathcal U}\subset \widetilde{\mathcal X}$ is a special set and $\widetilde{a} \in C_c\left( \widetilde{\mathcal X}\right)$ is positive element such that $\widetilde{a}|_{\widetilde{\mathcal X} \backslash \widetilde{\mathcal U}}= \{0\}$, and $a\in C_c\left(\mathcal X_0\right)$ is given by 
$
a =\sum_{\widehat{g} \in \widehat{G}} \widehat{g} \widetilde{a}
$,
 then following condition holds
$$
a\left(  \widetilde{\pi}_n\left(\widetilde{x}\right) \right)= \left( \sum_{\widehat{g} \in \widehat{G}} \widehat{g} \widetilde{a}\right) \left(  \widetilde{\pi}_n\left(\widetilde{x}\right) \right)= 
\left\{\begin{array}{c l}
\widetilde{a}\left( \widetilde{x} \right)  & \widetilde{x} \in \widetilde{\mathcal U} \\
0 & \widetilde{\pi}_n\left(\widetilde{x}\right) \notin \widetilde{\pi}_n\left(\widetilde{\mathcal U}\right)
\end{array}\right..
$$
From above equation it follows that
\begin{equation}\label{comm_alg_eqn}
\left( \sum_{\widehat{g} \in \widehat{G}} \widehat{g} \widetilde{a}\right)^2  = \sum_{\widehat{g} \in \widehat{G}} \widehat{g} \widetilde{a}^2.
\end{equation}
The equation \eqref{comm_alg_eqn} is purely algebraic and related to special subsets. From the Theorem \ref{comm_main_thm} it follows that the algebraic condition \eqref{comm_alg_eqn}  is sufficient for  construction of  $C_0\left( \widetilde{\mathcal X}\right)$. Thus noncommutative inverse limits of coverings can be constructed by purely algebraic methods.

\subsection{Locally compact spaces}

\paragraph*{} 

There are two equivalent definitions of $C_0\left(\mathcal{X}\right)$ and both of them are used in this article.
\begin{defn}\label{c_c_def_1}
	An algebra $C_0\left(\mathcal{X}\right)$ is the $C^*$-norm closure of the algebra $C_c\left(\mathcal{X}\right)$ of compactly supported continuous functions.
\end{defn}
\begin{defn}\label{c_c_def_2}
	A $C^*$-algebra $C_0\left(\mathcal{X}\right)$ is given by the following equation
	\begin{equation*}
	C_0\left(\mathcal{X}\right) = \left\{\varphi \in C_b\left(\mathcal{X}\right) \ | \ \forall \varepsilon > 0 \ \ \exists K \subset \mathcal{X} \ ( K \text{ is compact}) \ \& \ \forall x \in \mathcal X \backslash K \ \left|\varphi\left(x\right)\right| < \varepsilon  \right\},
	\end{equation*}
	i.e.
	\begin{equation*}
	\left\|\varphi|_{\mathcal X \backslash K}\right\| < \varepsilon.
	\end{equation*}
\end{defn}
\begin{thm}\label{comm_sep_thm}\cite{chun-yen:separability}
	For a locally compact Hausdorff space $\mathcal X$, the following are
	equivalent:
	\begin{enumerate}
		\item[(a)] The Abelian $C^*$-algebra $C_0\left(\mathcal X \right)$  is separable;
		\item[(b)]$\mathcal X$ is $\sigma$-compact and metrizable;
		\item[(c)]$\mathcal X$ is second-countable.
	\end{enumerate}
	
\end{thm}

\begin{cor}\label{com_a_u_cor}
	If $\mathcal X$ a is locally compact second-countable Hausdorff space then for any $x \in  \mathcal X$ and any open neighborhood $\mathcal U\subset\mathcal X$ there is a bounded positive continuous function $a: \mathcal X \to \R$ such that $a\left( x\right) \neq 0$ and $a\left(\mathcal X \backslash \mathcal U \right)= \{0\}$. 
\end{cor}
 \begin{defn}\cite{munkres:topology}
	If $\phi: \mathcal X \to \mathbb{C}$  is continuous then the \textit{support} of $\phi$ is defined to be the closure of the set $\phi^{-1}\left(\mathbb{\C}\backslash \{0\}\right)$ Thus if $x$ lies outside the support, there is some neighborhood of $x$ on which $\phi$ vanishes. Denote by $\supp \phi$ the support of $\phi$.
\end{defn}



\subsection{Hilbert modules}
\paragraph*{} We refer to \cite{blackadar:ko} for definition of Hilbert $C^*$-modules, or simply Hilbert modules. Let $A$ be a $C^*$- algebra, and let $X_A$ be an $A$-Hilbert module. Let $\langle \cdot, \cdot \rangle_{X_A}$ be the $A$-valued product on $X_A$. For any $\xi, \zeta \in X_A$ let us define an $A$-endomorphism $\theta_{\xi, \zeta}$ given by  $\theta_{\xi, \zeta}(\eta)=\xi \langle \zeta, \eta \rangle_{X_A}$ where $\eta \in X_A$. The operator  $\theta_{\xi, \zeta}$ shall be denoted by $\xi \rangle\langle \zeta$. The norm completion of a generated by operators $\theta_{\xi, \zeta}$ algebra is said to be an algebra of compact operators $\mathcal{K}(X_A)$. We suppose that there is a left action of $\mathcal{K}(X_A)$ on $X_A$ which is $A$-linear, i.e. action of  $\mathcal{K}(X_A)$ commutes with action of $A$.

  \subsection{$C^*$-algebras and von Neumann algebras}
 \paragraph*{} In this section I follow to \cite{pedersen:ca_aut}.
   \begin{definition}\label{strict_topology}\cite{pedersen:ca_aut}
  	Let $A$ be a $C^*$-algebra.  The {\it strict topology} on the multiplier algebra $M(A)$ is the topology generated by seminorms $\vertiii{x}_a = \|ax\| + \|xa\|$, ($a\in A$). 
  \end{definition}
  \begin{definition}	\label{strong_topology}\cite{pedersen:ca_aut} Let $\H$ be a Hilbert space. The {\it strong} topology on $B\left(\H\right)$ is the locally convex vector space topology associated with the family of seminorms of the form $x \mapsto \|x\xi\|$, $x \in B(\H)$, $\xi \in \H$.
  \end{definition}
  \begin{definition}\label{weak_topology}\cite{pedersen:ca_aut} Let $\H$ be a Hilbert space. The {\it weak} topology on $B\left(\H\right)$ is the locally convex vector space topology associated with the family of seminorms of the form $x \mapsto \left|\left(x\xi, \eta\right)\right|$, $x \in B(\H)$, $\xi, \eta \in \H$.
  \end{definition}
  
  \begin{theorem}\label{vN_thm}\cite{pedersen:ca_aut}
  	Let $M$ be a $C^*$-subalgebra of $B(\H)$, containing the identity operator. The following conditions are equivalent:
  	\begin{itemize}
  		\item $M=M''$ where $M''$ is the bicommutant of $M$;
  		\item $M$ is weakly closed;
  		\item $M$ is strongly closed.
  	\end{itemize}
  \end{theorem}
  
  \begin{definition}
  	Any $C^*$-algebra $M$ is said to be a {\it von Neumann algebra} or a {\it $W^*$- algebra} if $M$ satisfies to the conditions of the Theorem \ref{vN_thm}.
  \end{definition}
  \begin{definition} \cite{pedersen:ca_aut}
  	Let $A$ be a $C^*$-algebra, and let $S$ be the state space of $A$. For any $s \in S$ there is an associated representation $\pi_s: A \to B\left( \H_s\right)$. The representation $\bigoplus_{s \in S} \pi_s: A \to \bigoplus_{s \in S} B\left(\H_s \right)$ is said to be the \textit{universal representation}. The universal representation can be regarded as $A \to B\left( \bigoplus_{s \in S}\H_s\right)$.  
  \end{definition} 
  \begin{definition}\label{env_alg_defn}\cite{pedersen:ca_aut}
  	Let   $A$ be a $C^*$-algebra, and let $A \to B\left(\H \right)$ be the universal representation $A \to B\left(\H \right)$. The strong closure of $\pi\left( A\right)$ is said to be   the  {\it enveloping von Neumann algebra} or  the {\it enveloping $W^*$-algebra}  of $A$. The enveloping  von Neumann algebra will be denoted by $A''$.
  \end{definition}
  \begin{proposition}\label{env_alg_sec_dual}\cite{pedersen:ca_aut}
  	The enveloping von Neumann algebra $A''$ of a $C^*$-algebra $A$ is isomorphic, as a Banach space, to the second dual of $A$, i.e. $A'' \approx A^{**}$.
  \end{proposition}
  

  \begin{lemma}\label{increasing_convergent_w}\cite{pedersen:ca_aut} Let $\Lambda$ be an increasing net in the partial ordering.  Let $\left\{x_\lambda \right\}_{\la \in \La}$ be an increasing net of self-adjoint operators in $B\left(\H\right)$, i.e. $\la \le \mu$ implies $x_\la \le x_\mu$. If $\left\|x_\la\right\| \le \ga$ for some $\ga \in \mathbb{R}$ and all $\la$ then $\left\{x_\lambda \right\}$ is strongly convergent to a self-adjoint element $x \in B\left(\H\right)$ with $\left\|x_\la\right\| \le \ga$.
  \end{lemma}
  
  \paragraph*{}    For each $x\in B(\H)$ we define the {\it range projection} of $x$ (denoted by $[x]$) as projection on the closure of $x\H$. If $M$ is a von Neumann algebra and $x \in M$ then $[x]\in M$.
  
  \begin{prop}\label{polar_decomposition_prop}\cite{pedersen:ca_aut}
  For each element $x$ in   a von Neumann algebra $M$ there is a unique partial isometry $u\in M$ and positive $\left|x\right| \in M_+$ with $uu^*=[|x|]$ and  $x=|x|u$.
  \end{prop}
  \begin{defn}\label{polar_decomposition_defn}
  	The formula $x=|x|u$ in the Proposition \ref{polar_decomposition_prop} is said to be the \textit{polar decomposition}.
  \end{defn}
  \begin{empt}\label{comm_gns_constr}
  	Any separable $C^*$-algebra $A$ has a state $\tau$ which induces a faithful GNS representation  \cite{murphy}. There is a $\mathbb{C}$-valued product on $A$ given by
  	\begin{equation*}
  	\left(a, b\right)=\tau\left(a^*b\right).
  	\end{equation*}
  	This product induces a product on $A/\mathcal{I}_\tau$ where $\mathcal{I}_\tau =\left\{a \in A \ | \ \tau(a^*a)=0\right\}$. So $A/\mathcal{I}_\tau$ is a pre-Hilbert space. Let denote by $L^2\left(A, \tau\right)$ the Hilbert  completion of $A/\mathcal{I}_\tau$.  The Hilbert space  $L^2\left(A, \tau\right)$ is a space of a  GNS representation of $A$.
  \end{empt}

  \section{Noncommutative finite-fold coverings}
\subsection{Basic construction}

\begin{definition}
	If $A$ is a $C^*$- algebra then an action of a group $G$ is said to be {\it involutive } if $ga^* = \left(ga\right)^*$ for any $a \in A$ and $g\in G$. The action is said to be \textit{non-degenerated} if for any nontrivial $g \in G$ there is $a \in A$ such that $ga\neq a$. 
\end{definition}
\begin{definition}\label{fin_def_uni}
	Let $A \hookto \widetilde{A}$ be an injective *-homomorphism of unital $C^*$-algebras. Suppose that there is a non-degenerated involutive action $G \times \widetilde{A} \to \widetilde{A}$ of a finite group $G$, such that $A = \widetilde{A}^G\stackrel{\text{def}}{=}\left\{a\in \widetilde{A}~|~ a = g a;~ \forall g \in G\right\}$. There is an $A$-valued product on $\widetilde{A}$ given by
	\begin{equation}\label{finite_hilb_mod_prod_eqn}
	\left\langle a, b \right\rangle_{\widetilde{A}}=\sum_{g \in G} g\left( a^* b\right) 
	\end{equation}
	and $\widetilde{A}$ is an $A$-Hilbert module. We say that a triple $\left(A, \widetilde{A}, G \right)$ is an \textit{unital noncommutative finite-fold  covering} if $\widetilde{A}$ is a finitely generated projective $A$-Hilbert module.
\end{definition}
\begin{remark}
	Above definition is motivated by the Theorem \ref{pavlov_troisky_thm}.
\end{remark}
\begin{definition}\label{fin_comp_def}
	Let $A$, $\widetilde{A}$ be $C^*$-algebras and let  $A \hookto \widetilde{A}$ be an inclusion such  that following conditions hold:
	\begin{enumerate}
		\item[(a)] 
		There are unital $C^*$-algebras $B$, $\widetilde{B}$  and inclusions 
		$A \subset B$,  $\widetilde{A}\subset \widetilde{B}$ such that $A$ (resp. $B$) is an essential ideal of $\widetilde{A}$ (resp. $\widetilde{B}$) and $A = B\bigcap \widetilde{A}$,
		\item[(b)] There is an unital  noncommutative finite-fold covering $\left(B ,\widetilde{B}, G \right)$,
		\item[(c)] $G\widetilde{A} = \widetilde{A}$.
	\end{enumerate}
	
	The triple $\left(A, \widetilde{A},G \right)$ is said to be a \textit{noncommutative finite-fold covering with compactification}. The group $G$ is said to be the \textit{covering transformation group} (of $\left(A, \widetilde{A},G \right)$ ) and we use the following notation
	\begin{equation}\label{group_cov_eqn}
	G\left(\widetilde{A}~|~A \right) \stackrel{\mathrm{def}}{=} G.
	\end{equation}
\end{definition}
\begin{remark}
	The Definition \ref{fin_comp_def} is motivated by the Lemma \ref{comm_fin_lem}.
\end{remark}
\begin{remark}
Any unital noncommutative finite-fold covering is a noncommutative finite-fold covering with compactification.
\end{remark}
\begin{definition}\label{fin_def}
	Let $A$, $\widetilde{A}$ be $C^*$-algebras, $A\hookto\widetilde{A}$ an injective *-homomorphism and $G\times \widetilde{A}\to \widetilde{A}$ an involutive non-degenerated action of a finite group $G$  such  that following conditions hold:
	\begin{enumerate}
		\item[(a)] 
		$A \cong \widetilde{A}^G \stackrel{\mathrm{def}}{=} \left\{a\in \widetilde{A}  ~|~ Ga = a \right\}$,
		\item[(b)] 
		There is a family $\left\{\widetilde{I}_\la \subset \widetilde{A} \right\}_{\la \in \La}$ of closed ideals of $\widetilde{A}$ such that 
		\be\label{gi-i}
		G\widetilde{I}_\la = \widetilde{I}_\la.
		\ee
		Moreover $\bigcup_{\la \in \La} \widetilde{I}_\la$ (resp. $\bigcup_{\la \in \La} \left( A \bigcap \widetilde{I}_\la\right) $ ) is a dense subset of $\widetilde{A}$ (resp. $A$), and for any $\la \in \La$ there is a natural  noncommutative finite-fold covering with compactification $\left(\widetilde{I}_\la \bigcap A, \widetilde{I}_\la , G \right)$.  
	\end{enumerate}
We say that the triple  $\left(A, \widetilde{A},G \right)$ is a \textit{noncommutative finite-fold covering}.
\end{definition}

\begin{remark}
	The Definition \ref{fin_def} is motivated by the Theorem \ref{comm_fin_thm}.
\end{remark}
\begin{remark}
	Any noncommutative finite-fold covering with compactification is a  noncommutative finite-fold covering.
\end{remark}
\begin{definition}
	The injective *-homomorphism $A \hookto \widetilde{A}$ 
	from the Definition \ref{fin_def}
	is said to be a \textit{noncommutative finite-fold covering}.
\end{definition}
\begin{definition}\label{hilbert_product_defn}
	Let $\left(A, \widetilde{A}, G\right)$ be a    noncommutative finite-fold covering.  The algebra  $\widetilde{A}$  is a Hilbert $A$-module with an $A$-valued  product given by
	\begin{equation}\label{fin_form_a}
	\left\langle a, b \right\rangle_{\widetilde{A}} = 
	\sum_{g \in G} g(a^*b); ~ a,b \in \widetilde{A}.
	\end{equation}
	We say that this structure of Hilbert $A$-module is {\it induced by the covering} $\left(A, \widetilde{A}, G\right)$. Henceforth we shall consider $\widetilde{A}$ as a right $A$-module, so we will write $\widetilde{A}_A$. 
\end{definition}


\subsection{Induced representation}\label{induced_repr_fin_sec}

\begin{empt}\label{induced_repr_constr}
	Let $\left(A, \widetilde{A}, G\right)$ be a noncommutative finite-fold covering, and let $\rho: A \to B\left(\H\right)$ be a representation. If $X=\widetilde{A}\otimes_A \H$ is the algebraic tensor product then there is a sesquilinear $\C$-valued product $\left(\cdot, \cdot\right)_{X}$ on $X$  given by
	\begin{equation}\label{induced_prod_equ}
	\left(a \otimes \xi, b \otimes \eta \right)_{X}= \left(\xi, 	\left\langle a, b \right\rangle_{\widetilde{A}} \eta\right)_{\H}
	\end{equation}
	where $ \left(\cdot, \cdot\right)_{\H}$ means the Hilbert space product on $\H$, and $\left\langle \cdot, \cdot \right\rangle_{\widetilde{A}}$ is given by \eqref{fin_form_a}. So $X$ is a pre-Hilbert space. There is a natural map $p: \widetilde{A} \times \left( \widetilde{A}\otimes_A \H \right)\to \widetilde{A}\otimes_A \H$ given by
	$$
	(a, b \otimes \xi) \mapsto ab \otimes \xi.
	$$ 
\end{empt}

\begin{defn}\label{induced_repr_defn}
	
	Use notation of the Definition \ref{hilbert_product_defn}, and \ref{induced_repr_constr}.
	If $\widetilde{\H}$ is the Hilbert completion of  $X=\widetilde{A}\otimes_A \H$ then the map $p: \widetilde{A} \times \left( \widetilde{A}\otimes_A \H \right)\to \widetilde{A}\otimes_A \H$ induces the representation $\widetilde{\rho}: \widetilde{A} \to B\left( \widetilde{\H} \right)$. We say that $\widetilde{\rho}$ \textit{is induced by the pair} $\left(\rho,\left(A, \widetilde{A}, G\right)  \right)$.  
\end{defn}
\begin{rem}
	Below any $\widetilde a \otimes \xi\in\widetilde{A}\otimes_A \H$ will be regarded as element in $\widetilde{\H}$.
\end{rem}
\begin{lem}
	If $A \to B\left(\H \right) $ is faithful then $\widetilde{\rho}: \widetilde{A} \to B\left( \widetilde{\H} \right)$ is faithful. 
\end{lem}
\begin{proof}
	If $\widetilde{a} \in \widetilde{A}$ is a nonzero element then
	$$
	a =\left\langle \widetilde{a}~\widetilde{a}^*, \widetilde{a}~\widetilde{a}^*\right\rangle_{\widetilde{A}} = \sum_{g \in G}g\left(\widetilde{a}^*\widetilde{a}~\widetilde{a}~\widetilde{a}^* \right) \in A 
	$$
	is a nonzero positive element. There is $\xi \in \H$ such that $\left( \xi, a\xi\right)_{\H} > 0$.
	However
	$$
	\left( \xi, a\xi\right)_{\H} = \left( \widetilde{a}\widetilde{\xi}, \widetilde{a}\widetilde{\xi}\right)_{\widetilde{\H}}
	$$
	where $\widetilde{\xi} = \widetilde{a}^*\otimes \xi \in \widetilde{A}\otimes_A \H \subset \widetilde{\H}$. Hence $\widetilde{a}\widetilde{\xi} \neq 0$.
\end{proof}
\
\begin{empt}
	Let $\left(A, \widetilde{A}, G\right)$ be a  noncommutative finite-fold covering, let $\rho: A \to B\left(\H \right)$ be a faithful representation, and let  $\widetilde{\rho}: \widetilde{A} \to B\left( \widetilde{\H} \right)$ is induced by the pair $\left(\rho,\left(A, \widetilde{A}, G\right)  \right)$. There is the natural action of $G$ on $\widetilde{\H}$ induced by the map
	$$
	g \left( \widetilde{a} \otimes \xi\right)  = \left( g\widetilde{a} \right) \otimes \xi; ~ \widetilde{a} \in \widetilde{A}, ~ g \in G, ~ \xi \in \H. 
	$$
	There is the natural orthogonal inclusion $\H \subset \widetilde{\H}$ induced by inclusions
	$$
	A \subset\widetilde{A}; ~~ A \otimes_A \H \subset\widetilde{A} \otimes_A \H.
	$$
	Action of $g$ on $\widetilde{A}$ can be defined by representation as $g \widetilde{a} = g \widetilde{a} g^{-1}$, i.e.
	$$
	(g\widetilde{a}) \xi = g\left(\widetilde{a} \left( g^{-1}\xi \right)  \right);~ \forall \xi \in \widetilde{\H}.
	$$
	
\end{empt}

\begin{defn}\label{mult_G_act_defn} If $M\left(\widetilde{A} \right)$ is the multiplier algebra of $\widetilde{A}$ then there is the natural action of $G$ on $M\left(\widetilde{A} \right)$ such that for any $\widetilde{a}\in M\left(\widetilde{A} \right)$, $\widetilde{b}\in\widetilde{A}$ and $g \in G$ a following condition holds
	$$
	\left(g \widetilde{a} \right)\widetilde{b} = g\left(\widetilde{a} \left( g^{-1}\widetilde{b} \right)  \right)
	$$
	We say that action of $G$ on $M\left(\widetilde{A} \right)$ is \textit{induced} by the action  of $G$ on $\widetilde{A}$.
\end{defn}
\begin{lem}\label{ind_mult_inv_lem}
	If an	action of $G$ on $M\left(\widetilde{A} \right)$ is induced by the action  of $G$ on $\widetilde{A}$ then
	\begin{equation}\label{mag_ma_eqn}
	M\left(\widetilde{A} \right)^G \subset M\left(\widetilde{A}^G \right) .
	\end{equation}
\end{lem}
\begin{proof}
	If $a \in M\left(\widetilde{A} \right)^G$ and $b \in \widetilde{A}^G$ then $ab \in \widetilde{A}$ is such that $g\left( ab\right) =\left( ga\right)\left(gb \right) = ab  \in \widetilde{A}^G$.
\end{proof}

\section{Noncommutative infinite coverings}
\subsection{Basic construction}\label{bas_constr}

This section contains a noncommutative generalization of infinite coverings.
\begin{definition}\label{comp_defn}
	Let
	\begin{equation*}
	\mathfrak{S} =\left\{ A =A_0 \xrightarrow{\pi_1} A_1 \xrightarrow{\pi_2} ... \xrightarrow{\pi_n} A_n \xrightarrow{\pi^{n+1}} ...\right\}
	\end{equation*}
	be a sequence of $C^*$-algebras and noncommutative finite-fold coverings such that:
	\begin{enumerate}
		\item[(a)] Any composition $\pi_{n_1}\circ ...\circ\pi_{n_0+1}\circ\pi_{n_0}:A_{n_0}\to A_{n_1}$ corresponds to the noncommutative covering $\left(A_{n_0}, A_{n_1}, G\left(A_{n_1}~|~A_{n_0}\right)\right)$;
		\item[(b)] If $k < l < m$ then $G\left( A_m~|~A_k\right)A_l = A_l$ (Action of $G\left( A_m~|~A_k\right)$ on $A_l$ means that $G\left( A_m~|~A_k\right)$ acts on $A_m$, so $G\left( A_m~|~A_k\right)$ acts on $A_l$ since $A_l$ a subalgebra of $A_m$);
		\item[(c)] If $k < l < m$ are nonegative integers then there is the natural exact sequence of covering transformation groups
		\begin{equation*}
		\{e\}\to G\left(A_{m}~|~A_{l}\right) \xrightarrow{\iota} G\left(A_{m}~|~A_{k}\right)\xrightarrow{\pi}G\left(A_{l}~|~A_{k}\right)\to\{e\}
		\end{equation*}
		where the existence of the homomorphism $G\left(A_{m}~|~A_{k}\right)\xrightarrow{\pi}G\left(A_{l}~|~A_{k}\right)$ follows from (b).
		
	\end{enumerate}
	The sequence
	$\mathfrak{S}$
	is said to be an \textit{(algebraical)  finite covering sequence}. 
	For any finite covering sequence we will use the notation $\mathfrak{S} \in \mathfrak{FinAlg}$.
\end{definition}
\begin{definition}\label{equiv_act_defn}
	Let $\widehat{A} = \varinjlim A_n$  be the $C^*$-inductive limit \cite{murphy}, and suppose that $\widehat{G}= \varprojlim G\left(A_n~|~A \right) $ is the projective limit of groups \cite{spanier:at}. There is the natural action of $\widehat{G}$ on $\widehat{A}$. A non-degenerate faithful representation $\widehat{A} \to B\left( \H\right) $ is said to be \textit{equivariant} if there is an action of $\widehat{G}$ on $\H$ such that for any $\xi \in \H$ and $g \in  \widehat{G}$ the following condition holds
	\begin{equation}\label{equiv_act_eqn}
	\left(ga \right) \xi = g\left(a\left(g^{-1}\xi \right)  \right) .
	\end{equation}
\end{definition}
\begin{example}
	Let $S$ be the state space of $\widehat{A}$, and let $\widehat{A} \to B\left(\bigoplus_{s \in S} \H_s \right)$ be the universal representation. There is the natural action of $\widehat{G}$ on $S$ given by
	$$
	\left(gs \right)\left(  a\right)  = s\left( ga\right); ~ s \in S,~ a \in \widehat{A},~ g \in \widehat{G}.
	$$
	The action of $\widehat{G}$ on $S$ induces the action of $\widehat{G}$ on $\bigoplus_{s \in S} \H_s$. It follows that the universal representation is equivariant.
\end{example}
\begin{example}\label{equiv_exm}
	Let $s$ be a faithful state which corresponds to the representation $\widehat{A} \to B\left(\H_s \right)$  and $\left\{g_n\in \widehat{G}\right\}_{n \in \N}= \widehat{G}$  is a bijection. The state
	$$
	\sum_{n \in \N}\frac{g_ns }{2^{n}}
	$$
	corresponds to an equvariant representation $\widehat{A} \to B\left(\bigoplus_{g \in \widehat{G}}\H_{gs} \right)$.
\end{example}

\begin{definition}\label{special_el_defn}
	Let $\pi:\widehat{A} \to B\left( \H\right) $ be an equivariant representation.  A positive element  $\overline{a}  \in B\left(\H \right)_+ $ is said to be \textit{special} (with respect to $\pi$) if following conditions hold:
	\begin{enumerate}
		\item[(a)] For any $n \in \mathbb{N}^0$  the following  series 
		\begin{equation*}
		\begin{split}
		a_n = \sum_{g \in \ker\left( \widehat{G} \to  G\left( A_n~|~A \right)\right)} g  \overline{a} 
		\end{split}
		\end{equation*}
		is strongly convergent and the sum lies in   $A_n$, i.e. $a_n \in A_n $;		
		\item[(b)]
		If $f_\eps: \R \to \R$ is given by 
		\begin{equation}\label{f_eps_eqn}
		f_\eps\left( x\right)  =\left\{
		\begin{array}{c l}
		0 &x \le \eps \\
		x - \eps & x > \eps
		\end{array}\right.
		\end{equation}
		then for any $n \in \mathbb{N}^0$ and for any $z \in A$   following  series 
		\begin{equation*}
		\begin{split}
		b_n = \sum_{g \in \ker\left( \widehat{G} \to  G\left( A_n~|~A \right)\right)} g \left(z  \overline{a} z^*\right) ,\\
		c_n = \sum_{g \in \ker\left( \widehat{G} \to  G\left( A_n~|~A \right)\right)} g \left(z  \overline{a} z^*\right)^2,\\
		d_n = \sum_{g \in \ker\left( \widehat{G} \to  G\left( A_n~|~A \right)\right)} g f_\eps\left( z  \overline{a} z^* \right) 
		\end{split}
		\end{equation*}
		are strongly convergent and the sums lie in   $A_n$, i.e. $b_n,~ c_n,~ d_n \in A_n $; 
		\item[(c)] For any $\eps > 0$ there is $N \in \N$ (which depends on $\overline{a}$ and $z$) such that for any $n \ge N$ a following condition holds
		\begin{equation}\label{square_condition_equ}
		\begin{split}
		\left\| b_n^2 - c_n\right\| < \eps.
		\end{split}
		\end{equation}	
	\end{enumerate}
	
	An element  $\overline{   a}' \in B\left( \H\right) $ is said to be \textit{weakly special} if 
	$$
	\overline{   a}' = x\overline{a}y; \text{ where }   x,y \in \widehat{A}, \text{ and } \overline{a} \in B\left(\H \right)  \text{ is special}.
	$$
	
\end{definition}
\begin{lemma}\label{stong_conv_inf_lem}
	If $\overline{a} \in B\left( \H\right)_+$ is a special element and $\overline{G}_n=\ker\left( \widehat{G} \to  G\left( A_n~|~A \right)\right)$ then from
	\begin{equation*}
	\begin{split}
	a_n = \sum_{g \in \overline{G}_n} g \overline{a},
	\end{split}
	\end{equation*}
	it follows that $\overline{a} = \lim_{n \to \infty} a_n$ in the sense of the strong convergence. Moreover  one has $\overline{a} =\inf_{n \in \N}a_n$.
\end{lemma}
\begin{proof}
	From the Lemma  \ref{increasing_convergent_w} it follows that the decreasing lower-bounded sequence $\left\{a_n\right\}$ is strongly convergent and $\lim_{n \to \infty} a_n=\inf_{n \in \N}a_n$. From $a_n > \overline{a}$ it follows that $\inf_{n \in \N}a_n \ge \overline{a}$. If $\inf_{n \in \N}a_n > \overline{a}$ then there is  $\xi \in \H$ such that 
	$$
	\left(\xi,\left( \inf_{n \in \N} a_n\right) \xi \right) > \left(\xi,\overline{a}\xi \right),
	$$ 
however one has
	\begin{equation*}
	\begin{split}
	\left(\xi,\left( \inf_{n \in \N} a_n\right) \xi \right)=	\inf_{n \in \N}\left(\xi, a_n\xi \right) = \inf_{n \in \N}\left(\xi,\left( \sum_{g \in \overline{G}_n} g \overline{a}\right) \xi \right)= \inf_{n \in \N} \sum_{g \in \overline{G}_n}\left(\xi, g\overline{a}\xi \right)= \left(\xi, \overline{a}\xi \right).
	\end{split}
	\end{equation*}
	It follows that $\overline{a}=\inf_{n \in \N}a_n$.
	
\end{proof}

\begin{corollary}\label{special_cor}
	Any weakly special element lies in the enveloping von Neumann algebra $\widehat{A}''$ of $\widehat{A}=\varinjlim A_n$. If $\overline{A}_\pi \subset B\left( \H\right)$ is the $C^*$-norm completion of an algebra generated by weakly special elements then $\overline{A}_\pi \subset \widehat{A}''$.
\end{corollary}
\begin{lemma}
	If $\overline{a}\in B\left( \H\right)$ is special, (resp.  $\overline{a}'\in B\left( \H\right)$ weakly special) then for any  $g \in \widehat{G}$ the element  $g\overline{a}$ is special, (resp. $g\overline{a}'$ is weakly special).
\end{lemma}
\begin{proof}
	If	$\overline{a}  \in B\left( \H\right)$ is special then $g \overline{a}$ satisfies to (a)-(c) of the Definition \ref{special_el_defn}, i.e. $g\overline{   a}$ is special. If $\overline{   a}'$ is weakly special then form
	$$
	\overline{   a}' = x\overline{a}y; \text{ where }   x,y \in \widehat{A}, \text{ and } \overline{a} \in B\left(\H \right)  \text{ is special},
	$$
	it turns out that
	$$
	g\overline{   a}' = \left( gx\right) \left( g\overline{a}\right) \left( gy\right),
	$$
	i.e. $g\overline{   a}'$ is weakly special.
\end{proof}
\begin{corollary}\label{disconnect_group_action_cor}
	If $\overline{A}_\pi \subset B\left( \H\right)$ is the $C^*$-norm completion of algebra generated by weakly special elements, then there is a natural action of $\widehat{G}$ on $\overline{A}_\pi$.
\end{corollary}

\begin{definition}\label{main_defn_full}
	Let $\mathfrak{S} =\left\{ A =A_0 \xrightarrow{\pi^1} A_1 \xrightarrow{\pi^2} ... \xrightarrow{\pi^n} A_n \xrightarrow{\pi^{n+1}} ...\right\}$ be an  algebraical  finite covering sequence. Let  $\pi:\widehat{A} \to B\left( \H\right) $ be an equivariant representation.	 Let $\overline{A}_\pi \subset B\left( \H\right)$ be the $C^*$-norm completion of algebra generated by weakly special elements. We say that $\overline{A}_\pi$ is the {\it disconnected inverse noncommutative limit} of $\downarrow\mathfrak{S}$ (\textit{with respect to $\pi$}). 
	The triple  $\left(A, \overline{A}_\pi, G\left(\overline{A}_\pi~|~ A\right)\stackrel{\mathrm{def}}{=} \widehat{G}\right)$ is said to be the  {\it disconnected infinite noncommutative covering } of $\mathfrak{S}$ (\textit{with respect to $\pi$}). If $\pi$ is the universal representation then "with respect to $\pi$" is dropped and we will write
	$\left(A, \overline{A}, G\left(\overline{A}~|~ A\right)\right)$.	
\end{definition}
\begin{definition}\label{main_sdefn}
	Any  maximal irreducible subalgebra $\widetilde{A}_\pi \subset \overline{A}_\pi$  is said to be a {\it connected component} of $\mathfrak{S}$ ({\it with respect to $\pi$}). The maximal subgroup $G_\pi\subset G\left(\overline{A}_\pi~|~ A\right)$ among subgroups $G\subset G\left(\overline{A}_\pi~|~ A\right)$ such that $G\widetilde{A}_\pi=\widetilde{A}_\pi$  is said to be the $\widetilde{A}_\pi$-{\it  invariant group} of $\mathfrak{S}$. If $\pi$ is the universal representation then "with respect to $\pi$" is dropped. 
\end{definition}

\begin{remark}
	From the Definition \ref{main_sdefn} it follows that $G_\pi \subset G\left(\overline{A}_\pi~|~ A\right)$ is a normal subgroup.
\end{remark}
\begin{definition}\label{good_seq_defn} Let $$\mathfrak{S} = \left\{ A =A_0 \xrightarrow{\pi^1} A_1 \xrightarrow{\pi^2} ... \xrightarrow{\pi^n} A_n \xrightarrow{\pi^{n+1}} ...\right\} \in \mathfrak{FinAlg},$$ and let $\left(A, \overline{A}_\pi, G\left(\overline{A}_\pi~|~ A\right)\right)$ be a disconnected infinite noncommutative covering of $\mathfrak{S}$ with respect to an equivariant representation $\pi: \varinjlim A_n\to B\left(\H \right) $. Let $\widetilde{A}_\pi\subset \overline{A}_\pi$  be a connected component of $\mathfrak{S}$ with respect to $\pi$, and let $G_\pi \subset  G\left(\overline{A}_\pi~|~ A\right)$  be the $\widetilde{A}_\pi$ -  invariant group of $\mathfrak{S}$.
	Let  $h_n : G\left(\overline{A}_\pi~|~ A\right) \to  G\left( A_n~|~A \right)$ be the natural surjective  homomorphism. The  representation $\pi: \varinjlim A_n\to B\left(\H \right)$ is said to be \textit{good} if it satisfies to following conditions:
	\begin{enumerate}
		\item[(a)] The natural *-homomorphism $ \varinjlim A_n \to  M\left(\widetilde{A}_\pi \right)$ is injective,
		\item[(b)] If $J\subset G\left(\overline{A}_\pi~|~ A\right)$ is a set of representatives of $G\left(\overline{A}_\pi~|~ A\right)/G_\pi$, then the algebraic direct sum
		\begin{equation*}
		\bigoplus_{g\in J} g\widetilde{A}_\pi
		\end{equation*}
		is a dense subalgebra of $\overline{A}_\pi$,
		\item [(c)] For any $n \in \N$ the restriction $h_n|_{G_\pi}$ is an epimorphism, i. e. $h_n\left(G_\pi \right) = G\left( A_n~|~A \right)$.
	\end{enumerate}
	If $\pi$ is the universal representation we say that $\mathfrak{S}$ is \textit{good}.
\end{definition}

\begin{definition}\label{main_defn}
	Let $\mathfrak{S}=\left\{A=A_0 \to A_1 \to ... \to A_n \to ...\right\} \in \mathfrak{FinAlg}$ be  an algebraical  finite covering sequence. Let $\pi: \widehat{A} \to B\left(\H \right)$ be a good representation.    A connected component $\widetilde{A}_\pi \subset \overline{A}_\pi$  is said to be the {\it inverse noncommutative limit of $\downarrow\mathfrak{S}$ (with respect to $\pi$)}. The $\widetilde{A}_\pi$-invariant group $G_\pi$  is said to be the {\it  covering transformation group of $\mathfrak{S}$} ({\it with respect to $\pi$}).  The triple $\left(A, \widetilde{A}_\pi, G_\pi\right)$ is said to be the  {\it infinite noncommutative covering} of $\mathfrak{S}$  ({\it with respect to $\pi$}). 
	We will use the following notation 
	\begin{equation*}
	\begin{split}
	\varprojlim_\pi \downarrow \mathfrak{S}\stackrel{\mathrm{def}}{=}\widetilde{A}_\pi,\\
	G\left(\widetilde{A}_\pi~|~ A\right)\stackrel{\mathrm{def}}{=}G_\pi.
	\end{split}
	\end{equation*}	
	If $\pi$ is the universal representation then "with respect to $\pi$" is dropped and we will write $\left(A, \widetilde{A}, G\right)$, $~\varprojlim \downarrow \mathfrak{S}\stackrel{\mathrm{def}}{=}\widetilde{A}$ and  $ G\left(\widetilde{A}~|~ A\right)\stackrel{\mathrm{def}}{=}G$.
\end{definition}

\begin{definition}\label{inf_hilb_prod_defn}
	Let $\mathfrak{S}=\left\{A=A_0 \to A_1 \to ... \to A_n \to ...\right\} \in \mathfrak{FinAlg}$ be  an algebraical  finite covering sequence. Let $\pi: \widehat{A} \to B\left(\H \right)$ be a good representation.   Let $\left(A, \widetilde{A}_\pi, G_\pi\right)$  be the  infinite noncommutative covering of $\mathfrak{S}$  ( with respect to $\pi$). Let $K\left( \widetilde{A}_\pi\right)$ be the Pedersen ideal of  $\widetilde{A}_\pi$. We say that $\mathfrak{S}$ \textit{allows inner product (with respect to $\pi$)} if following conditions hold:
	\begin{enumerate}
		\item[(a)]  Any $\widetilde{a} \in K\left( \widetilde{A}_\pi\right)$ is weakly special,
		\item[(b)]
		For any $n \in \N$, and $\widetilde{a}, \widetilde{b} \in K\left( \widetilde{A}_\pi\right)$ the series
		\begin{equation*}
		\begin{split}
		a_n = \sum_{g \in \ker\left( \widehat{G} \to  G\left( A_n~|~A \right)\right)} g \left(\widetilde{a}^* \widetilde{b}  \right) 
		\end{split}
		\end{equation*}
		is strongly convergent and $a_n \in A_n$.
	\end{enumerate}
	
\end{definition}

\begin{remark}\label{inf_hilb_prod_rem}
	If $\mathfrak{S}$ allows  inner product (with respect to $\pi$) then $K\left( \widetilde{A}_\pi\right)$ is a pre-Hilbert $A$ module such that the inner product is given by
	\begin{equation*}
	\begin{split}
	\left\langle \widetilde{a}, \widetilde{b}  \right\rangle  = \sum_{g \in  \widehat{G}} g \left(\widetilde{a}^* \widetilde{b}  \right) \in A 
	\end{split}
	\end{equation*}
	where the above series is strongly convergent. The completion of  $K\left( \widetilde{A}_\pi\right)$ with respect to a norm
	\begin{equation*}
	\begin{split}
	\left\| \widetilde{a}\right\| = \sqrt{\left\| \left\langle \widetilde{a}, \widetilde{a}  \right\rangle\right\|}
	\end{split}
	\end{equation*}
	is an $A$-Hilbert module. Denote by $X_A$ this completion. The ideal $K\left( \widetilde{A}_\pi\right)$ is a left $\widetilde{A}_\pi$-module, so $X_A$ is also $\widetilde{A}_\pi$-module. Sometimes we will write  $_{\widetilde{A}_\pi}X_A$ instead  $X_A$.
	
\end{remark}
\begin{definition}\label{inf_hilb_mod_defn}
	Let $\mathfrak{S}=\left\{A=A_0 \to A_1 \to ... \to A_n \to ...\right\} \in \mathfrak{FinAlg}$ and  $\mathfrak{S}$ allows inner product (with respect to $\pi$) then  then we say that given by the Remark  \ref{inf_hilb_prod_rem} $A$-Hilbert module $_{\widetilde{A}_\pi}X_A$ \textit{corresponds to the pair} $\left(\mathfrak{S}, \pi \right) $. If $\pi$ is the universal representation then we say that $_{\widetilde{A}}X_A$ \textit{corresponds to } $\mathfrak{S}$.
\end{definition}
\subsection{Induced representation}\label{inf_ind_repr_subsection}

\paragraph*{} 	Let $\pi:\widehat{A} \to B\left(\overline{\H}_\pi \right)$ be a good representation. Let  $\left(A, \widetilde{A}_\pi, G_\pi\right)$ be an infinite noncommutative covering  with respect to $\pi$ of $\mathfrak{S}$. Denote by $\overline{W}_\pi \subset B\left(\overline{\H}_\pi \right)$ the $\widehat{A}$-bimodule of weakly special elements, and denote by
\begin{equation}\label{wealky_spec_eqn}
\widetilde{W}_\pi = \overline{W}_\pi \bigcap  \widetilde{A}_\pi.
\end{equation}
If $\pi$ is the universal representation then we write $\widetilde{W}$ instead $\widetilde{W}_\pi$.

\begin{lem}\label{w_conv_lem}
	If $\widetilde{a}, \widetilde{b} \in \widetilde{W}_\pi $ are weakly special elements  then a series
	$$
	\sum_{g \in G_\pi} g\left(\widetilde{a}^*\widetilde{b} \right)
	$$
	is strongly convergent.
\end{lem}
\begin{proof}

	From the definition of weakly special element one has $$\widetilde{a}^* = x \widetilde{c} y$$ where $\widetilde{c}$ is a (positive) special element and $x,y \in \widehat{A}$.	A series
	$$
	\sum_{g \in G_\pi} g \widetilde{c}
	$$
	is strongly convergent. For any $\xi\in\overline{\H}_\pi$ and $\varepsilon > 0$ there is a finite subset $G'\subset G_\pi$ such that for any finite $G''$ which satisfies to $G' \subset G'' \subset G_\pi$ following condition holds
	$$
	\left\| \sum_{g \in G''\backslash G'} \left( g\widetilde{b}\right)  \xi\right\| < \frac{\eps}{\left\|x\right\|\left\|\sum_{g\in G_\pi}g\widetilde{c}\right\|\left\|y\right\| }.
	$$
	Hence one has
	$$
	\left\| \sum_{g \in G''\backslash G'} \left( g\left( \widetilde{a}^*\widetilde{b}\right)\right)   \xi\right\| < \eps,
	$$
	i.e. the series 
	$$
	\sum_{g \in G_\pi} g\left(\widetilde{a}^*\widetilde{b} \right)
	$$
	is strongly convergent and $\sum_{g \in G_\pi} g\left(\widetilde{a}^*\widetilde{b} \right) \in \widehat{A}''$. 
\end{proof}

\begin{defn}\label{ss_defn}
	Element $\widetilde{a} \in \widetilde{A}_\pi$ is said to be \textit{square-summable} if the series 
	\be\label{ss_eqn}
	\sum_{g \in G_\pi} g\left(\widetilde{a}^*\widetilde{a} \right)
	\ee
	is strongly convergent to a bounded operator. Denote by $L^2\left(\widetilde{A}_\pi \right)$ (or 
	$L^2\left(\widetilde{A}\right)$ the $\C$-space of square-summable operators.
\end{defn}
\begin{rem}
	If $\widetilde{b} \in \widehat{A}$, and $\widetilde{a}\in L^2\left(\widetilde{A}\right)$ then
	$$
	\left\|	\sum_{g \in G_\pi}g\left(\widetilde{b}\widetilde{a} \right)^* \left(\widetilde{b}\widetilde{a}\right) \right\| \le \left\|\widetilde{b} \right\|^2\left\| 	\sum_{g \in G_\pi} g\left(\widetilde{a}^*\widetilde{a} \right)\right\| ,~~
	\left\|	\sum_{g \in G_\pi}g\left(\widetilde{a}\widetilde{b} \right)^* \left(\widetilde{a}\widetilde{b}\right) \right\| \le \left\|\widetilde{b} \right\|^2\left\| 	\sum_{g \in G_\pi} g\left(\widetilde{a}^*\widetilde{a} \right)\right\| 
	$$
	it turns out
	\be\label{act_on_l2_eqn}
	\widehat{A}L^2\left(\widetilde{A}_\pi\right) \subset L^2\left(\widetilde{A}_\pi\right),~~
	L^2\left(\widetilde{A}_\pi\right)\widehat{A} \subset L^2\left(\widetilde{A}_\pi\right),
	\ee
	i.e. there is the left and right action of $\widehat{A}$ on $L^2\left(\widetilde{A}\right)$.
\end{rem}
\begin{rem}
	If $a, b \in L^2\left(\widetilde{A}_\pi \right)$ then sum	$\sum_{g \in G_\pi} g\left(\widetilde{a}^*\widetilde{b} \right) \in \widehat{A}''$ is bounded and $G_\pi$-invariant, hence
	$\sum_{g \in G_\pi} g\left(\widetilde{a}^*\widetilde{b} \right) \in A'' $.
\end{rem}
\begin{rem}
	From the Lemma \ref{w_conv_lem} it turns out $\widetilde{W}_\pi\subset L^2\left(\widetilde{A}_\pi\right)$
\end{rem}
\begin{empt}\label{inf_repr_constr}
	Let $A \to B\left(\H \right)$ be a  representation.  Denote by $\widetilde{\H}$ a Hilbert completion of a pre-Hilbert space
	\begin{equation}\label{inf_ind_prod_eqn}
	\begin{split}
	L^2\left( \widetilde{A}_\pi \right) \otimes_A \H,\\
	\text{with a scalar product} 	\left(\widetilde{a} \otimes \xi, \widetilde{b} \otimes \eta \right)_{\widetilde{\H}} = \left( \xi, \left( \sum_{g \in G_\pi } g \left( \widetilde{a}^*\widetilde{b}\right) \right) \eta \right)_{\H}.
	\end{split}
	\end{equation}  
	There is the left action of $\widehat{A}$ on $L^2\left(\widetilde{A}_\pi\right) \otimes_{A} \H$ given by
	$$
	\widetilde{b}\left(\widetilde{a} \otimes \xi \right) = \widetilde{b}\widetilde{a} \otimes \xi
	$$
	where $\widetilde{a} \in 	L^2\left( \widetilde{A}_\pi \right) $, $\widetilde{b} \in \widehat{A}$, $\xi \in \H$.	
	The left action of $\widehat{A}$ on $L^2\left( \widetilde{A}_\pi\right)  \otimes_A \H$  induces following  representations
	\begin{equation*}
	\begin{split}
	\widehat{\rho}:\widehat{A} \to B\left( \widetilde{\H}\right),\\	
	\widetilde{\rho}:\widetilde{A}_\pi \to B\left( \widetilde{\H}\right).	
	\end{split}
	\end{equation*}

\end{empt}
\begin{defn}\label{inf_ind_defn}
	The constructed in \ref{inf_repr_constr} representation  $\widetilde{\rho}:\widetilde{A}_\pi \to  B\left( \widetilde{\H}\right)$ is said to be \textit{induced} by  $\left( \rho, \mathfrak{S}, \pi\right)  $. We also say that  $\widetilde{\rho}$ is  \textit{induced} by $\left( \rho, \left( A, \widetilde{A}_\pi, G\left(\widetilde{A}_\pi~|~A \right) \right), \pi\right) $. If $\pi$ is an universal representation we say that  $\widetilde{\rho}$ is  \textit{induced} by  $\left( \rho, \mathfrak{S}\right)$ and/or $\left( \rho, \left( A, \widetilde{A}, G\left(\widetilde{A}~|~A \right) \right)\right) $.  
\end{defn}
\begin{rem}
	If $\rho$ is faithful, then  ${\rho}$ is faithful.
\end{rem}
\begin{rem}\label{a_act_hilb_rem} 
	There is an action of $G_\pi$ on $\widetilde{\H}$ induced by the natural action of $G_\pi$ on the $\widetilde{A}_\pi$-bimodule $L^2\left( \widetilde{A}_\pi\right) $. If the representation $\widetilde A_\pi \to 	B\left( \widetilde{\H} \right)$ is faithful then an action of $ G_\pi$ on $\widetilde A_\pi$ is given by
	$$ 
	\left( g  \widetilde a\right) \xi =   g \left(  \widetilde a  \left(  g^{-1}\widetilde\xi\right) \right); ~ \forall  g  \in {G}, ~ \forall\widetilde a  \in \widetilde{A}_\pi, ~\forall\widetilde \xi \in \widetilde{\H}.
	$$
\end{rem}
\begin{empt}
	If $\mathfrak{S}$ allows inner product with respect to $\pi$ then for any representation $A \to B\left( \H\right)$ an algebraic tensor product $_{\widetilde{A}_\pi}X_A \otimes_A \H$ is a pre-Hilbert space with the product given by
	\begin{equation*}
	\left(a \otimes \xi, b \otimes \eta \right) = \left(\xi, \left\langle a, b \right\rangle\eta \right) 
	\end{equation*} 
	(cf. Definitions \ref{inf_hilb_mod_defn} and \ref{inf_hilb_prod_defn})
\end{empt}

\begin{lem}
	Suppose $\mathfrak{S}$ allows inner product with respect to $\pi$ and any $\widetilde{a} \in K\left( \widetilde{A}_\pi\right)$ is weakly special. If $\widetilde{\H}$ (resp. $\widetilde{\H}'$) is a Hilbert norm completion of 	$W_\pi \otimes_{A} \H$ (resp. $_{\widetilde{A}_\pi}X_A \otimes_A \H$) then there is the natural isomorphism $\widetilde{\H} \cong \widetilde{\H}'$.
\end{lem}
\begin{proof}
	From $K\left( \widetilde{A}_\pi\right) \subset W_\pi$ and taking into account that $K\left( \widetilde{A}_\pi\right)$ is dense in $_{\widetilde{A}_\pi}X_A$ it turns out  $\widetilde{\H}' \subset \widetilde{\H}$. If $\widetilde{a} \in  W_\pi$ is a positive element and $f_\eps$ is given by \eqref{f_eps_eqn} then
	\begin{enumerate}
		\item[(a)] $f_\eps\left(\widetilde{a} \right) \in K\left( \widetilde{A}_\pi\right)$,
		\item[(b)] $\lim_{\eps \to 0} f_\eps\left(\widetilde{a} \right)=\widetilde{a}$.
	\end{enumerate}
	
	From (a) it follows that $f_\eps\left(\widetilde{a} \right) \otimes \xi \in _{\widetilde{A}_\pi}X_A \otimes_A \H $ for any $\xi \in \H$. From (b) it turns out $\widetilde{a} \otimes \xi \in \widetilde{\H}'$. From this fact it follows the natural inclusion $\widetilde{\H} \subset \widetilde{\H}'$. 
	Mutually inverse inclusions  $\widetilde{\H} \subset \widetilde{\H}'$ and $\widetilde{\H}' \subset \widetilde{\H}$ yield the isomorphism $\widetilde{\H} \cong \widetilde{\H}'$.
\end{proof}

\begin{empt}\label{h_n_to_h_constr}
	Let $\H_n$ be a Hilbert completion of $A_n \otimes_A \H$ which is constructed in the section \ref{induced_repr_fin_sec}. Clearly
	\begin{equation}\label{tensor_n_equ}
	L^2\left(\widetilde{A}_\pi\right)\otimes_{A_n} \H_n = 	L^2\left(\widetilde{A}_\pi\right) \otimes_{A_n} \left( A_n \otimes_A \H\right) = L^2\left(\widetilde{A}_\pi\right) \otimes_A \H.
	\end{equation}
	
\end{empt}

\section{Quantization of topological coverings}\label{top_chap}

\subsection{Finite-fold coverings}
\paragraph*{}


The following lemma supplies the quantization of coverings with compactification.
\begin{lemma}\label{comm_fin_lem}
	If $\mathcal X$, $\widetilde{\mathcal X}$  are  locally compact spaces, and  $\pi: \widetilde{\mathcal X}\to \mathcal X$ is a surjective continuous map, then following conditions are equivalent:
	\begin{enumerate}
		\item [(i)] The map $\pi: \widetilde{\mathcal X}\to \mathcal X$ is a finite-fold covering with  compactification,
		\item[(ii)] There is a natural  noncommutative finite-fold covering with compactification $$\left(C_0\left(\mathcal  X \right), C_0\left(\widetilde{\mathcal X} \right), G    \right).$$
	\end{enumerate}
\end{lemma}

\begin{proof}
	(i)=>(ii) Denote by ${\mathcal X} \subset {\mathcal Y}$, $\widetilde{\mathcal X} \subset \widetilde{\mathcal Y}$ compactifications such that $\overline{\pi} : \widetilde{\mathcal Y} \to {\mathcal Y}$ is a finite-fold (topological) covering. Let $G = G\left(\widetilde{\mathcal Y}~|~{\mathcal Y} \right) $ be a group of covering transformations. If  $B = C\left( {\mathcal Y}\right)$ and $\widetilde{B}=C\left( \widetilde{\mathcal Y}\right)$ then $A = C_0\left( {\mathcal X}\right)$ (resp. $\widetilde{A}=C_0\left( \widetilde{\mathcal X}\right)$) is an essential ideal of $B$ (resp. $\widetilde{B}$). Taking into account $A=C_0\left( {\mathcal X}\right) = C_0\left( \widetilde{\mathcal X}\right)\bigcap C\left( {\mathcal Y}\right)= B \bigcap \widetilde{A}$ one concludes that these algebras satisfy to the condition (a) of the Definition \ref{fin_comp_def}. From the Theorem \ref{pavlov_troisky_thm} it turns out that the triple $\left( C\left( {\mathcal Y}\right), C\left( \widetilde{\mathcal Y}\right)  ,G\right)=\left(B ,\widetilde{B}, G \right)$ is an unital  noncommutative finite-fold covering. So the condition (b) of the Definition \ref{fin_comp_def} holds.
From $G \widetilde{\mathcal X} = \widetilde{\mathcal X}$ it turns out $G \widetilde{A}= G  C_0\left( \widetilde{\mathcal X}\right) = C_0\left( \widetilde{\mathcal X}\right)= \widetilde{A}$, i.e. the condition (c) of the Definition \ref{fin_comp_def} holds.
	\\
	(ii)=>(i) 
	If $A = C_0\left( {\mathcal X}\right)$, $\widetilde{A} = C_0\left( \widetilde{\mathcal X}\right)$ and inclusions $A \subset B$,  $\widetilde{A}\subset \widetilde{B}$ are such that $A$ (resp. $B$) is an essential ideal of $\widetilde{A}$ (resp. $\widetilde{B}$) then there are compactifications ${   \mathcal X } \hookto {   \mathcal Y }$ and $\widetilde{   \mathcal X } \hookto \widetilde{   \mathcal Y }$ such that $B = C\left(\mathcal Y \right)$,  $\widetilde{B} = C\left(\widetilde{\mathcal Y} \right)$.  From the condition (b) of the Definition \ref{fin_def} it turns out that the triple $\left(B ,\widetilde{B}, G \right)=\left( C \left( {\mathcal Y} \right), C \left( \widetilde{\mathcal Y} \right), G\right) $ is an unital  noncommutative finite-fold covering. From the Theorem \ref{pavlov_troisky_thm} it follows that the *-homomorphism $C\left( {\mathcal Y} \right)\hookto C \left( \widetilde{\mathcal Y} \right)$ induces a finite-fold (topological) covering $\overline{\pi}: \widetilde{\mathcal Y} \to {\mathcal Y}$. From condition (c) of of the Definition \ref{fin_comp_def} it turns out $G  C_0\left( \widetilde{\mathcal X}\right)=  C_0\left( \widetilde{\mathcal X}\right)$ or, equivalently
	\be\label{comm_gx=gx}
	G\widetilde{\mathcal X} = \widetilde{\mathcal X}.
	\ee
	 From $A= B \bigcap \widetilde{ A}$ or, equivalently $C_0\left( {\mathcal X}\right) = C_0\left( \widetilde{\mathcal X}\right)\bigcap C\left( {\mathcal Y}\right)$ and \eqref{comm_gx=gx} it turns out that $\pi$ is the restriction of finite-fold covering $\overline{\pi}$, i.e. $\pi = \overline{\pi}|_{\widetilde{\mathcal X}}$. So $\pi$ is a finite-fold covering. 
\end{proof}
\begin{lemma}\label{comm_fin_top_lem}
	Let $\pi:\widetilde{\mathcal X} \to \mathcal X$ be a surjective map of topological spaces such that there is a family of open subsets $\left\{\mathcal U_\la \subset \mathcal X\right\}_{\la \in \La}$ such that
	\begin{enumerate}
		\item[(a)] $\mathcal X = \bigcup_{\la \in \La} \mathcal U_\la$,
		\item[(b)] For any $\la \in \La$ the natural map $\pi^{-1}\left(\mathcal U_\la \right)\to \mathcal U_\la$ is a covering. 
	\end{enumerate}
Then the map $\pi:\widetilde{\mathcal X} \to \mathcal X$ is a covering.
\end{lemma}
\begin{proof}
For any point $x_0 \in \mathcal X$ there is $\la \in \La$ such that $x_0 \in \mathcal U_\la$. The map $\pi^{-1}\left(\mathcal U_\la \right)\to \mathcal U_\la$ is a covering, it follows that there is an open neighborhood $\mathcal V$ of $x_0$ such that $\mathcal V \subset \mathcal U_\la$ and $\mathcal V$ is evenly covered by $\pi$.
\end{proof}

\begin{theorem}\label{comm_fin_thm}
	If $\mathcal X$, $\widetilde{\mathcal X}$  are  locally compact spaces, and  $\pi: \widetilde{\mathcal X}\to \mathcal X$ is a surjective continuous map, then following conditions are equivalent:
	\begin{enumerate}
		\item [(i)] The map $\pi: \widetilde{\mathcal X}\to \mathcal X$ is a finite-fold regular covering,
		\item[(ii)] There is the natural  noncommutative finite-fold covering $\left(C_0\left(\mathcal  X \right), C_0\left(\widetilde{\mathcal X} \right), G    \right)$.
	\end{enumerate}
\end{theorem}
\begin{proof}
(i)=>(ii) We need check that  $\left(C_0\left(\mathcal  X \right), C_0\left(\widetilde{\mathcal X} \right), G    \right)$ satisfies to condition (a), (b) of the Definition \ref{fin_def}.\\
(a) Covering $\pi$ is regular, so from the Proposition  \ref{reg_cov_prop} it turns out ${\mathcal X} = \widetilde{\mathcal X}/G$ where $G = G\left( \widetilde{\mathcal X}~|~{\mathcal X}\right)$ is a covering group. From   ${\mathcal X} = \widetilde{\mathcal X}/G$ it follows that $C_0\left( {\mathcal X}\right)  = C_0\left( \widetilde{\mathcal X}\right)^G$.\\
(b)
 The space $\mathcal X$ is locally compact, so for any $x \in  \mathcal X$ there is a compact neighborhood $\overline{   \mathcal U }$. The maximal open subset $   \mathcal U \subset \overline{   \mathcal U }$ is an open neighborhood of $x$. So there is family of open subsets $\left\{\mathcal U_\la \subset \mathcal X\right\}_{\la \in \La}$ such that
	\begin{itemize}
	\item $\mathcal X = \bigcup_{\la \in \La} \mathcal U_\la$,
	\item For any $\la \in \La$ the closure  $\overline{\mathcal U}_\la$ of $\mathcal U_\la$ in $\mathcal X$ is compact. 
\end{itemize}
Since $\pi$ is a finite-fold covering the set $\pi^{-1}\left(\overline{\mathcal U}_\la \right)$ is compact for any $\la \in \La$. 
If $\widetilde{I}_\la \subset C_0\left(\widetilde{\mathcal X} \right)$ is a closed ideal given by
$$
\widetilde{I}_\la \stackrel{\mathrm{def}}{=}C_0\left( \pi^{-1}\left( \mathcal U_\la \right)\right)   \cong \left\{\widetilde{a} \in C_0\left( \widetilde{\mathcal X}\right) ~|~ \widetilde{a}\left(\widetilde{\mathcal X} \backslash \pi^{-1}\left( \mathcal U_\la\right)  \right)= \left\{0\right\}  \right\}
$$
then $\widetilde{I}_\la \subset C\left(\overline{\mathcal U}_\la \right)$ is an essential ideal of the unital algebra  $C\left(\pi^{-1}\left( \overline{\mathcal U}_\la \right) \right)$. From $G \pi^{-1}\left( \mathcal U_\la\right) =  \pi^{-1}\left( \mathcal U_\la\right)$ it follows that $G\widetilde{I}_\la = \widetilde{I}_\la$.
If $I_\la = C_0\left(  \mathcal X  \right) \bigcap \widetilde{I}_\la$  then 
 one has
\be\nonumber
\begin{split}
{I}_\la	= C_0\left(\mathcal U_\la \right) \cong  \left\{a \in C_0\left( {\mathcal X}\right) ~|~ a\left({\mathcal X} \backslash  \mathcal U_\la  \right)= \left\{0\right\}  \right\}
\end{split}
\ee 
hence ${I}_\la$ is an essential ideal of an unital algebra $C\left(\overline{\mathcal U}_\la  \right)$.  The restriction map $$\pi_{\pi^{-1}\left( \overline{   \mathcal U }_\la\right)}:\pi^{-1}\left( \overline{   \mathcal U }_\la\right) \to \overline{   \mathcal U }_\la$$ is a finite-fold covering of compact spaces, so from the Theorem \ref{pavlov_troisky_thm} it follows that $$\left(C\left(\overline{\mathcal U}_\la  \right), C\left(\pi^{-1}\left(\overline{\mathcal U}_\la \right)  \right), G \right)$$ is an unital  noncommutative finite-fold covering. It turns out that  $$\left(	I_\la,	\widetilde{I}_\la, G \right)= \left(C_0\left( \mathcal U_\la \right), C_0\left( \pi^{-1}\left(\mathcal U_\la \right), G \right)\right)$$ is a noncommutative finite-fold covering with compactification.  From  $\mathcal X = \bigcup_{\la \in \La} \mathcal U_\la$ (resp. $\widetilde{\mathcal X} = \bigcup_{\la \in \La} \pi^{-1}\left( \mathcal U_\la\right) $)  it turns out that $\bigcup_{\la \in \La} I_\la$ (resp. $\bigcup_{\la \in \La} \widetilde{I}_\la$) is a dense subset of $C_0\left( \mathcal X\right) $ (resp. $C_0\left( \widetilde{\mathcal X}\right) $).\\
(ii)=>(i)
Let  $\left\{\widetilde{I}_\la \subset C_0\left(\widetilde{\mathcal X} \right)  \right\}_{\la \in \La}$ be a family of closed ideals from the condition (b) of the Definition \ref{fin_def}, and let  ${I}_\la = \widetilde{I}_\la \bigcap C_0\left({   \mathcal X } \right) $. If $\widetilde{\mathcal U}_\la \subset \widetilde{\mathcal X}$ is a given by
$$
\widetilde{\mathcal U}_\la \stackrel{\mathrm{def}}{=}\left\{\widetilde{x} \in \widetilde{\mathcal X}~|~ \exists~ \widetilde{a}\in  \widetilde{I}_\la; ~\widetilde{a}\left( \widetilde{x}\right) \neq 0\right\}
$$
then from $G \widetilde{I}_\la = \widetilde{I}_\la$ it turns out $G\widetilde{\mathcal U}_\la=\widetilde{\mathcal U}_\la$.
If $\mathcal U_\la \subset \mathcal X$ is given by
$$
\mathcal U_\la = \left\{x \in\mathcal X~|~\exists a \in I_\la;~~ a\left( x\right) \neq 0  \right\}
$$
then $\mathcal U_\la = \pi \left(\widetilde{\mathcal U}_\la \right)$, and $\widetilde{\mathcal U}_\la= \pi^{-1}\left(\mathcal U_\la \right) $, hence
there is the natural *-isomorphism
$$
\widetilde{I}_\la \cong C_0\left(\pi^{-1}\left(  \mathcal U_\la\right)  \right). 
$$
Any covering is an open map, so if $\overline{   \mathcal U }_\la$ is the closure of $\mathcal U_\la$ in $\mathcal X$ then $\pi^{-1}\left(\overline{   \mathcal U }_\la \right)$ is the closure of $\widetilde{\mathcal U}_\la$ in $\widetilde{\mathcal X}$. Following conditions hold:
\begin{itemize}
	\item $\overline{   \mathcal U }_\la$ (resp. $\pi^{-1}\left( \overline{   \mathcal U }_\la\right)$ ) is a compactification of $\mathcal U_\la$, (resp. $\pi^{-1}\left( {   \mathcal U }_\la\right)$),
	\item $I_\la = C_0\left( \mathcal U_\la\right) $, (resp. $\widetilde{I}_\la =C_0\left( \pi^{-1}\left( {   \mathcal U }_\la\right)\right) $) is an essential ideal of $C\left( \overline{   \mathcal U }_\la\right) $ (resp. $C\left( \pi^{-1}\left( \overline{   \mathcal U }_\la\right)\right) $),
	\item The triple $\left(C\left(  \overline{   \mathcal U }_\la\right), C\left( \pi^{-1}\left( \overline{   \mathcal U }_\la\right)\right), G \right)$ is an unital noncommutative finite-fold covering.
\end{itemize}
It follows that the triple $
\left(I_\la, \widetilde{I}_\la, G \right)=\left(C_0\left(\mathcal U_\la \right), C_0\left(\pi^{-1}\left(\mathcal U_\la \right)  \right), G  \right)
$
is a noncommutative finite-fold covering with compactification, hence from the Lemma  \ref{comm_fin_lem} it follows that the natural map  $\pi^{-1} \left(  \mathcal U_\la\right) \to \mathcal U_\la$ is a covering. From (b) of the Definition \ref{fin_def} it follows that $\bigcup_{\la \in \La} I_\la$ is dense subset of $C_0\left( \mathcal X\right)$ it turns out $\mathcal X = \bigcup \mathcal U_\la$, hence from the Lemma \ref{comm_fin_top_lem} it follows that  $\pi: \widetilde{\mathcal X}\to \mathcal X$ is a finite-fold covering.

\end{proof}

\subsection{Infinite coverings}

\paragraph*{}
This section supplies a purely algebraic  analog of the topological construction given by the Subsection \ref{inf_to}. Suppose that
\begin{equation*}
\mathfrak{S}_\mathcal{X} = \left\{\mathcal{X}_0 \xleftarrow{}... \xleftarrow{} \mathcal{X}_n \xleftarrow{} ... \right\}
\end{equation*}
is  a topological  finite covering sequence. From the Theorem \ref{comm_fin_thm} it turns out that
$\mathfrak{S}_{C_0\left( \mathcal{X}\right)}=\left\{C_0\left( \mathcal{X}_0\right)\to ... \to C_0\left( \mathcal{X}_n\right) \to ...\right\} $ is an algebraical  finite covering sequence.  The following theorem and the corollary gives the construction of $\widehat{C_0\left(\mathcal X \right)} = \varinjlim  C_0\left( \mathcal{X}_n\right) $.
\begin{theorem}\label{direct_lim_state_thm}\cite{takeda:inductive}
	If a $C^*$-algebra $A$ is a $C^*$-inductive limit of $A_\ga$ ($\ga \in \Ga$), the
	state space $\Om$ of A is homeomorphic to the projective limit of the state spaces $\Om_\ga$ of $A_\ga$.
\end{theorem}
\begin{corollary}\label{direct_lim_state_cor}\cite{takeda:inductive}
	If a commutative $C^*$-algebra $A$ is a $C^*$-inductive limit of the
	commutative  $C^*$-algebras $A_\ga$ ($\ga \in \Ga$), the spectrum $\mathcal X$ of $A$ is the projective limit of spectrums $\mathcal X_\ga$ of $A_\ga$ ($\ga \in \Ga$).
\end{corollary}
\begin{empt}
	From the Corollary \ref{direct_lim_state_cor} it turns out $\widehat{C_0\left(\mathcal X \right)} = C_0\left(\widehat{\mathcal X} \right)$ where  $\widehat{\mathcal X}= \varprojlim \mathcal X_n$. If $\overline{  \mathcal X}$ is the disconnected inverse limit of $\mathfrak{S}_\mathcal{X}$ then there is the natural bicontinuous map $f:\overline{  \mathcal X} \to \widehat{  \mathcal X}$. The map induces the injective *-homomorphism  $C_0\left(\widehat{\mathcal X} \right) \hookto  C_b\left(\overline{\mathcal X} \right)$. It follows that there is the natural inclusion of enveloping von Neumann algebras $C_0\left(\widehat{\mathcal X} \right)'' \hookto  C_0\left(\overline{\mathcal X} \right)''$.
	Denote by $G_n = G\left(\mathcal X_n~|~\mathcal X \right)$ groups of covering transformations and $\widehat{G} = \varprojlim G_n$. Denote by $\overline{\pi}:\overline{  \mathcal X} \to \mathcal X$, $\overline{\pi}_n:\overline{  \mathcal X} \to \mathcal X_n$, $\pi^n:  \mathcal X_n \to \mathcal X$, $\pi^m_n:  \mathcal X_m \to \mathcal X_n$ ($m > n$) the natural covering projections.
\end{empt}
\begin{lemma}\label{comm_c_cp_lem}
	Following conditions hold:
	\begin{enumerate}
		\item [(i)]	If $\overline{  \mathcal U} \subset \mathcal \overline{  \mathcal X}$ is a compact set then there is $N \in \N$ such that for any $n \ge N$ the restriction $\overline{\pi}_n|_{\overline{  \mathcal U}}:\overline{  \mathcal U} \xrightarrow{\approx} \overline{\pi}_n\left( {\overline{  \mathcal U}}\right)$ is a homeomorphism,
		\item[(ii)]  If $\overline{a} \in C_c\left(\overline{  \mathcal X }\right)_+ $ is a positive element then there there is $N \in \N$ such that for any $n \ge \N$ following condition holds
		\begin{equation}\label{comm_a_eqn}
	a_n\left(\overline{   \pi }_n \left( \overline{x}\right)\right) =\left\{
	\begin{array}{c l}
	\overline{   a}\left( \overline{x}\right) &\overline{x} \in   \supp~ \overline{a} ~\&~ \overline{   \pi }_n \left( \overline{x}\right) \in \supp~ a_n \\
	0 &\overline{   \pi }_n \left( \overline{x}\right) \notin \supp~ a_n 
	\end{array}\right.
	\end{equation}	
		where 	$$
		a_n  = \sum_{g \in \ker\left( \widehat{G} \to  G_n\right)}g\overline{a}.
		$$
	\end{enumerate}
\end{lemma}
\begin{proof}
(i)	The set $\overline{  \mathcal U}$ is compact, hence $\supp ~\overline{a}$ is a finite disconnected union of connected compact sets, i.e.
	$$
\overline{  \mathcal U}=	\bigsqcup_{j = 1}^M \overline{  \mathcal V }_j
	$$
		It is known \cite{spanier:at} that any covering is an open map, and any open map maps any closed set onto a closed set, so for any $n \in \N$ the set $\overline{\pi}_n\left( {\overline{  \mathcal U}}\right)$ is compact. For any $n \in \N$ denote by $c_n \in \N$ the number of connected components of $\overline{\pi}_n\left( {\overline{  \mathcal U}}\right)$. If $n > m$ then any connected component of $\overline{\pi}_n\left( {\overline{  \mathcal U}}\right) $ is mapped into a connected component of  $\overline{\pi}_m\left( {\overline{  \mathcal U}}\right)$, it turns out $c_n \ge c_m$. Clearly $c_n \le M$. The sequence $\left\{c_n\right\}_{n \in \N}$ is  non-decreasing and $c_n \le M$. It follows that there is $N \in \N$ such that $c_N = M$. For any $n > N$ the set	$\overline{\pi}_n\left( {\overline{  \mathcal U}}\right)$ is mapped homeomorphically onto  $\overline{\pi}_N\left( {\overline{  \mathcal U}}\right)$, hence from the sequence of homeomorphisms it follows 
	$$
	\dots\cong \overline{\pi}_n\left( {\overline{  \mathcal U}}\right) \cong \dots \cong \overline{  \mathcal U}
		$$
	it follows that $\overline{\pi}_n|_{\overline{  \mathcal U}}:\overline{  \mathcal U} \xrightarrow{\approx} \overline{\pi}_n\left( {\overline{  \mathcal U}}\right)$ is a homeomorphism.\\
	(ii)
	The set $\supp \overline{a} = \overline{  \mathcal U}$ is compact, it follows that from (i) and $\overline{a}>0$ that $\supp \overline{a}$ is mapped homeomorphically onto $\supp~ a_N$. 	
	It turns out that if
	$$
	a_n  = \sum_{g \in \ker\left( \widehat{G} \to  G_n\right)}g\overline{a}
	$$
and $n \ge N$	then $a_n$ is given by 	\eqref{comm_a_eqn}.
\end{proof}
\begin{lemma}\label{comm_c_c_lem}
	If $\mathcal X$ is a locally compact Hausdorff space then any positive element $\overline{a} \in C_c\left(\overline{  \mathcal X }\right)_+ $ is special.
\end{lemma}
\begin{proof} 
	From the Lemma \ref{comm_c_cp_lem} it follows that there is $N \in \N$ such that the equation \eqref{comm_a_eqn} holds. It turns out that for any $z \in C_0\left(\mathcal X \right)$, $n \ge \N$ and $f_\eps$ given by \eqref{f_eps_eqn}  the series
	\begin{equation*}
	\begin{split}
	b_n = \sum_{g \in \ker\left( \widehat{G} \to  G_n\right)} g \left(z  \overline{a} z^*\right) ,\\
	c_n = \sum_{g \in \ker\left( \widehat{G} \to  G_n\right)} g \left(z  \overline{a} z^*\right)^2,\\
	d_n = \sum_{g \in \ker\left( \widehat{G} \to  G_n\right)} f_\eps \left(z  \overline{a} z^*\right)\\
	\end{split}
	\end{equation*}
	are given by
\begin{equation}\label{comm_bcd_eqn}
\begin{split}
	b_n\left(\overline{   \pi }_n \left( \overline{x}\right)\right) =\left\{
\begin{array}{c l}
	z\left(\overline{   \pi }_n \left( \overline{x}\right)\right)\overline{   a}\left( \overline{x}\right)	z^*\left(\overline{   \pi }_n \left( \overline{x}\right)\right) &\overline{x} \in   \supp~ \overline{a} ~\&~ \overline{   \pi }_n \left( \overline{x}\right) \in \supp~ a_n \\
0 &\overline{   \pi }_n \left( \overline{x}\right) \notin \supp~ a_n 
\end{array}\right.,\\
	c_n\left(\overline{   \pi }_n \left( \overline{x}\right)\right) =\left\{
\begin{array}{c l}
\left( z\left(\overline{   \pi }_n \left( \overline{x}\right)\right)\overline{   a}\left( \overline{x}\right)	z^*\left(\overline{   \pi }_n\right) \left( \overline{x}\right)\right)^2 &\overline{x} \in   \supp~ \overline{a} ~\&~ \overline{   \pi }_n \left( \overline{x}\right) \in \supp~ a_n \\
0 &\overline{   \pi }_n \left( \overline{x}\right) \notin \supp~ a_n 
\end{array}\right.,\\
	d_n\left(\overline{   \pi }_n \left( \overline{x}\right)\right) =\left\{
\begin{array}{c l}
f_\eps\left( z\left(\overline{   \pi }_n \left( \overline{x}\right)\right)\overline{   a}\left( \overline{x}\right)	z^*\left(\overline{   \pi }_n \left( \overline{x}\right)\right)\right)  &\overline{x} \in   \supp~ \overline{a} ~\&~ \overline{   \pi }_n \left( \overline{x}\right) \in \supp~ a_n \\
0 &\overline{   \pi }_n \left( \overline{x}\right) \notin \supp~ a_n 
\end{array}\right..\\
\end{split}	
	\end{equation}
	From  \eqref{comm_bcd_eqn} it turns out $	b_n^2 = c_n$, i.e. $\overline{   a}$	 satisfies to the condition (c) of the Definition \ref{special_el_defn}. Otherwise \eqref{comm_a_eqn}, \eqref{comm_bcd_eqn} from 
that $a_n,~ b_n, ~c_n,~ d_n\in  C_0\left(\mathcal X_n \right)$  for any $n \ge N$.	  If $n < N$ then 
	\begin{equation*}
	\begin{split}
	a_n = \sum_{g\in G\left( \mathcal X_N~|~\mathcal X_n\right)} g a_N,\\
	b_n = \sum_{g\in G\left( \mathcal X_N~|~\mathcal X_n\right)} g b_N,\\
	c_n = \sum_{g \in G\left( \mathcal X_N~|~\mathcal X_n\right)} g c_N,\\
	d_n = \sum_{g \in G\left( \mathcal X_N~|~\mathcal X_n\right)} g d_N.
	\end{split}
	\end{equation*}
	Above sums are finite, 
	it turns out $a_n,~ b_n, ~c_n,~ d_n\in C_0\left(\mathcal X_n \right)$ for any $n \in \N^0$, i.e. $\overline{   a}$ satisfies to conditions (a), (b) of the Definition \ref{special_el_defn}.

\end{proof}
\begin{corollary}\label{comm_c_c_cor}
	If $\overline{A}$ is a  disconnected inverse noncommutative limit of $$\mathfrak{S}_{C_0\left( \mathcal{X}\right)}=\left\{C_0\left( \mathcal{X}_0\right)\to ... \to C_0\left( \mathcal{X}_n\right) \to ...\right\} $$ then $C_0\left(\overline{\mathcal X} \right)\subset\overline{A}$.
\end{corollary}
\begin{proof}
	From the Lemma 	\ref{comm_c_c_lem} it follows that $C_c\left(\overline{\mathcal X} \right) \subset \overline{A}$, and taking into account the Definition \ref{c_c_def_1} one has $   C_0\left(\overline{\mathcal X} \right)\subset\overline{A}$.
\end{proof}

\begin{lemma}\label{comm_main_lem}
	Suppose that $\mathcal X$ is a  locally compact Hausdorff space. Let $\overline{a} \in C_0\left(\overline{   \mathcal X }\ \right)''_+$ be such that following conditions hold:
	\begin{enumerate}
		\item[(a)]  If $f_\eps$ is given by \eqref{f_eps_eqn} then following series
		\begin{equation*}
		\begin{split}
		a_n = \sum_{g \in \ker\left( \widehat{G} \to  G_n\right)} g \overline a,\\
		b_n = \sum_{g \in \ker\left( \widehat{G} \to  G_n\right)} g \overline a^2,\\
		c_n = \sum_{g \in \ker\left( \widehat{G} \to  G_n\right)} g f_\eps\left( \overline a\right) ,\\
		\end{split}
		\end{equation*}
		are strongly convergent and  $a_n, b_n, c_n \in C_0\left(\mathcal X_n \right)$,
		\item[(b)] For any $\eps > 0$ there is $N \in \N$ such that 
		\begin{equation*}
		\begin{split}
		\left\|a^2_n-b_n\right\| < \eps; ~\forall n \ge N.
		\end{split}
		\end{equation*}
	\end{enumerate}
	Then $\overline{a} \in C_0\left(\overline{   \mathcal X }\ \right)_+$.
\end{lemma}
\begin{proof}
	The dual space $C_0\left(\overline{   \mathcal X }\right)^*$ of $C_0\left(\overline{   \mathcal X }\right)$ is a space of Radon measures on $\overline{   \mathcal X }$. 
	If $\overline{f}: \overline{   \mathcal X } \to \R$ is given by
	\begin{equation*}
	\begin{split}
	\overline{f}\left( 	\overline{x}\right)   = \lim_{n \to \infty} a_n\left(\overline{\pi}_n\left( \overline{x}\right)\right)= \inf_{n \in \N} a_n\left(\overline{\pi}_n\left( \overline{x}\right)\right)
	\end{split}
	\end{equation*}
	then from the Proposition \ref{env_alg_sec_dual} and the Lemma \ref{stong_conv_inf_lem} it follows that 	$\overline{f}$ represents $\overline{a}$, i.e. following conditions hold:
	\begin{itemize}
		\item The function $\overline{f}$ defines a following functional
		\begin{equation*}\label{comm_func_eqn}
		\begin{split}
		C_0\left(\overline{   \mathcal X }\right)^* \to \C,\\
		\mu \mapsto \int_{	\overline{   \mathcal X }} \overline{f}~d\mu 
		\end{split}
		\end{equation*}
		where $\mu$ is a Radon measure on $\overline{   \mathcal X }$,
		\item The functional corresponds to $\overline{   a}\in 	C_0\left(\overline{   \mathcal X }\right)^{**}=C_0\left(\overline{   \mathcal X }\right)''$.
	\end{itemize}
	If $m > n$ then 
	\begin{equation}\label{comm_an_bn_eqn}
	\begin{split}
	a_n = \sum_{g \in G\left(\mathcal X_m ~|  \mathcal X_n\right) } g  a_m,\\
	b_n = \sum_{g \in G\left(\mathcal X_m ~|  \mathcal X_n\right) } g  b_m.\\
	\end{split}
	\end{equation}
	Let $M \in \N$ be such that for any $n \ge M$ following condition holds
	\begin{equation}\label{comm_delta_eqn}
	\begin{split}
	\left\|a^2_n - b_n\right\| < 2 \eps^2.
	\end{split}
	\end{equation}
	
Let $n > M$, $p_n =  \pi^{n}_M: \mathcal X_m \to\mathcal X_M$, and let $\widetilde{x}_1, \widetilde{x}_2 \in \mathcal X_{n}$  be such that
	\begin{equation}\label{comm_neq_eqn}
\begin{split}
\widetilde{x}_1\neq \widetilde{x}_2,\\
p_n\left( \widetilde{x}_1\right) = p_n\left( \widetilde{x}_2\right)=x,\\
a_n\left( \widetilde{x}_1\right) \ge \eps; ~a_n\left( \widetilde{x}_2\right) \ge \eps.
\end{split}
\end{equation}
  From  \eqref{comm_an_bn_eqn} it turns out
	\begin{equation*}
	\begin{split}
	a_M\left(x \right)  = \sum_{\widetilde{x} \in p_n^{-1}\left( x\right)  } a_{n}\left( \widetilde{x}\right) ,\\
	b_M\left(x \right)  = \sum_{\widetilde{x} \in p_n^{-1}\left( x\right)  } b_{n}\left( \widetilde{x}\right).\\
	\end{split}
	\end{equation*}
	From the above equation and  $a^2_n \ge b_n$ it turns out
	\begin{equation*}
	\begin{split}
a_M^2\left(x \right)= \sum_{\widetilde{x} \in p_n^{-1}\left( x\right)} a_n^2\left( \widetilde{x}\right) + \sum_{\substack{\left(\widetilde{x}', \widetilde{x}'' \right)\in p_n^{-1}\left(x \right)\times p_n^{-1}\left(x \right)\\\widetilde{x}' \neq \widetilde{x}'' }}~ a_n\left(\widetilde{x}' \right) a_n\left(\widetilde{x}'' \right)\ge\\
\ge
 \sum_{\widetilde{x} \in p_n^{-1}\left( x\right)} b_n\left( \widetilde{x}\right) + a_n\left( \widetilde{x}_1\right) a_n\left( \widetilde{x}_2\right)+a_n\left( \widetilde{x}_2\right) a_n\left( \widetilde{x}_1\right)=\\= b_M\left(x \right) + 2  a_n\left( \widetilde{x}_1\right) a_n\left( \widetilde{x}_2\right). 
	\end{split}
	\end{equation*}

	Taking into account $a_n\left( \widetilde{x}_1\right) \ge \eps$, $a_n\left( \widetilde{x}_2\right) \ge \eps$ one has
	\begin{equation*}
\begin{split}
a^2_M\left(x \right) - b_M\left(x \right)  \ge 2 \eps^2,\\\left\|a^2_M - b_M\right\| \ge 2 \eps^2.
\end{split}
\end{equation*}
So \eqref{comm_neq_eqn} contradicts with \eqref{comm_delta_eqn}, it follows that
\begin{equation}\label{comm_x1_eq_x2_eqn}
\begin{split}
p_n\left( \widetilde{x}_1\right) = p_n\left( \widetilde{x}_2\right)=x ~\&~
a_n\left( \widetilde{x}_1\right) \ge \eps~\&~a_n\left( \widetilde{x}_2\right) \ge \eps \Rightarrow \widetilde{x}_1 = \widetilde{x}_2.
\end{split}
\end{equation}
If $f_\eps$ is given by \eqref{f_eps_eqn} and
$$
	c_n = \sum_{g \in \ker\left( \widehat{G} \to  G_n\right)} g f_\eps\left( \overline a\right)
$$
then
\begin{equation*}
\begin{split}
\supp c_n = \left\{x \in \mathcal X_n~|~\inf_{m > n} ~~\max_{\widetilde x \in \pi^m_n\left( x\right) } a_m\left(\widetilde x \right)\ge \eps \right\}= \\
= \left\{x \in \mathcal X_n~|~\exists~ \overline{x}\in \overline{   \mathcal X }; ~\overline{\pi}_n\left(\overline{x} \right) = x ~\&~\overline{f}\left(\overline{x} \right) \ge \eps \right\}.
\end{split}
\end{equation*}
Indeed $f_\eps\left( \overline{   a}\right) $ as a functional on $C_0\left(\overline{   \mathcal X } \right)^*$ is represented by the following  function
\begin{equation*}
\begin{split}
\overline{f}_\eps: \overline{   \mathcal X } \to \R,\\
\overline{   x}\mapsto f_\eps\left(\overline{f}\left(\overline{   x} \right)  \right). 
\end{split}
\end{equation*} 

From $\widetilde x \in \supp c_n$ it turns out $a_n\left(\widetilde x \right) \ge \eps$ and taking into account \eqref{comm_x1_eq_x2_eqn} one concludes that the restriction $\pi^n_M|_{\supp c_n}$ is an injective map. Clearly $\pi^n_M\left(\supp c_n \right) = \supp c_M$, so there is a bijection $\supp c_n \xrightarrow{\approx}\supp c_M$. The map $\pi^n_M$ is a covering and it is known \cite{spanier:at} that any covering is an open map. Any bijective open map is a homeomorphism, hence one has a sequence of homeomorphisms
\begin{equation}\label{comm_homeo_sec_eqn}
\supp c_M \xleftarrow{\approx}\dots \xleftarrow{\approx}\supp c_n \xleftarrow{\approx}\dots
\end{equation}
If $\overline{   \mathcal U } \subset \overline{   \mathcal X }$ is given by
$$
\overline{   \mathcal U } = \bigcap_{n = M}^\infty \overline{   \pi }^{-1}_n\left(\supp c_n \right) 
$$
then from \eqref{comm_homeo_sec_eqn} it turns out that  $ \overline \pi_M$ homeomorphically maps $\overline{   \mathcal U }$ onto $\supp c_M$. Moreover following condition holds 
\begin{equation*}
\begin{split}
f_\eps\left(\overline{a} \right)  \left( \overline{x}\right) =\left\{
\begin{array}{c l}
c_M\left(\overline \pi_M\left(\overline{x} \right)  \right)  &\overline{x} \in   \overline{   \mathcal U } \\
0 &\overline{x} \notin   \overline{   \mathcal U }
\end{array}\right..
\end{split}	
\end{equation*}
From the above equation it follows that $f_\eps\left(\overline{a} \right)$ is a continuous function, i.e. $f_\eps\left(\overline{a} \right)\in C_b\left(\overline{   \mathcal X } \right) $ 
From the Definition \ref{c_c_def_2} it  turns out $D = \left\{x \in \mathcal X_M~|~a_M\left(x\right)\ge \eps  \right\}$ is compact, therefore the closed subset $\supp c_M \subset D$ is  compact, hence $\overline{   \mathcal U }= \supp f_\eps\left(\overline{a} \right)\approx \supp c_M$ is also compact. It turns out
 $f_\eps\left(\overline{a} \right)\in C_c\left(\overline{   \mathcal X } \right) $. From $\left\| f_\eps\left(\overline{a} \right)- \overline{a} \right\| \le \eps$ it follows that $\overline{a}= \lim_{\eps \to 0}f_\eps\left(\overline{a} \right)$ and from the Definition \ref{c_c_def_1} it turns out $\overline{a} \in C_0\left(\overline{   \mathcal X } \right)$.

\end{proof}
\begin{corollary}\label{comm_main_cor}
If $\mathfrak{S}_{C_0\left( \mathcal{X}\right)}=\left\{C_0\left( \mathcal{X}_0\right)\to ... \to C_0\left( \mathcal{X}_n\right) \to ...\right\} $ and $\overline{A}$ is a  disconnected inverse noncommutative limit of $\downarrow\mathfrak{S}_{C_0\left( \mathcal{X}\right)}$ then following conditions hold:
\begin{enumerate}
	\item [(i)] Any special element $\overline{a} \in C_0\left(\overline{   \mathcal X }\ \right)''_+$ of $\mathfrak{S}_{C_0\left( \mathcal{X}\right)}$  lies in $C_0\left(\overline{   \mathcal X }\ \right)$, i.e. $\overline{a} \in C_0\left(\overline{   \mathcal X }\ \right)$,
	\item[(ii)] $C_0\left(\overline{\mathcal X} \right)\subset \overline{A} $.
\end{enumerate}

\end{corollary}
\begin{proof}
(i) Let $\left\{e_\lambda \in C_0\left( \mathcal{X}\right) \right\}_{\lambda \in \Lambda}$ be an approximate unit of $C_0\left( \mathcal{X}\right)$. From the Definition \ref{special_el_defn} it follows that $\overline{b}_\lambda =  e_\lambda \overline{a} e_\lambda$ satisfies to conditions of the Lemma \ref{comm_main_lem}. Otherwise from the Lemma \ref{comm_main_lem} it turns out $\overline{b}_\lambda \in C_0\left(\overline{\mathcal X} \right)$. From the $C^*$-norm limit 
$
\lim_{\lambda \in \Lambda} \overline{b}_\lambda = \overline{a}
$
it follows that $\overline{a}\in C_0\left(\overline{\mathcal X} \right)$.
\\
(ii) Follows from (i) and the Definitions \ref{special_el_defn}, \ref{main_defn_full}. 
\end{proof}

\begin{empt}\label{comm_transitive_constr}
	Let $\widetilde{\mathcal X} \subset \overline{\mathcal X}$ be a connected component of $ \overline{\mathcal X}$ and suppose that $$G \subset G\left(\varprojlim C_0\left(\mathcal X_n \right) ~|~C_0\left(  \mathcal X\right)   \right)$$ is the maximal among subgroups $G'$ such that $G'\widetilde{\mathcal X} = \widetilde{\mathcal X}$. If $J\subset\widehat{G}$ is a set of representatives of $\widehat{G}/G$ then  from the \eqref{top_disconnected_repr_eqn} it follows that
	\begin{equation*}
	\overline{\mathcal X}=  \bigsqcup_{g \in J} g \widetilde{\mathcal X}
	\end{equation*}
	and $C_0\left( \overline{\mathcal X}\right) $ is a $C^*$-norm completion of the direct sum
	\begin{equation}\label{comm_transitive_eqn}
	\bigoplus _{g \in J}  C_0\left( g\widetilde{\mathcal X} \right) .
	\end{equation}

\end{empt}

\begin{theorem}\label{comm_main_thm}
	If $\mathfrak{S}_{\mathcal X} = \left\{\mathcal{X} = \mathcal{X}_0 \xleftarrow{}... \xleftarrow{} \mathcal{X}_n \xleftarrow{} ...\right\} \in \mathfrak{FinTop}$ and
	$$\mathfrak{S}_{C_0\left(\mathcal{X}\right)}=
	\left\{C_0(\mathcal{X})=C_0(\mathcal{X}_0)\to ... \to C_0(\mathcal{X}_n) \to ...\right\} \in \mathfrak{FinAlg}$$ is an algebraical  finite covering sequence then following conditions hold:
	\begin{enumerate}
		\item [(i)] $\mathfrak{S}_{C_0\left(\mathcal{X}\right)}$ is good,
		\item[(ii)] There are  isomorphisms:

		\begin{itemize}
			\item $\varprojlim \downarrow \mathfrak{S}_{C_0\left(\mathcal{X}\right)} \approx C_0\left(\varprojlim \downarrow \mathfrak{S}_{\mathcal X}\right)$;
			\item $G\left(\varprojlim \downarrow \mathfrak{S}_{C_0\left(\mathcal{X}\right)}~|~ C_0\left(\mathcal X\right)\right) \approx G\left(\varprojlim \downarrow \mathfrak{S}_{\mathcal{X}}~|~ \mathcal X\right)$.
		\end{itemize}
	\end{enumerate}
	
\end{theorem}
\begin{proof}
	The proof of this theorem uses a following notation:
	\begin{itemize}
		\item The topological inverse limit $\widetilde{\mathcal X}= \varprojlim \downarrow \mathfrak{S}_{\mathcal{X}}$;
		\item The limit in the category of spaces and continuous maps $\widehat{\mathcal X} = \varprojlim \mathcal X_n$;
		\item The disconnected covering space $\overline{\mathcal X}$ of $\mathfrak{S}_{\mathcal X}$;
		\item The disconnected covering algebra $\overline{A}$ of $\mathfrak{S}_{C_0\left(\mathcal{X}\right)}$;
		\item A connected component $\widetilde{A}\subset \overline{A}$;
		\item The disconnected $\overline{G}_{\mathcal X} = \varprojlim G\left(\mathcal X_n ~|~ \mathcal X \right)$ and the connected $G_{\mathcal X} = G\left(\widetilde{\mathcal X} ~| ~\mathcal X \right)=G\left(\varprojlim \downarrow \mathfrak{S}_{\mathcal{X}}~|~ \mathcal X\right)$ covering groups of $\mathfrak{S}_{\mathcal X}$;
		\item The disconnected group $\overline{G}_{C_0\left(\mathcal X\right)} = \varprojlim G\left(C_0\left(\mathcal X_n\right) ~| ~C_0\left(\mathcal X\right) \right)$ and $\widetilde{A}$-invariant group $G_A$. 
	\end{itemize}
	
	From the Corollary \ref{comm_c_c_cor} it follows that $\overline{A}\subset C_0\left(\overline{\mathcal X}\right)$. From the Corollary  \ref{comm_main_cor} it turns out $C_0\left(\overline{\mathcal X}\right) \subset \overline{A}$, hence $\overline{A}= C_0\left(\overline{\mathcal X}\right)$. If $~J \subset \overline{G}_{\mathcal{X}}$ is a set of representatives of $\overline{G}_{\mathcal X}/G\left(\widetilde{\mathcal X}~|~ \mathcal X\right)$ then $\overline{\mathcal X} = \bigsqcup_{\overline{g}\in J}\overline{g}\widetilde{\mathcal X}$ is the disconnected union of connected homeomorphic spaces, i.e. $\overline{g}\widetilde{\mathcal X}\xrightarrow{\approx}\widetilde{\mathcal X}$. 
	
	(i) We need check (a) - (c) of the Definition \ref{good_seq_defn}.
	$\overline{A}= C_0\left(\overline{\mathcal X}\right)$ is the $C^*$-norm completion of the algebraical the direct sum \eqref{comm_transitive_eqn}.
	Any maximal irreducible subalgebra of $\overline{A}$ is isomorphic to $C_0\left(\widetilde{\mathcal X}\right)$. The map $\widetilde{\mathcal X} \to {\mathcal X}_n$ is a covering for any $n \in \N$, it turns out $C_0\left( \mathcal{X}_n\right) \hookto C_b\left(\widetilde{\mathcal X}\right)= M\left(C_0\left(\widetilde{\mathcal X}\right) \right)$ is injective *-homomorphism. It follows that the natural *-homomorphism $C_0\left(\widehat{\mathcal X}\right)= \lim_{n \to \infty}C_0 \left( {\mathcal X}_n\right) \hookto C_b\left(\widetilde{\mathcal X}\right)=M\left(  C_0\left(\widetilde{\mathcal X}\right)\right) $ is injective, i.e. the  condition (a) holds. The algebraic direct sum $\bigoplus_{\overline{g}\in J} \overline{g}C_0\left(\widetilde{\mathcal X}\right)$ is is a dense subalgebra of $\overline{A}$, i.e. condition (b) holds. The homomorphism $ G\left(\widetilde{\mathcal X} ~| ~\mathcal X \right) \to  G\left(\mathcal X_n ~| ~\mathcal X \right)$ is surjective for any $n \in \N$. From the following isomorphisms
	\begin{equation*}
	\begin{split}
	G_A \approx G\left(\widetilde{\mathcal X} ~| ~\mathcal X \right), \\
	G\left(C_0\left( \mathcal X_n\right)  ~| ~C_0\left( \mathcal X\right)  \right) \approx G\left(\mathcal X_n ~| ~\mathcal X \right),
	\end{split}
	\end{equation*}
	it turns out that $G_A \to G\left(C_0\left( \mathcal X_n\right)  ~| ~C_0\left( \mathcal X\right)  \right)$ is surjective, i.e. condition (c) holds.
	\newline
	(ii) From the proof of (i) it turns out
	\begin{equation*}
	\begin{split}
	\varprojlim \downarrow\mathfrak{S}_{\mathcal X} = \widetilde{\mathcal X};~~\widetilde{A} = C_0\left(\widetilde{\mathcal X}\right),
	\\
	\varprojlim \downarrow \mathfrak{S}_{C_0\left(\mathcal{X}\right)} = \widetilde{A}= C_0\left(\widetilde{\mathcal X}\right)= C_0\left(\varprojlim \downarrow\mathfrak{S}_{\mathcal X} \right),
	\\
	G\left(\varprojlim \downarrow \mathfrak{S}_{\mathcal{X}}~|~ \mathcal X\right) = G_{\mathcal X}=G_A = G\left(\varprojlim \downarrow \mathfrak{S}_{C_0\left(\mathcal{X}\right)}~|~ C_0\left(\mathcal X\right)\right).
	\end{split}
	\end{equation*}
	
\end{proof}

\section{Continuous trace $C^*$-algebras and their  coverings}\label{cont_tr_exm}

\paragraph*{}   Let $A$ be a $C^*$-algebra. For each positive $x\in A_+$ and irreducible representation $\pi: A \to B\left( \H\right)$   the (canonical) trace of $\pi(x)$ depends only on the equivalence class of $\pi$, so that we may define a function $\hat x : \hat A \to [0,\infty]$ by $\hat x(t)=\mathrm{Tr}(\pi(x))$ where $\hat A$ is the space of equivalence classes of irreducible representations. From Proposition 4.4.9 \cite{pedersen:ca_aut} it follows that $\hat x$ is lower semicontinuous function  in the Jacobson topology.

\begin{defn}\label{continuous_trace_c_a_defn}\cite{pedersen:ca_aut} We say that element $x\in A$ has {\it continuous trace} if $\hat x \in C_b(\hat A)$. We say that $C^*$-algebra has {\it continuous trace} if a set of elements with continuous trace is dense in $A$.
\end{defn}

\begin{defn}\label{abelian_element_defn}\cite{pedersen:ca_aut}
	A positive element in $C^*$ - algebra $A$ is {\it Abelian} if subalgebra $xAx \subset A$ is commutative.
\end{defn}
\begin{defn}\cite{pedersen:ca_aut}
	We say that a $C^*$-algebra $A$ is of type $I$ if each non-zero quotient of $A$ contains a non-zero
	Abelian element. If $A$ is even generated (as $C^*$-algebra) by its Abelian elements we say
	that it is of type $I_0$.
\end{defn}

\begin{prop}\label{abelian_element_proposition}\cite{pedersen:ca_aut}
	A positive element $x$ in $C^*$-algebra $A$ is Abelian if $\mathrm{dim}~\pi(x) \le 1$ for every irreducible representation $\pi:A \to B(\H)$.
\end{prop}


\begin{thm}\label{peder_id_thm}  \cite{pedersen:ca_aut} For each $C^*$-algebra $A$ there is a dense hereditary ideal $K(A)$,
	which is minimal among dense ideals.
	
\end{thm}
 
\begin{defn}
	The ideal $K(A)$ from the theorem \ref{peder_id_thm} is said to be the {\it Pedersen ideal of $A$}. Henceforth Pedersen ideal shall be denoted by $K(A)$.
\end{defn}
\begin{prop}\label{continuous_trace_c_a_proposition}\cite{pedersen:ca_aut}
	Let $A$ be a $C^*$ - algebra with continuous trace. Then
	\begin{enumerate}
		\item[(i)] $A$ is of type $I_0$;
		\item[(ii)] $\hat A$ is a locally compact Hausdorff space;
		\item[(iii)] For each $t \in \hat A$ there is an Abelian element $x \in A$ such that $\hat x \in K(\hat A)$ and $\hat x(t) = 1$.
	\end{enumerate}
	The last condition is sufficient for $A$ to have continuous trace.
\end{prop}

\begin{rem}\label{ctr_is_ccr}
	From \cite{dixmier_tr}, Proposition 10, II.9 it follows that a continuous trace
	$C^*$-algebra $A$ is always a $CCR$-algebra, i.e. for every irreducible
	representation $\rho:A \to B(H)$ following condition hold
	\begin{equation}\label{ccr_compact}
	\rho\left(A\right) \approx \mathcal{K}
	\end{equation}
\end{rem}

	
	
\begin{thm}\label{ctr_di_do_thm}\cite{ros:ctr} (Dixmier–Douady). Any stable separable algebra $A$ of continuous
	trace over a second-countable locally compact Hausdorff space $\mathcal{X}$ is isomorphic to
$\Ga_0\left( \mathcal{X}\right)$ , the sections vanishing at infinity of a locally trivial bundle of algebras
	over $\mathcal{X}$, with fibres $\mathcal K$ and structure group $\Aut(\mathcal{K}) = PU = U/T$. Classes of 	such bundles are in natural bijection with the \v{C}ech cohomology group $H_3(\mathcal{X}, \Z)$.
	The 3-cohomology class $\delta(A)$ attached to (the stabilisation of) a continuous-trace
	algebra $A$ is called its Dixmier–Douady class.
\end{thm}

\begin{rem}
	Any commutative $C^*$-algebra has continuous trace. So described in the Section \ref{top_chap}  case is a special case of described in the Section  \ref{cont_tr_exm} construction.
\end{rem}

	\begin{empt}
	For any $x \in \hat A$ denote by $\rho_x: A \to B\left(\H \right)$ a representation which corresponds to $x$. For any $a \in A$ denote by $\supp a \subset \hat A$ the closure of the set
	$$
	\supp a \stackrel{\mathrm{def}}{=}\left\{x \in \hat A~|~\rho_x\left(a \right)\neq 0  \right\}.
	$$
\end{empt}

\subsection{Basic construction}
\paragraph*{}
Let $A$ be a continuous trace $C^*$-algebra such that the spectrum $\hat A= \mathcal X$ of is a second-countable locally compact Hausdorff space. For any open subset $\mathcal U \subset \mathcal X$ denote by
$$
A\left( \mathcal U\right) = \left\{a \in A~|~ \rho_x\left(a \right)= 0; ~\forall x \in \mathcal X \backslash \mathcal U  \right\}
$$
where $\rho_x$ is an irreducible representation which corresponds to $x \in \mathcal X$.
If $\mathcal V \subset \mathcal U$ then there is a natural inclusion $A\left( \mathcal V\right) \hookto A\left( \mathcal U\right)$. Let $\pi: \widetilde{\mathcal X}\to \mathcal X$ be a topological covering. Let $\widetilde{\mathcal U} \subset \widetilde{\mathcal X}$ be a connected open subset  homeomorphically mapped onto $\mathcal U = \pi\left(  \widetilde{\mathcal U}\right)$, and suppose that the closure of $\widetilde{\mathcal U}$ is compact. Denote by  $\widetilde{A}\left( {\widetilde{\mathcal U}}\right) $ the algebra such that $\widetilde{A}\left( {\widetilde{\mathcal U}}\right) \approx A\left( \mathcal U\right)$. If $\widetilde{\mathcal V} \subset \widetilde{\mathcal U}$ and $\mathcal V = \pi\left( \widetilde{\mathcal V}\right)$  then the inclusion $A\left( \mathcal V\right) \hookto A\left( \mathcal U\right)$ naturally induces an inclusion $i^{\widetilde{\mathcal V}}_{\widetilde{\mathcal U}} :\widetilde{A}\left( {\widetilde{\mathcal V}}\right) \hookto \widetilde{A}\left( {\widetilde{\mathcal U}}\right)$.
Let us consider $\widetilde{\mathcal U}$ as indexes, and let
\begin{equation*}
\begin{split}
A' = \bigoplus_{\widetilde{\mathcal U}} \widetilde{A}\left( {\widetilde{\mathcal U}}\right) / I,\\
\end{split}
\end{equation*}
where $\bigoplus$ means the algebraic direct sum of $C^*$-algebras and  $I$ is the two sided ideal generated by elements $i^{\widetilde{\mathcal U}_1\bigcap\widetilde{\mathcal U}_2}_{\widetilde{\mathcal U}_1}\left(a \right) - i^{\widetilde{\mathcal U}_1\bigcap\widetilde{\mathcal U}_2}_{\widetilde{\mathcal U}_2}\left(a \right)$, for any $a \in A\left(\widetilde{\mathcal U}_1\bigcap\widetilde{\mathcal U}_2 \right)$. There is the natural $C^*$-norm of the direct sum on $\bigoplus_{\widetilde{\mathcal U}} \widetilde{A}\left( {\widetilde{\mathcal U}}\right)$ and let us define the norm on $A' = \bigoplus_{\widetilde{\mathcal U}} \widetilde{A}\left( {\widetilde{\mathcal U}}\right) / I$ given by
\be\label{ctr_norm_eqn}
\left\| a + I\right\|= \inf_{a' \in I}  \left\| a + a'\right\|; ~\forall \left( a + I\right)  \in \bigoplus_{\widetilde{\mathcal U}} \widetilde{A}\left( {\widetilde{\mathcal U}}\right) / I
\ee
 \begin{defn}\label{ctr_cov_defn}
 If $A\left(\widetilde{\mathcal X} \right) $ is completion of $A'$ with respect to the given by \eqref{ctr_norm_eqn} then we say that $A\left(\widetilde{\mathcal X} \right)$  is an \textit{induced by  $\pi: \widetilde{\mathcal X}\to \mathcal X$ covering} of $A$.  
 \end{defn}
The action of $G\left(\widetilde{\mathcal X}~|~\mathcal X\right)$ on  $\widetilde{\mathcal X}$ induces an action of $G\left(\widetilde{\mathcal X}~|~\mathcal X\right)$ on  $A'$, so there is a natural action of $G\left(\widetilde{\mathcal X}~|~\mathcal X\right)$ on $A\left(\widetilde{\mathcal X} \right)$.
\begin{definition}
	We say that the action of $G\left(\widetilde{\mathcal X}~|~\mathcal X\right)$ on  $\widetilde{\mathcal X}$ \textit{induces} the action of $G\left(\widetilde{\mathcal X}~|~\mathcal X\right)$ on  $A\left(\widetilde{\mathcal X} \right)$.
\end{definition} 
From the Proposition \ref{continuous_trace_c_a_proposition} it follows that $A\left(\widetilde{\mathcal X} \right) $ is a continuous trace $C^*$-algebra, and the spectrum of $A\left(\widetilde{\mathcal X} \right)$ coincides with $\widetilde{\mathcal X}$. If $G=G\left(\widetilde{\mathcal X}~|~\mathcal X\right)$ is a finite group then
\begin{equation}\label{ctr_inv_eqn}
A = A\left(\widetilde{\mathcal X} \right)^G
\end{equation}
and the above equation induces an injective *-homomorphism $A \hookto A\left(\widetilde{\mathcal X} \right)$.
\begin{lem}\label{ctr_adm_lem}
	Let $A$ be a continuous trace  $C^*$-algebra, and let $\mathcal X = \hat A$ be a spectrum of $A$.  Suppose that $\mathcal X$ is a second-countable locally compact Hausdorff space and $B$ is a $C^*$-algebra such that
	\begin{itemize}
		\item $A \subset B \subset A''$,
		\item For any $b \in B_+$ and $x_0 \in  \mathcal X$ such that $\rho_{x_0}\left(b \right)\neq 0$  there is an open neighborhood  $\mathcal W \subset \mathcal X$ of $x_0$ and an Abelian $z \in A$ such that
		\begin{equation*}
		\begin{split}
		\supp z \subset \mathcal W,\\
		\tr\left(z b z \right) \in C_0\left(\mathcal X \right),\\
				\tr\left(z b z \right) \left(x_0 \right)\neq 0.
				\end{split}
		\end{equation*}
	\end{itemize}
	Then $B = A$.
\end{lem}

\begin{proof}

The spectrum $\hat B$ of $B$ coincides with the spectrum of $A$ as a set. Let $\mathcal V \subset \mathcal X$ be a closed subset with respect to topology of $\hat B$. There is a closed ideal $I \subset B$ which corresponds to $\mathcal V$.  Denote by $I_+$ the positive part of $I$.  For any $x_0 \in \mathcal X \backslash \mathcal V$ there is  $b \in I_+$ such that $\rho_{x_0}\left(b \right) \neq 0$. There is an Abelian element $\overline{z} \in A$ such that $\tr\left( \rho_{x_0}\left(\overline{z}b\overline{z} \right)\right)  \neq 0$. If $\mathcal W \subset\mathcal X$ is an open neighborhood of $x_0$ then from the Corollary \ref{com_a_u_cor} it follows that there is a bounded positive continuous function $a: \mathcal X \to \R$ such that $a\left( x_0\right) \neq 0$ and $a\left(\mathcal X \backslash \mathcal W \right)= \{0\}$. If $z = a\overline{z}$ then $z$ is an Abelian document, $\tr\left(z b z \right)\left(x_0 \right)=\left( \tr\left(\overline{z} b \overline{z} \right)\left(x_0 \right)\right) a^2\left(x_0 \right)  \neq 0$ and $\supp z \subset \mathcal W$. From $\tr\left(z b z \right) \in C_0\left(\mathcal X \right)$ it turns out that there is an open (with respect to topology of $\hat A$) neighborhood $\mathcal U$ of $x_0$ such that $\tr\left( \rho_{x}\left(zbz \right) \right) \neq 0$ for any $x \in \mathcal U$, i.e. $\mathcal V \bigcap \mathcal U = \emptyset$. It follows that $\mathcal V $ is a closed subset with respect to the  topology of $\hat A$. Hence there is a homeomorhism $\hat A \approx \hat B$. Below we apply the method of proof of the Theorem 6.1.11 \cite{pedersen:ca_aut}. Let us consider the set $M$ of elements in $B_+$ with continuous trace, $M$ is hereditary and the closure of $M$ is the positive part of an ideal $J$ of $B$. However for any $x \in \mathcal{ X}= \hat B$ there is an Abelian $a \in K\left( A\right) $ such that $\tr\left( \rho_x\left(a \right) \right) \neq 0$. It turns out $J$ is not contained in any primitive ideal of $B$, hence $J = B$. It turns out $B$ has continuous trace. From this fact it turns out $\rho_x\left(A \right) = \rho_x\left(B \right)\approx \mathcal K$ for any $x \in \mathcal X$. Taking into account $\rho_x\left(A \right) = \rho_x\left(B \right)$, homeomorphism $\hat A \approx \hat B$ and the Theorem \ref{ctr_di_do_thm} one has $B = A$.
\end{proof}

\subsection{Finite-fold coverings}
\paragraph*{}
If $\pi:\widetilde{\mathcal X} \to \mathcal X$ is a finite-fold covering, such that ${\mathcal X}$ and $\widetilde{\mathcal X}$ are  compact Hausdorff  spaces, then there is a finite family $\left\{\mathcal U_\iota\subset  \mathcal X\right\}_{\iota \in I_0}$ of connected open subsets of $ \mathcal X$ evenly covered by $\pi$ such that $ \mathcal X = \bigcup_{\iota \in I_0} \mathcal U_\iota$. There is a  partition of unity subordinated to $\left\{\mathcal U_\iota\right\}$, i.e.
$$
1_{C\left(\mathcal X \right) }= \sum_{\iota \in I_0}a_\iota
$$
where $a_\iota \in C\left(\mathcal X \right)_+$ is such that $\supp a_\iota \subset \mathcal U_\iota$. Denote by $e_\iota = \sqrt{a_\iota}\in C\left(\mathcal X \right)_+$.
For any $\iota \in I_0$ we select $\widetilde{\mathcal U}_\iota \subset \widetilde{\mathcal X}$ such that $\widetilde{\mathcal U}_\iota$ is homemorphically mapped onto $\mathcal U_\iota$. If $\widetilde{e}_\iota \in C\left( \widetilde{\mathcal X}\right) $ is given by
\begin{equation*}
\widetilde{e}_\iota\left(\widetilde{x} \right)  = \left\{
\begin{array}{c l}
e_\iota\left(\pi\left( \widetilde{x}\right) \right)  & \widetilde{x} \in \widetilde{\mathcal U}_\iota \\
0 & \widetilde{x} \notin \widetilde{\mathcal U}_\iota
\end{array}\right.
\end{equation*}
and $G = G\left( \widetilde{\mathcal X}~|{\mathcal X} \right)$ then
\begin{equation*}
\begin{split}
1_{C\left(\widetilde{\mathcal X} \right) }= \sum_{ \left(g, \iota\right)\in G \times I_0}g\widetilde{e}^2_\iota,\\
\widetilde{e}_\iota \left( g\widetilde{e}_\iota\right) =0; \text{ for any nontrivial } g \in G.
\end{split}
\end{equation*}
If $I= G\times I_0$ and $\widetilde{e}_{\left(g, \iota\right)}= g\widetilde{e}_\iota$ the from the above equation it turns out
\begin{equation}\label{ctr_unity_eqn}
\begin{split}
1_{C\left(\widetilde{\mathcal X} \right) }= \sum_{\iota \in I}\widetilde{e}_\iota \left\rangle \right\langle \widetilde{e}_\iota
\end{split}
\end{equation}
where $\widetilde{e}_\iota \left\rangle \right\langle \widetilde{e}_\iota$ means a compact operator induced by the $C^*$-Hilbert module structure given by \eqref{finite_hilb_mod_prod_eqn}.
\begin{proposition}\label{mult_str_pos_prop}\cite{apt_mult}
	If $B$ is a $C^*$-subalgebra of $A$ containing an
	approximate unit for $A$, then $M\left(B \right) \subset M\left( A\right)$  (regarding $B''$ as a subalgebra	of $A''$).
\end{proposition}

\begin{lem}\label{ctr_mult_lem}
	Let $A$ be a continuous trace algebra, and let $\hat A = \mathcal X$ be the spectrum of $A$. Suppose that $\mathcal X$ is a  locally compact second-countable Hausdorff space, and let $\pi: \widetilde{\mathcal X} \to \mathcal X$ be a finite-fold covering. There is the natural *-isomorphism $M\left(A \right) \cong M\left( A\left(  \widetilde{\mathcal X}\right) \right)^G$ of multiplier algebras.
\end{lem}
\begin{proof}
	For any $x \in \mathcal X$ there is an open neighborhood $\mathcal U$ such that $A\left( \mathcal U\right) \cong C_0\left( \mathcal U \right) \otimes \mathcal{K}$. Since $\mathcal X$ is second-countable there is an enumerable family $\left\{\mathcal{U}_k\right\}_{k \in \N}$ such that $A\left( \mathcal U_k\right) \cong C_0\left( \mathcal U_k \right) \otimes \mathcal{K}$ for any $k \in \N$. There is a family $\left\lbrace a_k \in C_0\left( \mathcal X\right)_+  \right\rbrace_{k \in \N}$ such that 
	\begin{itemize}
		\item $\supp a_k \subset  \mathcal U_k$, 
		\item 	
		\be\label{ctr_u_eqn}
		1_{C_b\left( \mathcal{X}\right) } = \sum_{k=0}^\infty a_k
		\ee where sum of the series means the strict convergence (cf. Definition \ref{strict_topology}).
	\end{itemize}

	There is an enumerable family $\left\{e_k\in \mathcal{K} \right\}_{k \in \N}$ of rank-one positive mutually orthogonal operators such that
\be\label{ctr_k_eqn}
	1_{M\left(\mathcal K\right) }= \sum_{k=0}^\infty e_k
\ee
	where above sum assumes strict topology (cf. Definition \ref{strict_topology}).
	The family of products $\left\lbrace u_{jk}= a_j\otimes e_k\right\rbrace_{j,k \in \N} $ is enumerable and let us introduce an enumeration of $\left\lbrace u_{jk} \right\rbrace_{j,k \in \N}$, i.e. $\left\lbrace u_{jk} \right\rbrace_{j,k \in \N}= \left\lbrace u_{p} \right\rbrace_{p \in \N}$. From \eqref{ctr_u_eqn} and \eqref{ctr_k_eqn} it follows that
	\be\label{ctr_q_eqn}
	1_{M\left(A\right) }= \sum_{\substack{j=1 \\ k=1}}^\infty u_{jk}=\sum_{p=0}^\infty u_p.
	\ee
If $h \in A$ is given by
	$$
	h = \sum_{p=0}^\infty \frac{1}{2^p} u_p
	$$
	and $\tau: A \to \C$ is a state such that $\tau\left( h\right) = 0$ then  from $u_p > 0$ for any $p \in \N$ it follows that $\tau\left(u_p \right) = 0$ for any $p \in \N$.
	However from \eqref{ctr_q_eqn} it turns out
	$$
	1 = \tau\left( 1_{M\left(A\right) }\right)= \tau\left(\sum_{p=0}^\infty u_p \right) = \sum_{p=0}^\infty \tau\left( u_p \right),
	$$
	and above equation contradicts with  $\tau\left(u_p \right) = 0$ for any $p \in \N$. It follows that $\tau\left( h\right) \neq 0$	for any state $\tau$, i.e. $h$ is strictly positive element of $A$.
	Similarly one can prove that $h$ is strictly positive element of $A\left(  \widetilde{\mathcal X}\right)$ because
	$$
		1_{M\left(A\left(  \widetilde{\mathcal X}\right)\right) }= \sum_{p=0}^\infty u_p.
	$$
From the Proposition \ref{mult_str_pos_prop} it follows that there is the natural injective *-homomorphism $f: M\left( A\right) \hookto M\left(A\left(  \widetilde{\mathcal X}\right) \right)$. Clearly $g f\left( a\right) = f\left(ga \right)= f\left( a\right) $ for any $a \in A$ and $g \in G$, it follows that $f\left( M\left( A\right)\right) \subset M\left(  \widetilde{A} \right)^G$, or equivalently $M\left( A\right) \subset M\left(  \widetilde{A} \right)^G$. Otherwise from the Lemma  \ref{ind_mult_inv_lem} one has $ M\left(  \widetilde{A} \right)^G\subset M\left( A\right)$. Taking into account mutually inverse inclusions $ M\left(  \widetilde{A} \right)^G\subset M\left( A\right)$ and $M\left( A\right) \subset M\left(  \widetilde{A} \right)^G$ we conclude that
	$$
	M\left(A \right) \cong M\left( A\left(  \widetilde{\mathcal X}\right) \right)^G.
	$$
\end{proof}

\begin{lem}\label{ctr_fin_lem}
	Let $A$ be a continuous trace algebra, and let $\hat A = \mathcal X$ be the spectrum of $A$. Suppose that $\mathcal X$ is a  locally compact second-countable Hausdorff space, and let $\pi: \widetilde{\mathcal X} \to \mathcal X$ be a finite-fold covering with compactification. Then the triple
	 $\left( A, A\left(\widetilde{\mathcal X} \right),G = G\left(\widetilde{\mathcal X}~|~\mathcal X \right)\right)$ is a finite-fold noncommutative covering with compactification.
\end{lem}
\begin{proof}
We need check conditions (a) - (c) of the Definition \ref{fin_comp_def}.\\
(a)
	There is the action of $G$ on $A\left(\widetilde{\mathcal X} \right)$ induced by the action of $G$ on $\widetilde{\mathcal X}$.
	From \eqref{ctr_inv_eqn} it turns out that  $A = A\left(\widetilde{\mathcal X} \right)^G$ and there is an injective *-homomorphism $A \hookto A\left(\widetilde{\mathcal X} \right)$. 
	Denote by $M\left(A \right)$ and  $M\left( A\left(\widetilde{\mathcal X}\right)  \right)$ multiplier algebras of $A$ and $A\left(\widetilde{\mathcal X} \right)$. 
	Denote by $\mathcal X \hookto \mathcal Y$, $\widetilde{\mathcal X}\hookto\widetilde{\mathcal Y}$ compactifications such that $\widetilde{   \pi}:\widetilde{\mathcal Y}\to \mathcal Y$ is a (topological) finite covering and  $\pi= \widetilde{   \pi}|_{\widetilde{\mathcal X}}$. From $C_b\left(\mathcal X \right)\subset  M\left( A\right)$, $C_b\left(\widetilde{\mathcal X} \right) \subset M\left( A\left(\widetilde{\mathcal X} \right)\right)$ and $C\left(\mathcal Y \right) \subset C_b\left(\mathcal X \right)$, $C\left(\widetilde{\mathcal Y} \right)\subset C_b\left(\widetilde{\mathcal X} \right)$ it follows that  $C\left(\mathcal Y \right) \subset M\left(A \right) $, $C\left(\widetilde{\mathcal Y} \right)\subset M\left( A\left(\widetilde{\mathcal X} \right)\right) $.
 If $B = C\left(\mathcal Y \right)  M\left( A\right)$ and $\widetilde{B} =  C\left(\widetilde {\mathcal Y} \right)    M\left( A\left(\widetilde{\mathcal X} \right) \right)$ then $A$ (resp. $A\left(\widetilde{\mathcal X} \right)$) is an essential ideal of $B$ (resp. $\widetilde{B}$).  Clearly $A = B \bigcap A\left( \widetilde{\mathcal X}\right)$. \\
 (b)
  Since $G\widetilde {\mathcal Y}= \widetilde {\mathcal Y}$ the action $G\times M\left( A\left(\widetilde{\mathcal X} \right)\right)\to M\left( A\left(\widetilde{\mathcal X} \right)\right)$ induces an action $G \times \widetilde B\to\widetilde B$. From the Lemma \ref{ctr_mult_lem} on has  the natural *-isomorphism $M\left( A\left(\widetilde{\mathcal X} \right)\right)^G \cong M\left( A\right)$. It follows that $B = C\left(\mathcal Y \right)  M\left( A\right)\cong C\left(\mathcal Y \right) M\left( A\left(\widetilde{\mathcal X} \right)\right)^G=\widetilde{B}^G$. From \eqref{ctr_unity_eqn} it turns out that there is a finite family $\left\{e_\iota\in C\left( \mathcal Y\right) \right\}_{\iota \in I}$ such that
 $$
 1_{C\left(\widetilde{\mathcal Y} \right) }=1_{\widetilde{B}  }= \sum_{\iota \in I}\widetilde{e}_\iota \left\rangle \right\langle \widetilde{e}_\iota.
 $$
 It turns out that any $\widetilde{b} \in \widetilde{B}$ is given by
 \begin{equation*}
 \begin{split}
\widetilde{b} = \sum_{\iota \in I}\widetilde{e}_\iota b_\iota~, \\ 
b_\iota = \left\langle \widetilde{b}, \widetilde{e}_\iota \right\rangle_{\widetilde{B}} \in B,
 \end{split}
 \end{equation*} 
 i.e. $\widetilde{B}$ is a finitely generated (by $\left\{e_\iota \right\}_{\iota \in I}$ ) right $B$-module. From the Kasparov Stabilization Theorem \cite{blackadar:ko} it turns out that $\widetilde{B}$ is a projective $B$ module. So $\left(B, \widetilde {B},G \right)$ is an unital finite-fold noncommutative covering. \\
 (c) Follows from $G \widetilde{\mathcal X}= \widetilde{\mathcal X}$.
\end{proof}
\begin{thm}\label{ctr_fin_thm}
	Let $A$ be a continuous trace algebra, and let $\hat A = \mathcal X$ be the spectrum of $A$. Suppose that $\mathcal X$ is a  locally compact second-countable Hausdorff space, and let $\pi: \widetilde{\mathcal X} \to \mathcal X$ be a finite-fold covering. Then the triple
	$\left( A, A\left(\widetilde{\mathcal X} \right),G = G\left(\widetilde{\mathcal X}~|~\mathcal X \right)\right)$ is a finite-fold noncommutative covering.
\end{thm}
\begin{proof}
	We need check (a), (b) of the Definition \ref{fin_def}.\\
	(a) Follows from $\mathcal X = \widetilde{\mathcal X}/G$.
	\\
	(b)
Let us consider a family $\left\{ \mathcal U_\la \subset \mathcal X\right\}_{\la \in \La}$ of open sets such that
\begin{itemize}
\item $\mathcal X = \bigcup_{\la \in \La} \mathcal U_\la$,
\item The closure $\overline{ \mathcal U}_\la$ of  $\mathcal U_\la$ in $\mathcal X$ is compact $\forall \la \in \La$.
\end{itemize}
Clearly $\pi^{-1}\left(\overline{ \mathcal U}_\la \right)\to \overline{ \mathcal U}_\la$ is a covering, so  $\pi^{-1}\left({ \mathcal U}_\la \right)\to { \mathcal U}_\la$ is a covering with compactification.
If  $\widetilde{I}_\la \stackrel{\mathrm{def}}{=}A\left( \pi^{-1}\left( {\mathcal U}_\la\right)\right)  \subset A\left( \widetilde{\mathcal X}\right)$ and $I_\la = \widetilde{I}_\la \bigcap A$ then from  $G\pi^{-1}\left( \mathcal U_\la\right) = \pi^{-1}\left( \mathcal U_\la\right)$ it follows that
\bean
G\widetilde{I}_\la = \widetilde{I}_\la,\\
I_\la = A\left({\mathcal U}_\la \right), 
\eean
i.e. $\widetilde{I}_\la$ satisfies to \eqref{gi-i}.
 From the Lemma \ref{ctr_fin_lem} it follows that there is a finite-fold noncommutative covering with compactification 
$
\left(I_\la,  \widetilde{I}_\la, G \right) =\left(A\left( \mathcal U_\la\right), A\left( \pi^{-1} \left(\mathcal U_\la \right)\right),G \right)
$. From the Definition \ref{ctr_cov_defn} and  $\mathcal X = \bigcup_{\la \in \La} \mathcal U_\la$ it follows that $\bigcup_{\la \in \La}  I_\la = \bigcup_{\la \in \La} A\left( \mathcal U_\la\right)$ (resp. $\bigcup_{\la \in \La}  \widetilde{I}_\la = \bigcup_{\la \in \La} A\left( \pi^{-1}\left( \mathcal U_\la\right) \right)$) is dense in $A$ (resp. $A\left( \widetilde{\mathcal X}\right)$.
\end{proof}

\subsection{Infinite coverings}\label{ctr_case_sec}
\paragraph*{}
Let $A$ be a continuous trace $C^*$-algebra such that the spectrum $\hat A= \mathcal X$ of is a second-countable locally compact Hausdorff space.  Suppose that
\begin{equation*}
\mathfrak{S}_\mathcal{X} = \left\{\mathcal{X}=\mathcal{X}_0 \xleftarrow{}... \xleftarrow{} \mathcal{X}_n \xleftarrow{} ... \right\}
\end{equation*}

is  a topological  finite covering sequence. From the Theorem \ref{ctr_fin_thm} it turns out that
$$\mathfrak{S}_{A\left( \mathcal{X}\right)}=\left\{A = A\left( \mathcal{X}_0\right)\to ... \to A\left( \mathcal{X}_n\right) \to ...\right\} $$ is an algebraical  finite covering sequence.
If $\widehat{A} = \varinjlim A\left( \mathcal{X}_n\right)$ then from the Theorem \ref{direct_lim_state_thm} it follows that there is the spectrum of $\widehat{A}$ is homeomorphic to  $\widehat{  \mathcal X}= \varprojlim \mathcal X_n$.
\begin{lem}
	If $\mathfrak{S}_\mathcal{X} = \left\{\mathcal{X}=\mathcal{X}_0 \xleftarrow{}... \xleftarrow{} \mathcal{X}_n \xleftarrow{} ... \right\} \in \mathfrak{FinTop}$ and  $\overline{\mathcal X}$ is disconnected inverse limit of $\mathfrak{S}_\mathcal{X}$, then there is the natural inclusion of $\widehat{A}'' \to A\left(\overline{\mathcal X} \right)'' $ of von Neumann enveloping algebras.
\end{lem}
\begin{proof}
Surjective maps $\overline{\mathcal X} \to \mathcal{X}_n $ give injective *-homomorphisms $A\left( \mathcal{X}_n\right) \hookto M\left(A\left( \overline{\mathcal X}\right)  \right)$, which induce the injective *-homomorphism $\widehat{A} \hookto   M\left(A\left( \overline{\mathcal X}\right)  \right)$. It turns out the injective *-homomorphism of von Neumann enveloping algebras $\widehat{A}'' \to A\left(\overline{\mathcal X} \right)'' $.
\end{proof}


\begin{empt}
	Denote by $G_n = G\left(\mathcal X_n~|~\mathcal X \right)$ groups of covering transformations and $\widehat{G} = \varprojlim G_n$. Denote by $\overline{\pi}_n:\overline{  \mathcal X} \to \mathcal X_n$, $\pi^n:  \mathcal X_n \to \mathcal X$ $\pi^m_n:  \mathcal X_m \to \mathcal X_n$ ($m > n$) the natural covering projections.
\end{empt}
\begin{lem}\label{ctr_1_lem}
	If $\overline{  \mathcal U }\subset \overline{  \mathcal X }$ is an open subset mapped homeomorphically onto $\mathcal U \subset \mathcal X$ then any positive element in $\overline{a}\in A\left( \overline{  \mathcal U }\right)_+\subset A\left( \overline{  \mathcal X}\right)_+ $ is special.
\end{lem}
\begin{proof}
	If $\mathcal{U}_n = \overline{\pi}_n\left(\overline{  \mathcal U } \right)$ then there is a *-isomorphism $\overline{\varphi }_n: A\left( \overline{  \mathcal U }\right)\xrightarrow{\approx} A\left( \mathcal{U}_n\right) $.
	 For any $n \in \mathbb{N}^0$ and $z \in A$ and $f_\eps$ given by \eqref{f_eps_eqn} 
\begin{equation*}
\begin{split}
a_n = \sum_{g \in \ker\left( \widehat{G} \to  G_n\right)} g  \overline{a} = \overline{\varphi }_n\left( \overline{a}\right),  \\
b_n = \sum_{g \in \ker\left( \widehat{G} \to  G_n\right)} g \left(z  \overline{a} z^*\right) = z\overline{\varphi }_n\left( \overline{a}\right)z^*,\\
c_n = \sum_{g \in \ker\left( \widehat{G} \to  G_n\right)} g \left(z  \overline{a} z^*\right)^2= \left(z\overline{\varphi }_n\left( \overline{a}\right)z^*\right)^2 ,\\
d_n = \sum_{g \in \ker\left( \widehat{G} \to  G_n\right)} g f_\eps\left(z  \overline{a} z^*\right)=f_\eps\left(z\overline{\varphi }_n\left( \overline{a}\right)z^*\right).
\end{split}
\end{equation*}

From the above equations it follows that $a_n,~ b_n,~ c_n,~ d_n \in A\left(\mathcal X_n \right)$ and $b_n^2 = c_n$, 
i.e. $\overline{   a}$ satisfies to the Definition \ref{special_el_defn}.

\end{proof}
\begin{cor}\label{ctr_c_cor}
	If	$\overline{A}$ is the disconnected  inverse noncommutative limit of $\mathfrak{S}_{A\left( \mathcal{X}\right)}$, then $ A\left(\overline{  \mathcal X} \right) \subset \overline{A}$.
\end{cor}
\begin{proof}
	From the Lemma \ref{ctr_1_lem} it turns out $A\left(\overline{  \mathcal U} \right) \subset \overline{A}$. However $A\left(\overline{  \mathcal X} \right)$ is the $C^*$-norm completion of its subalgebras $A\left(\overline{  \mathcal U} \right) \subset A\left(\overline{  \mathcal X} \right)$.
\end{proof}
\begin{lem}\label{ctr_lem1}
	If $\overline{a} \in A\left(\overline{   \mathcal X }\right)''$ is a special element and  $z \in K\left(A \right)$ is an Abelian element then 	$\overline{b}=	\tr\left(z\overline{a}z \right) \in C_0\left(\overline{\mathcal X} \right)$.

\end{lem}
\begin{proof}
	Any Abelian element is positive, hence $z = z^*$. If $f_\eps$ is given by \eqref{f_eps_eqn} and $\overline{b}'= z\overline{a}z$ then from (b) of the Definition \ref{special_el_defn} it turns out 
	\begin{equation*}
\begin{split}
b'_n =  \sum_{g \in \ker\left( \widehat{G} \to  G_n\right)}  g\overline{b}' \in A\left(\mathcal X_n \right),\\ 
c'_n =  \sum_{g \in \ker\left( \widehat{G} \to  G_n\right)} g\overline{b}'^2 \in A\left(\mathcal X_n \right),\\
d'_n =  \sum_{g \in \ker\left( \widehat{G} \to  G_n\right)} gf_\eps\left( \overline{b}'\right) \in A\left(\mathcal X_n \right).
\end{split}
\end{equation*}
From $z \in K\left(A \right)$ it turns out that $\supp \tr\left(z \right)$ is compact. The map $\pi_n :\mathcal X_n \to \mathcal X$ is a finite-fold covering, it turns out $\pi^{-1}_n\left(\supp \tr\left(z \right) \right)$ is compact. If $\supp b'_n ,~ \supp c'_n, ~\supp d'_n  \subset \pi^{-1}_n\left(\supp \tr\left(z \right) \right)$ it turns out that all sets  $\supp b'_n ,~ \supp c'_n,~ \supp d'_n$ are compact. Taking into account that all $b'_n$, $c'_n$, $d'_n$ are Abelian one has $b'_n,~c'_n, ~d'_n \in K \left( A\left(\mathcal X_n \right)\right) $ where $K \left( A\left(\mathcal X_n \right)\right)$ means the Pedersen ideal of $A\left(\mathcal X_n \right)$. It follows that
\begin{equation*}
\begin{split}
	b_n =  \tr\left(b'_n\right)=\sum_{g \in \ker\left( \widehat{G} \to  G_n\right)}  \tr\left(g \overline{b}') \right)\in C_c\left(\mathcal X_n\right),\\
		b^2_n =  \tr\left(b'^2_n\right)=\tr\left(b'_n\right)^2 \in C_c\left(\mathcal X_n\right),\\
	c_n= \tr\left(c'_n\right) =\sum_{g \in \ker\left( \widehat{G} \to  G_n\right)}  \tr\left(g \overline{b}'^2\right)=\sum_{g \in \ker\left( \widehat{G} \to  G_n\right)}  \tr\left(g \overline{b}'\right)^2\in C_c\left(\mathcal X_n\right),\\
	d_n= \tr\left(d'_n\right) =\sum_{g \in \ker\left( \widehat{G} \to  G_n\right)}  \tr\left(g f_\eps\left( \overline{b}'\right) \right)=\sum_{g \in \ker\left( \widehat{G} \to  G_n\right)} f_\eps\left(  \tr\left(g \overline{b}'\right)\right) \in C_c\left(\mathcal X_n\right).\\
\end{split}
\end{equation*}
From the above equations it follows that $b_n,~c_n,~ d_n$ satisfy to the condition (a) of the Lemma  \ref{comm_main_lem}.
From the condition (c) the Definition \ref{special_el_defn} it follows that for any for any $\eps > 0$ there is $N \in \N$ such that for any $n \ge N$ following condition holds
\begin{equation}\label{ctr_ineq_eqn}
\left\|b'^2_n -c'_n\right\| < \eps.
\end{equation}

	Both $b'_n$ and $c'_n$ are Abelian and the range projection of $b'_n$ equals to the range projection of $c'_n$, i.e. $\left[b'_n \right] = \left[c'_n \right]$,  it turns out 	
	\begin{equation*}
	\begin{split}
\left\|b^2_n-c_n\right\|=\left\|\tr\left(b'\right)^2 -\tr\left(c'_n\right) \right\|= \left\|b'^2_n-c'_n\right\|.
	\end{split}
	\end{equation*}
From \eqref{ctr_ineq_eqn} it follows that $\left\|b^2_n-c_n\right\| < \eps$ for any $n \ge N$. It means that $b_n$ and $c_n$ satisfy to condition (b) of the Lemma  \ref{comm_main_lem}.
	From the Lemma \ref{comm_main_lem} it turns out that $\overline{b}=	\tr\left(z\overline{a}z \right) \in C_0\left(\overline{\mathcal X} \right)$.
\end{proof}
\begin{lem}\label{ctr_top_lem}
If	$\overline{A}$ is the disconnected  inverse noncommutative limit of $\downarrow\mathfrak{S}_{A\left( \mathcal{X}\right)}$, then $\overline{A} = A\left(\overline{  \mathcal X} \right)$.
\end{lem}
\begin{proof}
From the Corollary \ref{ctr_c_cor} it follows that $A\left(\overline{  \mathcal X} \right) \subset \overline{A}$. From the Corollary \ref{special_cor} it follows that
\begin{equation}\label{ctr_inc_eqn}
A\left(\overline{  \mathcal X} \right) \subset \overline{A} \subset A\left(\overline{  \mathcal X} \right)''.
\end{equation}
Let $\overline{\pi}: \overline{  \mathcal X} \to {  \mathcal X}$ and let $\overline{a} \in A\left(\overline{  \mathcal X} \right)_+''$. Let $\overline{x} \in \overline{\mathcal X}$ be such that $\rho_{\overline{x}}\left( \overline{a}\right) \neq 0$  and let $\overline{  \mathcal W}$ be an open neighborhood of $x$ such that $\overline{\pi}$ homeomorphically maps $\overline{  \mathcal W}$ onto $\mathcal W = \overline{\pi}\left( \overline{  \mathcal W}\right)$.
If $\overline{z} \in K\left( A\left(\overline{  \mathcal X} \right)\right)$ is an Abelian element such that $\supp  \overline{z} \subset \overline{  \mathcal W}$ and $\rho_{\overline{x}}\left(\overline{z}~\overline{a}~\overline{z} \right)\neq 0$ then the element $z = \sum_{g \in \widehat{G}}g \overline{z}\in A$  is Abelian and $\supp z \in \mathcal W$. If $\overline{a}$ is special, then from the Lemma \ref{ctr_lem1} it turns out that 
$$
\tr\left( z\overline{a}z \right) \in C_0\left( \overline{\mathcal X}\right) .
$$
However from
$$
\rho_{\overline{x}}\left( \overline{z}~\overline{a}~\overline{z}\right) \left( \overline{a}\right) = \left\{
\begin{array}{c l}
\rho_{\overline{x}}\left( z\overline{a}z\right)  & \overline{x} \in \overline{   \mathcal W} \\
0 & \overline{x} \notin \overline{   \mathcal W}
\end{array}\right.
$$
it turns out
\begin{equation}\label{ctr_zaz_eqn}
\begin{split}
\tr\left( \overline{z}~\overline{a}~\overline{z} \right) \in C_0\left( \overline{\mathcal X}\right),\\
\tr\left( \overline{z}~\overline{a}~\overline{z} \right)\left(x \right) \neq 0.
\end{split}
\end{equation}
The set of special elements is dense in $\overline{A}_+$ it turns out that any $\overline{a} \in \overline{A}_+$ satisfies to \eqref{ctr_zaz_eqn}. Taking into account this fact and \eqref{ctr_inc_eqn} it turns out
	\begin{itemize}
	\item $A\left(\overline{  \mathcal X} \right) \subset \overline{A} \subset A\left(\overline{  \mathcal X} \right)''$,
	\item For any $\overline{a} \in \overline{A}_+$ and $x \in \overline{  \mathcal X} $ such that $\rho_{\overline{x}}\left(\overline{a} \right)\neq 0$ there is an open neighborhood  $\overline{  \mathcal W} \subset \overline{  \mathcal X}$ of $x$ and an Abelian $\overline{z} \in A\left(\overline{  \mathcal X} \right)$ such that
	\begin{equation*}
\begin{split}
\supp \overline{z} \subset \overline{\mathcal W},\\
\tr\left(\overline{z}~\overline{a}~\overline{z} \right) \in C_0\left(\overline{\mathcal X} \right),\\
\tr\left(\overline{z}~\overline{a}~\overline{z} \right) \left(x\right) \neq 0.
\end{split}
\end{equation*}
\end{itemize}
 From the Lemma \ref{ctr_adm_lem} it follows that $\overline{A}= A\left(\overline{  \mathcal X} \right)$.
\end{proof}

\begin{empt}\label{ctr_transitive_constr}
	Let $\widetilde{\mathcal X} \subset \overline{\mathcal X}$ be a connected component and let $G \subset G\left(\varprojlim C_0\left(\mathcal X_n \right) ~|~C_0\left(  \mathcal X\right)   \right)$ be maximal subgroup among subgroups $G' \subset G\left(\varprojlim C_0\left(\mathcal X_n \right) ~|~C_0\left(  \mathcal X\right)   \right)$ such that $G'\widetilde{\mathcal X} = \widetilde{\mathcal X}$. If $J\subset\widehat{G}$ is a set of representatives of $\widehat{G}/G$ then  from the \eqref{top_disconnected_repr_eqn} it follows that
	\begin{equation*}
	\overline{\mathcal X}=  \bigsqcup_{g \in J} g \widetilde{\mathcal X}.
	\end{equation*}
	and  the algebraic direct sum
	\begin{equation}\label{ctr_transitive_eqn}
	 \bigoplus _{g \in J}  A\left( g\widetilde{\mathcal X} \right) \subset A\left( \overline{\mathcal X}\right).
	\end{equation}
is a dense subalgebra of $A\left( \overline{\mathcal X}\right)$.

\end{empt}

\begin{thm}\label{ctr_main_thm}
	Let $A$ be $C^*$-algebra of continuous trace, and let $\mathcal{X}$ be the spectrum of $A$.
	Let 
	$$\mathfrak{S}_{\mathcal X} = \left\{\mathcal{X} = \mathcal{X}_0 \xleftarrow{}... \xleftarrow{} \mathcal{X}_n \xleftarrow{} ...\right\} \in \mathfrak{FinTop}$$ 
	be a topological  finite covering sequence,	and let
	$$\mathfrak{S}_{A\left( \mathcal{X}\right)}=\left\{A = A\left( \mathcal{X}_0\right)\to ... \to A\left( \mathcal{X}_n\right) \to ...\right\} \in \mathfrak{FinAlg}$$ be an algebraical   finite covering sequence. Following conditions hold:
	\begin{enumerate}
		\item [(i)] $\mathfrak{S}_{A\left(\mathcal{X}\right)}$ is good,
		\item[(ii)] There are  isomorphisms:

		\begin{itemize}
			\item $\varprojlim \downarrow \mathfrak{S}_{A\left(\mathcal{X}\right)} \approx A\left(\varprojlim \downarrow \mathfrak{S}_{\mathcal X}\right)$,
			\item $G\left(\varprojlim \downarrow \mathfrak{S}_{A\left(\mathcal{X}\right)}~|~ A\right) \approx G\left(\varprojlim \downarrow \mathfrak{S}_{\mathcal{X}}~|~ \mathcal X\right)$.
		\end{itemize}
	\end{enumerate}
	
\end{thm}
\begin{proof}
Similar to the proof ot the Theorem \ref{comm_main_thm}.
\end{proof}
\section{Noncommutative tori and their coverings}  
   \subsection{Fourier transformation}
   \paragraph*{}
        	There is a norm on $\mathbb{Z}^n$ given by
   \begin{equation}\label{mp_znorm_eqn}
   \left\|\left(k_1, ..., k_n\right)\right\|= \sqrt{k_1^2 + ... + k^2_n}.
   \end{equation}
 The space of complex-valued Schwartz  functions on $\Z^n$ is given by 
 \begin{equation*}
 \sS\left(\mathbb{Z}^n\right)= \left\{a = \left\{a_k\right\}_{k \in \mathbb{Z}^n} \in \mathbb{C}^{\mathbb{Z}^n}~|~ \mathrm{sup}_{k \in \mathbb{Z}^n}\left(1 + \|k\|\right)^s \left|a_k\right| < \infty, ~ \forall s \in \mathbb{N} \right\}.
 \end{equation*}

 Let $\mathbb{T}^n$ be an ordinary $n$-torus. We will often use real coordinates for $\mathbb{T}^n$, that is, view $\mathbb{T}^n$ as $\mathbb{R}^n / \mathbb{Z}^n$. Let $\Coo\left(\mathbb{T}^n\right)$ be an algebra of infinitely differentiable complex-valued functions on $\mathbb{T}^n$. 
 There is the bijective Fourier transformations  $\mathcal{F}_\T:\Coo\left(\mathbb{T}^n\right)\xrightarrow{\approx}\sS\left(\mathbb{Z}^n\right)$;  $f \mapsto \widehat{f}$ given by
 \begin{equation}\label{nt_fourier}
 \widehat{f}\left(p\right)= \mathcal F_\T (f) (p)= \int_{\mathbb{T}^n}e^{- 2\pi i x \cdot p}f\left(x\right)dx
 \end{equation}
 where $dx$ is induced by the Lebesgue measure on $\mathbb{R}^n$ and   $\cdot$ is the  scalar
 product on the Euclidean space $\R^n$.
 The Fourier transformation carries multiplication to convolution, i.e.
 \begin{equation*}
 \widehat{fg}\left(p\right) = \sum_{r +s = p}\widehat{f}\left(r\right)\widehat{g}\left(s\right).
 \end{equation*}
 The inverse Fourier transformation $\mathcal{F}^{-1}_\T:\sS\left(\mathbb{Z}^n\right)\xrightarrow{\approx} \Coo\left(\mathbb{T}^n\right)$;  $ \widehat{f}\mapsto f$ is given by
 $$
 f\left(x \right) =\mathcal{F}^{-1}_\T \widehat f\left( x\right)  = \sum_{p \in \Z^n} \widehat f\left( p\right)   e^{ 2\pi i x \cdot p}.
 $$
 There  is the $\C$-valued scalar product  on $\Coo\left( \T^n\right)$ given by
 $$
 \left(f, g \right) = \int_{\T^n}fg dx =\sum_{p \in \Z^n}\widehat{f}\left( -p\right) \widehat{g}\left(p \right).  
 $$ 
 Denote by $\SS\left( \R^{n}\right) $ be the space of
 complex Schwartz (smooth, rapidly decreasing) functions on $\R^{n}$. 
 \be\label{mp_sr_eqn}
 \begin{split}
 	\SS\left(\mathbb {R} ^{n}\right)=\left\{f\in C^{\infty }(\mathbb {R} ^{n}):\|f\|_{\alpha  ,\beta )}<\infty \quad \forall \alpha =\left( \al_1,...\al_n\right) ,\beta =\left( \bt_1,...\bt_n\right)\in \mathbb {Z} _{+}^{n}\right\},\\
 	\|f\|_{{\alpha ,\beta }}=\sup_{{x\in {\mathbb  {R}}^{n}}}\left|x^{\alpha }D^{\beta }f(x)\right|
 \end{split}
 \ee
 where 
 \bean
 x^\al = x_1^{\al_1}\cdot...\cdot x_n^{\al_n},\\
D^{\beta} = \frac{\partial}{\partial x_1^{\bt_1}}~...~\frac{\partial}{\partial x_n^{\bt_n}}.
 \eean
 The topology on $\SS\left(\mathbb {R} ^{n}\right)$ is given by seminorms $\|\cdot\|_{{\alpha ,\beta }}$.
 \begin{defn}\label{nt_*w_defn}
 	Denote by $\SS'\left( \R^{n}\right) $ the vector space dual to $\SS\left( \R^{n}\right) $, i.e. the space of continuous functionals on $\SS\left( \R^{n}\right)$. Denote by $\left\langle\cdot, \cdot \right\rangle:\SS'\left( \R^{n}\right)\times	\SS\left( \R^{n}\right)\to\C$ the natural pairing. We say that $\left\{a_n \in \SS'\left(\mathbb {R} ^{n}\right)\right\}_{n \in \N}$ is \textit{weakly-* convergent} to $a \in \SS'\left(\mathbb {R} ^{n}\right)$ if for any $b \in  \SS\left(\mathbb {R} ^{n}\right)$ following condition holds
 	$$
 	\lim_{n \to \infty}\left\langle a_n, b \right\rangle = \left\langle a, b \right\rangle.	
 	$$
 	We say that
 	$$
 	a = \lim_{n\to \infty}a_n
 	$$
 	in the \textit{sense of weak-* convergence}.
 \end{defn}

 Let $\mathcal F$ and $\mathcal F^{-1}$ be the ordinary and inverse Fourier transformations given by
 \begin{equation}\label{intro_fourier}
 \begin{split}
 \left(\mathcal{F}f\right)(u) = \int_{\mathbb{R}^{2N}} f(t)e^{-2\pi it\cdot u}dt,~\left(\mathcal F^{-1}f\right)(u)=\int_{\mathbb{R}^{2N}} f(t)e^{2\pi it\cdot u}dt 
 \end{split}
 \end{equation}
which satisfy  following conditions
 $$
 \mathcal{F}\circ\mathcal{F}^{-1}|_{\SS\left( \R^{n}\right)} = \mathcal{F}^{-1}\circ\mathcal{F}|_{\SS\left( \R^{n}\right)} = \Id_{\SS\left( \R^{n}\right)}.
 $$
 There is the $\C$-valued scalar product  on $\SS\left( \R^n\right)$ given by
 \begin{equation}\label{fourier_scalar_product_eqn}
 \left(f, g \right)_{L^2\left( \R^n\right) } = \int_{\R^n}fg dx =\int_{\R^n}\mathcal{F}f\mathcal{F}g dx. 
 \end{equation}
 which if $\mathcal{F}$-invariant, i.e.
\be\label{mp_inv_eqn}
 \left(f, g \right)_{L^2\left( \R^n\right) } = \left(\mathcal{F}f, \mathcal{F}g \right)_{L^2\left( \R^n\right) }.
\ee
\paragraph*{}
 There is the action of $\Z^n$ on $\R^n$ such that
$$
g x = x + g; ~ x \in \R^n,~ g \in \Z^n
$$
and $\T^n \approx \R^n / \Z^n$. For any $x \in \R^n$ and $C \in \R$ the series
$$
\sum_{k \in \Z^n}\frac{C}{1 + \left|x + k \right|^{n + 1} }
$$
is convergent, and taking into account \eqref{mp_sr_eqn} one concludes that for $f \in \SS\left(
\R^n \right)$ and $x \in \R^n$
the series
$$
\sum_{g \in \Z^n} D^{\bt}f\left(x + g \right) \left(x \right) = \sum_{g \in \Z^n}\left(gD^{\bt}f \right) \left(x \right) 
$$ 
is absolutely convergent. It follows that the series 
$$
\widetilde{h} = \sum_{g \in \Z^n} g f 
$$
is point-wise convergent and $\widetilde{h}$ is a smooth $\Z^n$ - invariant function. The periodic smooth function $\widetilde{h}$ corresponds to an element of $h \in \Coo\left(\T^n \right)$. This construction provides a map
\begin{equation}\label{mp_sooth_sum_eqn}
\begin{split}
\SS\left(\R^n\right) \to  \Coo\left(\T^n\right), \\
f \mapsto h = \sum_{g \in \Z^n} g f.
\end{split}
\end{equation}
If $\mathcal U = \left(0,1 \right)^n \subset \R^n$ is a fundamental domain of the action of $\Z^n$ on $\R^n$ then $\widetilde{h}_{\mathcal U }$ can be represented by the Fourier series
\bean
\widetilde{h}_{\mathcal U } \left(x \right) = \sum_{p \in \Z^n}c_p e^{2\pi ip x},\\
c_p =\int_{\mathcal U} \widetilde{h}\left( x\right)e^{-2\pi i px}~dx =  \sum_{g \in \Z^n} \int_{\mathcal U}f\left( x+g\right)e^{-2\pi i px}~dx = \int_{\R^n}f\left( x\right)e^{-2\pi i px}~dx=\widehat f\left( p\right) 
\eean
where $\widehat f = \mathcal F f$ is the Fourier transformation of $f$. So if  $\widehat h = \mathcal F_{\T} h$ is the Fourier transformation of $h$ then for any $p \in \Z^n$ a following condition holds
\begin{equation}\label{fourier_from_r_to_z_eqn}
\widehat h\left(p\right) = \widehat f\left( p\right). 
\end{equation}

\subsection{Noncommutative torus $\mathbb{T}^n_{\Theta}$}\label{nt_descr_subsec}

     \paragraph*{}

   Denote by $\cdot: \R^n \times \R^n \to \R$ the scalar product on the Euclidean vector space $\R^n$.  Let $\Theta$ be a real skew-symmetric $n \times n$ matrix, we will define a new noncommutative product $\star_{\Theta}$ on $\sS\left(\mathbb{Z}^n\right)$ given by
     \begin{equation}\label{nt_product_defn_eqn}
     \left(\widehat{f}\star_{\Theta}\widehat{g}\right)\left(p\right)= \sum_{r + s = p} \widehat{f}\left(r\right)\widehat{g}\left(s\right) e^{-\pi ir ~\cdot~ \Theta s}.
     \end{equation}
         and an involution
     \begin{equation*}
     \widehat{f}^*\left(p\right)=\overline{\widehat{f}}\left(-p\right)).
     \end{equation*}
     In result there is an involutive algebra $\Coo\left(\mathbb{T}^n_{\Theta}\right) =\left(\sS\left(\mathbb{Z}^n\right), + , \star_{\Theta}~, ^* \right)$. 
       There is a tracial  state on $\Coo\left(\mathbb{T}^n_{\Theta}\right)$ given by
     \begin{equation}\label{nt_state_eqn}
     \tau\left(f\right)= \widehat{f}\left(0\right).
     \end{equation}
       From  $\Coo\left(\mathbb{T}^n_{\Theta} \right) \approx \SS\left( \Z^n\right)$ it follows  that there is a $\C$-linear isomorphism 
           \begin{equation}\label{nt_varphi_inf_eqn}
           \varphi_\infty: \Coo\left(\mathbb{T}^n_{\Theta} \right) \xrightarrow{\approx}  \Coo\left(\mathbb{T}^n \right).
           \end{equation} 
           such that following condition holds
     \begin{equation}\label{nt_state_integ_eqn}
            \tau\left(f \right)=  \frac{1}{\left( 2\pi\right)^n }\int_{\mathbb{T}^n} \varphi_\infty\left( f\right) ~dx.
      \end{equation}
           
     Similarly to \ref{comm_gns_constr} there is the Hilbert space $L^2\left(\Coo\left(\mathbb{T}^n_{\Theta}\right), \tau\right)$ and the natural representation $\Coo\left(\mathbb{T}^n_{\Theta}\right) \to B\left(L^2\left(\Coo\left(\mathbb{T}^n_{\Theta}\right), \tau\right)\right)$ which induces the $C^*$-norm. The $C^*$-norm completion  $C\left(\mathbb{T}^n_{\Theta}\right)$ of $\Coo\left(\mathbb{T}^n_{\Theta}\right)$ is a $C^*$-algebra and there is a faithful representation
     \begin{equation}\label{nt_repr_eqn}
 C\left(\mathbb{T}^n_{\Theta}\right) \to B\left( L^2\left(\Coo\left(\mathbb{T}^n_{\Theta}\right), \tau\right)\right) .
     \end{equation}
        We will write $L^2\left(C\left(\mathbb{T}^n_{\Theta}\right), \tau\right)$ instead of $L^2\left(\Coo\left(\mathbb{T}^n_{\Theta}\right), \tau\right)$. There is the natural $\C$-linear map  $\Coo\left(\mathbb{T}^n_{\Theta}\right) \to L^2\left(C\left(\mathbb{T}^n_{\Theta}\right), \tau\right) $  and since $\Coo\left(\mathbb{T}^n_{\Theta}\right) \approx \sS\left( \mathbb{Z}^n \right)$ there is a linear map $\Psi_\Th:\sS\left( \mathbb{Z}^n \right) \to L^2\left(C\left(\mathbb{T}^n_{\Theta}\right), \tau\right) $. If $k \in \mathbb{Z}^n$ and $U_k \in  \sS\left( \mathbb{Z}^n \right)=\Coo\left(\mathbb{T}^n_{\Theta}\right)$ is such that 
     \begin{equation}\label{unitaty_nt_eqn}
     U_k\left( p\right)= \delta_{kp}: ~ \forall p \in \mathbb{Z}^n
     \end{equation}
     then
 \begin{equation}\label{nt_unitary_product}
 U_kU_p = e^{-\pi ik ~\cdot~ \Theta p} U_{k + p}; ~~~   U_kU_p = e^{-2\pi ik ~\cdot~ \Theta p}U_pU_k.
 \end{equation}
      
    If $\xi_k = \Psi_\Th\left(U_k \right)$ then from \eqref{nt_product_defn_eqn}, \eqref{nt_state_eqn} it turns out
   \begin{equation}\label{nt_h_product}
    \tau\left(U^*_k \star_\Th U_l \right) = \left(\xi_k, \xi_l \right)  = \delta_{kl},   
    \end{equation} 
     i.e. the subset $\left\{\xi_k\right\}_{k \in \mathbb{Z}^n}\subset L^2\left(C\left(\mathbb{T}^n_{\Theta}\right), \tau\right)$ is an orthogonal basis of  $L^2\left(C\left(\mathbb{T}^n_{\Theta}\right), \tau\right)$.
     Hence the Hilbert space  $L^2\left(C\left(\mathbb{T}^n_{\Theta}\right), \tau\right)$ is naturally isomorphic to the Hilbert space $\ell^2\left(\mathbb{Z}^n\right)$ given by
     \begin{equation*}
     \ell^2\left(\mathbb{Z}^n\right) = \left\{\xi = \left\{\xi_k \in \mathbb{C}\right\}_{k\in \mathbb{Z}^n} \in \mathbb{C}^{\mathbb{Z}^n}~|~ \sum_{k\in \mathbb{Z}^n} \left|\xi_k\right|^2 < \infty\right\}
     \end{equation*}
     and the $\C$-valued scalar product on $\ell^2\left(\mathbb{Z}^n\right)$ is given by
     \begin{equation*}
     \left(\xi,\eta\right)_{ \ell^2\left(\mathbb{Z}^n\right)}= \sum_{k\in \mathbb{Z}^n}    \overline{\xi}_k\eta_k.
     \end{equation*}
   
    An alternative description of $\C\left(\mathbb{T}^n_{\Theta}\right)$ is such that if
       \begin{equation}\label{nt_th_eqn}
   \Th = \begin{pmatrix}
0& \th_{12} &\ldots & \th_{1n}\\
\th_{21}& 0 &\ldots & \th_{2n}\\
\vdots& \vdots &\ddots & \vdots\\
\th_{n1}& \th_{n2} &\ldots & 0
\end{pmatrix}
    \end{equation}
    
    then $C\left(\mathbb{T}^n_{\Theta}\right)$ is the universal $C^*$-algebra generated by unitary elements   $u_1,..., u_n \in U\left( C\left(\mathbb{T}^n_{\Theta}\right)\right) $ such that following condition holds
    \begin{equation}\label{nt_com_eqn}
    u_ju_k = e^{-2\pi i \theta_{jk} }u_ku_j.
    \end{equation}
    Elements $u_j$ are given by
    \begin{equation*}
    \begin{split}
   u_j = U_{\mathfrak{j}},\\
   \mathfrak{j}=\left(0,\dots, \underbrace{ 1}_{j^{\text{th}}-\text{place}}, \dots, 0\right) .
    \end{split}
    \end{equation*}
      \begin{defn}\label{nt_uni_defn}
      	Unitary elements 
      	$u_1,\dots, u_n \in U\left(C\left(\mathbb{T}^n_{\theta}\right)\right)$ which satisfy the relation \eqref{nt_com_eqn}
      	are said to be \textit{generators} of $C\left(\mathbb{T}^n_{\Theta}\right)$. The set $\left\{U_l\right\}_{l \in \Z^n}$ is said to be the \textit{basis} of $C\left(\mathbb{T}^n_{\Theta}\right)$.
      \end{defn}
           If $a \in C\left(\mathbb{T}^n_{\Th}\right)$ is presented by a series
           $$
           a = \sum_{l \in \mathbb{Z}^{n}}c_l U_l;~~ c_l \in \mathbb{C}
           $$
           and the series $\sum_{l \in \mathbb{Z}^{n}}\left| c_l\right| $ is convergent then from the triangle inequality it follows that
           \begin{equation}\label{nt_norm_estimation}
           \left\|a \right\| \le \sum_{l \in \mathbb{Z}^{n}}\left| c_l\right|.
           \end{equation}
      \begin{defn}\label{nt_symplectic_defn}
      	If  $\Theta$ is non-degenerated, that is to say,
      	$\sigma(s,t) \stackrel{\mathrm{def}}{=} s\.\Theta t$ to be \textit{symplectic}. This implies even
      	dimension, $n = 2N$. Then one  selects
      	\begin{equation}\label{nt_simpectic_theta_eqn}
      	\Theta = \theta J
      	\stackrel{\mathrm{def}}{=} \th \begin{pmatrix} 0 & 1_N \\ -1_N & 0 \end{pmatrix}
      	\end{equation}
      	where  $\th > 0$ is defined by $\th^{2N} \stackrel{\mathrm{def}}{=} \det\Theta$.
      	Denote by $\Coo\left(\mathbb{T}^{2N}_\th\right)\stackrel{\mathrm{def}}{=}\Coo\left(\mathbb{T}^{2N}_\Th\right)$ and $C\left(\mathbb{T}^{2N}_\th\right)\stackrel{\mathrm{def}}{=}C\left(\mathbb{T}^{2N}_\Th\right)$.
      \end{defn}
   
\subsection{Finite-fold coverings}\label{nt_fin_cov}
     \paragraph{}  In this section we write $ab$ instead $a\star_\Th b$.
  Let $\Th$ be given by \eqref{nt_th_eqn}, and let $C\left(\mathbb{T}^n_\Theta\right)$ be a noncommutative torus. If  $\left(k_1, \dots, k_n\right) \in \mathbb{N}^n$ and
               $$
               \widetilde{\Theta} = \begin{pmatrix}
               0& \widetilde{\theta}_{12} &\ldots & \widetilde{\theta}_{1n}\\
               \widetilde{\theta}_{21}& 0 &\ldots & \widetilde{\theta}_{2n}\\
               \vdots& \vdots &\ddots & \vdots\\
               \widetilde{\theta}_{n1}& \widetilde{\theta}_{n2} &\ldots & 0
               \end{pmatrix}
               $$
       is a skew-symmetric matrix such that
            \begin{equation*}
            e^{-2\pi i \theta_{rs}}= e^{-2\pi i \widetilde{\theta}_{rs}k_rk_s}
            \end{equation*}
            then there is a *-homomorphism $C\left(\mathbb{T}^n_\Th\right)\to C\left(\mathbb{T}^n_{\widetilde{\Th}}\right)$ given by
            \begin{equation}\label{nt_cov_eqn}
            u_j \mapsto v~^{k_j}_j; ~ j = 1,...,n
            \end{equation}
            where $u_1,..., u_n \in C\left(\mathbb{T}^n_{\Th}\right)$ (resp. $v_1,..., v_n \in C\left(\mathbb{T}^n_{\widetilde{\Th}}\right)$) are unitary generators of $C\left(\mathbb{T}^n_{\Th}\right)$ (resp. $C\left(\mathbb{T}^n_{\widetilde{\Th}}\right)$).	
            There is an involutive action of $G=\mathbb{Z}_{k_1}\times...\times\mathbb{Z}_{k_n}$ on $C\left(\mathbb{T}^n_{\widetilde{\Th}}\right)$ given by
            \begin{equation*}
            \left(\overline{p}_1,..., \overline{p}_n\right)v_j = e^{\frac{2\pi i p_j}{k_j}}v_j,
            \end{equation*}
            and a following condition holds $C\left(\mathbb{T}^n_{\Th}\right)=C\left(\mathbb{T}^n_{\widetilde{\Th}}\right)^G$.
            Otherwise there is a following $C\left(\mathbb{T}^n_{\Th}\right)$ - module isomorphism
            $$
           C\left(\mathbb{T}^n_{\widetilde{\Th}}\right) = \bigoplus_{\left(\overline{p}_1, ... \overline{p}_n \right)\in\mathbb{Z}_{k_1}\times...\times\mathbb{Z}_{k_n} } v_1^{p_1} \cdot ... \cdot v_n^{p_n} C\left(\mathbb{T}^n_{\Th}\right) \approx C\left(\mathbb{T}^n_{\Th}\right)^{k_1\cdot ... \cdot k_n}
            $$
            i.e. $C\left(\mathbb{T}^n_{\widetilde{\Th}}\right)$ is a finitely generated projective Hilbert $C\left(\mathbb{T}^n_{\Th}\right)$-module.
            It turns out the following theorem. 
             \begin{thm}\label{nt_fin_cov_lem}
            
           The triple $\left(C\left(\mathbb{T}^n_{\Th}\right), C\left(\mathbb{T}^n_{\widetilde{\Th}}\right),\mathbb{Z}_{k_1}\times...\times\mathbb{Z}_{k_n}\right)$  is an unital noncommutative finite-fold  covering.
           \end{thm}
                        \subsection{Moyal plane and a representation of the noncommutative torus}\label{nt_ind_repr_subsubsec}
       \begin{defn}
       	Denote the \textit{Moyal plane} product $\star_\th$ on $\SS\left(\R^{2N} \right)$ given by
       	$$
       	\left(f \star_\th h \right)\left(u \right)= \int_{y \in \R^{2N} } f\left(u - \frac{1}{2}\Th y\right) g\left(u + v \right)e^{2\pi i y \cdot v }  dydv
       	$$ 
       	where $\Th$ is given by \eqref{nt_simpectic_theta_eqn}.
       \end{defn}
    \begin{defn}\label{mp_mult_defn}\cite{varilly_bondia:phobos}
    	\label{df:Moyal-alg} 	Denote by $\SS'\left( \R^{n}\right) $ the vector space dual to $\SS\left( \R^{n}\right) $, i.e. the space of continuous functionals on $\SS\left( \R^{n}\right)$.
    	The Moyal product can be defined, by duality, on larger sets than
    	$\SS\left(\R^{2N}\right)$. For $T \in \SS'\left(\R^{2N}\right)$, write the evaluation on $g \in \SS\left(\R^{2N}\right)$ as
    	$\<T, g> \in \C$; then, for $f \in \SS$ we may define $T \star_{\theta} f$ and
    	$f \star_{\theta} T$ as elements of~$\SS'\left(\R^{2N}\right)$ by
    	\begin{equation}\label{mp_star_ext_eqn}
    	\begin{split}
    	\<T \star_{\theta} f, g> \stackrel{\mathrm{def}}{=} \<T, f \star_{\theta} g>\\
    	\<f \star_{\theta} T, g> \stackrel{\mathrm{def}}{=} \<T, g \star_{\theta} f>	\end{split}
    	\end{equation}  using the continuity of the
    	star product on~$\SS\left(\R^{2N}\right)$. Also, the involution is extended to  by
    	$\<T^*,g> \stackrel{\mathrm{def}}{=} \overline{\<T,g^*>}$.	
    %
    \end{defn}
       \begin{rem}
       	It is proven in \cite{moyal_spectral} that the domain of  the Moyal plane product can be extended up to $L^2\left(\R^{2N} \right)$. 
       \end{rem}
       \begin{lem}\label{nt_l_2_est_lem}\cite{moyal_spectral}
       	If $f,g \in L^2 \left(\R^{2N} \right)$, then $f\star_\th g \in L^2 \left(\R^{2N} \right)$ and $\left\|f\right\|_{\mathrm{op}} < \left(2\pi\th \right)^{-\frac{N}{2}} \left\|f\right\|_2$.
       	where	$\left\|\cdot\right\|_{2}$ is the $L^2$-norm given by
       	\begin{equation}\label{nt_l2_norm_eqn}
       	\left\|f\right\|_{2} \stackrel{\mathrm{def}}{=} \left|\int_{\R^{2N}} \left|f\right|^2 dx \right|^{\frac{1}{2}}.
       	\end{equation}
       	and the operator norm $\|T\|_{\mathrm{op}} \stackrel{\mathrm{def}}{=}\sup\set{\|T \star g\|_2/\|g\|_2 : 0 \neq g \in L^2\left( \R^{2N})\right) }$ 
       \end{lem}
       
       \begin{defn}\label{mp_star_alg_defn}
       	Denote by $\SS\left(\R^{2N}_\th \right)$  (resp. $L^2\left(\R^{2N}_\th \right)$ ) the operator algebra  which is $\C$-linearly isomorphic to $\SS\left(\R^{2N} \right)$  (resp. $L^2\left(\R^{2N} \right)$ ) and product coincides with  $\star_\th$. Both  $\SS\left(\R^{2N}_\th \right)$ and $L^2\left(\R^{2N}_\th \right)$ act on the Hilbert space $L^2\left(\R^{2N} \right)$. Denote by
       	\begin{equation}\label{mp_psi_th_eqn}
     \Psi_\th:  \SS\left(\R^{2N} \right)\xrightarrow{\approx}\SS\left(\R^{2N}_\th \right)
       	\end{equation}
       	the natural $\C$-linear isomorphism.  
       \end{defn}
       \begin{empt}
       	There is the tracial property \cite{moyal_spectral} of the Moyal product
       	\begin{equation}\label{nt_tracial_prop}
       	\int_{\R^{2N}} \left( f\star_\th g\right) \left(x \right)dx =  \int_{\R^{2N}}  f\left(x \right) g\left(x \right)dx.
       	\end{equation}
       	The Fourier transformation of the star product satisfies to the following condition.
       \begin{equation}\label{mp_fourier_eqn}
      	\mathcal{F}\left(f \star_\th g\right) \left(x \right) =    	\int_{\R^{2N}}\mathcal{F}{f}\left(x-y \right) \mathcal{F}{g}\left(y\right)e^{ \pi i  y \cdot \Th x }~dy.
       \end{equation}
 
       \end{empt} 
       \begin{defn}\label{r_2_N_repr}\cite{moyal_spectral} Let $\SS'\left(\R^{2N} \right)$ be a vector space dual to $\SS\left(\R^{2N} \right)$. Denote by $C_b\left(\R^{2N}_\th\right)\stackrel{\mathrm{def}}{=} \set{T \in \SS'\left(\R^{2N}\right) : T \star_\th g \in L^2\left(\R^{2N}\right) \text{ for all } g \in L^2(\R^{2N})}$, provided with the operator norm \begin{equation}\label{mp_op_norm_eqn}
       		\|T\|_{\mathrm{op}} \stackrel{\mathrm{def}}{=}\sup\set{\|T \star_\th g\|_2/\|g\|_2 : 0 \neq g \in L^2(\R^{2N})}.
       	\end{equation}
       	Denote by $C_0\left(\R^{2N}_\th \right)$ the operator norm completion of $\SS\left(\R^{2N}_\th \right).$  
       \end{defn}
          \begin{rem}
       	Obviously $\SS\left(\R^{2N}_\th\right)  \hookto C_b\left(\R^{2N}_\th\right)$. But $\SS\left(\R^{2N}_\th\right)$ is not dense in $C_b\left(\R^{2N}_\th\right)$, i.e. $C_0\left(\R^{2N}_\th\right) \subsetneq C_b\left(\R^{2N}_\th\right)$  (cf. \cite{moyal_spectral}).
       \end{rem}
   
          \begin{rem}
   	$L^2\left(\R^{2N}_\th\right)$ is the $\|\cdot\|_2$ norm completion of $\SS\left(\R^{2N}_\th\right)$ hence 
   	from the Lemma \ref{nt_l_2_est_lem} it follows that 
   	\begin{equation}\label{mp_2_op_eqn}
   	L^2\left(\R^{2N}_\th\right) \subset C_0\left(\R^{2N}_\th\right).
   	\end{equation} 
   \end{rem}
   \begin{rem}
   	Notation of the Definition \ref{r_2_N_repr} differs from \cite{moyal_spectral}. Here symbols $A_\th, \A_\th, A^0_\th$ are replaced with $C_b\left(\R^{2N}_\th\right), \SS\left(\R^{2N}_\th\right), C_0\left(\R^{2N}_\th\right)$ respectively.
   \end{rem}
   \begin{rem}
   	The $\C$-linear space $C_0\left(\R^{2N}_\th \right)$ does not isomorphic to $C_0\left(\R^{2N}\right)$  (cf. \cite{moyal_spectral}). 
   \end{rem}
\begin{empt}\cite{moyal_spectral}
	By plane waves we understand all functions of the form
	$$
	x \mapsto \exp(ik\cdot x) 
	$$
	for  $k\in \R^{2N}$.  One obtains for the Moyal
	product of plane waves:
	\begin{equation}\label{mp_wave_prod_eqn}
	\begin{split}
\exp\left(ik\cdot\right) \star_{\Theta}\exp\left(ik\cdot\right)=\exp\left(ik\cdot\right) \star_{\theta}\exp\left(ik\cdot\right)= \exp\left(i\left( k+l\right) \cdot\right) e^{-\pi i k \cdot \Th l}
	\end{split}
	\end{equation}

\end{empt}
  \begin{rem}\cite{moyal_spectral}
  	The algebra $C_b\left(\R^{2N}_\th \right)$  contains all plane waves.
  \end{rem}
\begin{rem}\label{nt_c_k_eqn}
If $ \left\{c_k \in \C\right\}_{k \in N^0} $ is such that $\sum_{k = 0}^\infty \left|c_k\right| < \infty$ then from $\left\| \exp\left(ik\cdot\right)\right\|_{\mathrm{op}}=1$ it turns out  $\left\|\sum_{k = 0}^\infty \exp\left(ik\cdot\right)\right\|_{\mathrm{op}}< \sum_{k = 0}^\infty \left|c_k\right| < \infty$, i.e. $\sum_{k = 0}^\infty c_k\exp\left(ik\cdot\right) \in C_b\left(\R^{2N}_\th \right)$. 
\end{rem}
 \begin{empt}
	The equation \eqref{mp_wave_prod_eqn} is similar to  the equation \eqref{nt_unitary_product} which defines $C\left(\mathbb{T}^{2N}_{\theta}\right)$. From this fact and from the Remark \ref{nt_c_k_eqn} it follows that there is an injective *-homomorphism
	$\Coo\left(\mathbb{T}^{2N}_{\theta}\right) \hookto C_b\left(\R_\th^{2N} \right);~	U_k \mapsto \exp\left(2\pi ik\cdot\right)$. An algebra $\Coo\left(\mathbb{T}^{2N}_{\theta}\right)$ is dense in $C\left(\mathbb{T}^{2N}_{\theta}\right)$ so there is an injective *-homomorphism
	$C\left(\mathbb{T}^{2N}_{\theta}\right) \hookto C_b\left(\R_\th^{2N} \right)$. The faithful representation $C_b\left(\mathbb{R}^{2N}_{\theta}\right)\to B\left(L^2\left( \R^{2N}\right) \right)$
	gives a representation $\pi: C\left(\mathbb{T}^{2N}_{\theta}\right) \to B\left(L^2\left( \R^{2N}\right) \right)$
	\begin{equation}\label{nt_l2r_eqn}
	\begin{split}
	\pi: C\left(\mathbb{T}^{2N}_{\theta}\right) \to B\left(L^2\left( \R^{2N}\right) \right), \\
	U_k \mapsto \exp\left(2\pi ik\cdot\right)
	\end{split}
	\end{equation}
	
	where $U_k\in C\left(\mathbb{T}^n_{\Theta}\right)$ is given by the Definition \ref{nt_uni_defn}.
\end{empt} 
\begin{empt}\label{mp_scaling_constr}

	Let us consider the unitary dilation operators $E_a$ given
	by
	$$
	E_af(x) \stackrel{\mathrm{def}}{=} a^{N/2} f(a^{1/2}x),
	$$
	It is proven in \cite{moyal_spectral} that
	\begin{equation}\label{eq:starscale}
	f {\star_{\theta}} g =
	(\th/2)^{-N/2} E_{2/\th}(E_{\th/2}f \star_2 E_{\th/2}g).
	\end{equation}
	We can simplify our construction by setting $\th = 2$. Thanks to
	the scaling relation~\eqref{eq:starscale} any qualitative result can is true if it is true in case of 
	$\th = 2$. We use the following notation
	\begin{equation}\label{mp_times_eqn}
	f {\times} g\stackrel{\mathrm{def}}{=}f {\star_{2}} g
	\end{equation}
\end{empt}
\begin{lem}\label{mp_ab_delta_lem}
	Let $a, b \in \SS\left(\R^{2N} \right)$. For any $\Delta \in \R^{2N}$ let $a_\Delta \in  \SS\left(\R^{2N} \right) $ be such that $a_\Delta\left(x \right)= a\left(x + \Delta  \right)$. For any $m \in \N$ there is a constant $C^{a,b}_m$ such that
	\be\nonumber
	\left\| a_\Delta \times b\right\|_2 < \frac{C^{a,b}_m}{\left( 1+\left\|\Delta \right\|\right)^m  }
	\ee
	where $\left\| \cdot\right\|_2$ is given by \eqref{nt_l2_norm_eqn}.
\end{lem}
\begin{proof}
	From the definition of Schwartz functions it follows that for any $f \in \sS\left( \mathbb{R}^{2N} \right)$ and any $m \in \mathbb{N}$ there is $C^f_m>0$ such that
	\begin{equation}\label{nt_c_f_m_eqn}
	\left|f \left(u\right)\right|<\frac{C^f_m}{\left( 1 + \left\|u\right\|\right)^m }.
	\end{equation}
From \eqref{mp_fourier_eqn} it follows that
$$
\mathcal{F}\left(a_\Delta \times b  \right)\left(x \right)  =    	\int_{\R^{2N}}\mathcal{F}a_\Delta\left(x-y \right) \mathcal{F}b\left(y\right)e^{ \pi i  y \cdot \Th x }~dy = \int_{\R^{2N}}c\left(y-\Delta -x \right) d\left(y\right)e^{ \pi i  y \cdot \Th x }~dy
$$
where $c\left( x\right) = \mathcal{F}a\left(-x \right)$, $d\left( x\right) = \mathcal{F}b\left(x \right)$	If $\xi =\mathcal{F}\left(a_\Delta \times b  \right)$ then $\xi \in L^2\left(\R^{2N} \right)$. Let $\xi=\xi_1+ \xi_2$ where $\xi_1, \xi_2 \in L^2\left(\R^{2N} \right)$ are given by
\be
\begin{split}
	\xi_1\left(x \right)= \begin{cases}
\mathcal{F}\left(a_{\Delta}  b\right)\left(x \right)& \left\| x\right\| \le  \frac{\left\| \Delta\right\| }{2}\\
0& \left\| x\right\| >  \frac{\left\| \Delta\right\| }{2}
	\end{cases},\\
	\xi_2\left(x \right)= \begin{cases}
0 & \left\| x\right\|\le  \frac{\left\| \Delta\right\| }{2}\\
	\mathcal{F}\left(a_{\Delta}  b\right)\left(x \right)& \left\| x\right\| >  \frac{\left\| \Delta\right\| }{2}
\end{cases}.\\
\end{split}
\ee

 From \eqref{nt_c_f_m_eqn} it turns out
\begin{equation}\label{mp_xi_1_eqn}
\begin{split}
\left| \xi_1\left(x \right)\right|  \le ~ \int \left| c\left( t - \Delta - x \right)  d\left(t \right)e^{\frac{\pi i l}{m_j} \cdot \Th t}  \right|dt \le
  \int_{\mathbb{R}^{2N}}\frac{C^{c}_{M}}{\left(1 + \left\|t - \Delta - x\right\|  \right)^{M}}~\frac{C^{d}_{2M}}{\left(1 + \left\|t\right\|  \right)^{2M}} dt  =
\\
=  \int_{\mathbb{R}^{2N}}\frac{C^{c}_{M}}{\left(1 + \left\|t - \Delta - x\right\|  \right)^{M}\left(1 + \left\|t\right\|  \right)^{M} }~\frac{C^{d}_{2M}}{\left(1 + \left\|t\right\|  \right)^{M}} dt  \le
\\
\le \sup_{x \in \mathbb{R}^{2N}, ~\left\| x\right\|\le \frac{\left\| \Delta\right\|}{2 },~s\in \mathbb{R}^{2N}}~ \frac{C^{c}_{M}C^{d}_{2M}}{\left(1 + \left\|s - \Delta - x\right\|  \right)^{M} \left(1 + \left\|s\right\|  \right)^{M}}
\times  \int_{\mathbb{R}^{2N}}\frac{1}{\left(1 + \left\|t\right\|  \right)^{M}} dt.
\end{split}
\end{equation}
	If $x,y \in \mathbb{R}^{2N}$  then from the triangle inequality it follows that $\left\|x + y\right\|>  \left\|y\right\| - \left\|x\right\|$, hence 
$$
\left(1 + \left\|x\right\| \right)^M \left(1 + \left\|x+ y\right\| \right)^M \ge \left(1 + \left\|x\right\| \right)^M \left(1 +  \max\left(0, \left\|y\right\| - \left\|x\right\| \right)\right)^M.
$$
If $ \left\|x\right\| \le \frac{\left\|y\right\|}{2}$ then $\left\|y\right\| - \left\|x\right\| \ge \frac{\left\|y\right\|}{2}$ and 
\begin{equation}\label{nt_triangle_eqn}
\left(1 + \left\|x\right\| \right)^M \left(1 + \left\|x+ y\right\| \right)^M > \left( \frac{\left\|y\right\|}{2}\right)^M. 
\end{equation}

Clearly if $ \left\|x\right\| > \frac{\left\|y\right\|}{2}$ then condition \eqref{nt_triangle_eqn} also holds, hence\eqref{nt_triangle_eqn} is always true.
	Clearly if $ \left\|x\right\| > \frac{\left\|y\right\|}{2}$ then condition \eqref{nt_triangle_eqn} also holds, hence\eqref{nt_triangle_eqn} is always true.
It turns out from $\left\|-x-\Delta\right\|> \frac{\left\|\Delta\right\|}{2} $ and \eqref{nt_triangle_eqn} that
$$
\mathrm{inf}_{x \in \mathbb{R}^{2N}, ~\left\| x\right\|\le \frac{\left\| \Delta\right\|}{2 },~s\in \mathbb{R}^{2N}} \left(1 + \left\|s - \Delta - x\right\|  \right)^{M} \left(1 + \left\|s\right\|  \right)^{M} >\left\|\frac{\Delta}{4}\right\|^M,
$$
hence from \eqref{mp_xi_1_eqn} it turns out
\be\nonumber
\left| \xi_1\left(x \right)\right| \le \frac{4^MC^{c}_{M}C^{d}_{2M}}{\left\|\Delta\right\|^M} \times  \int_{\mathbb{R}^{2N}}\frac{1}{\left(1 + \left\|t\right\| \right)^{M}} dt
\ee
There is the well known integral
$$
\int_{x \in \R^{2N}, ~\left\| x\right\|\le \frac{\left\| \Delta\right\|}{2 }} 1 dx = \frac{\pi^N}{\Gamma\left(N+1 \right) }\left( \frac{\left\|\Delta\right\|}{2}\right)^{2N}
$$
where $\Gamma$ is the Euler gamma function.
If $M > 2N$ then the integral $C' = \int_{\mathbb{R}^{2N}}\frac{1}{\left(1 + \left\|t\right\| \right)^{M}} dt$ is convergent, it turns out
\be\nonumber
\left| \xi_1\right|^2_2 \le \left( \frac{4^MC'C^{c}_{M}C^{d}_{2M}}{\left\|\Delta\right\|^M}\right)^2   \int_{x \in \R^{2N}, ~\left\| x\right\|\le \frac{\left\| \Delta\right\|}{2 }} 1 dx = \frac{4^MC'C^{c}_{M}C^{d}_{2M}}{\left\|\Delta\right\|^M}\frac{\pi^N}{\Gamma\left(N+1 \right) }\left( \frac{\left\|\Delta\right\|}{2}\right)^{2N}.
\ee
 
If $M = 2N+ m$ then from the above equation it turns out that there is $C_1 > 0$ such that
\be\label{mp_xi1_eqn}
\left| \xi_1\right|^2_2 \le \frac{C_1}{\left\|\Delta\right\|^m}.
\ee
 If $\left( \cdot, \cdot \right)_{L^2\left(\R^{2N} \right) }$ is the given by \eqref{fourier_scalar_product_eqn} scalar product then from \eqref{mp_inv_eqn} it turns out
\begin{equation*}
\begin{split}
\left| \xi_2\left(x \right) \right|\le\left|\int c\left( t - \Delta - x \right)  d\left(t \right)e^{\pi i x \cdot \Th t}  dt\right|=\\
=\left|\left(    c\left( \bullet- \Delta-x\right) ,~ d\left(\bullet\right)e^{\pi i x\cdot \Th\bullet}   \right)_{L^2\left(\R^{2N} \right) }  \right|=\\
=\left|\left(  \mathcal{F}\left(  c\left( \bullet- \Delta-x\right)\right) ,~ \mathcal{F}\left( d\left(\bullet\right)e^{\pi i x\cdot \Th\bullet} \right)  \right)_{L^2\left(\R^{2N} \right) }  \right|=\\=
\left|\int_{\mathbb{R}^{2N}} 
\mathcal{F}\left( c\right)\left( \bullet - \Delta -x\right)\left(u\right)\mathcal{F}\left(d\left(\bullet\right)e^{\pi i x\cdot \Th\bullet} \right)\left(u\right)  du \right|\le\\
\le \int_{\mathbb{R}^{2N}}\left| e^{-i\left(-\Delta - x\right) \cdot u}\mathcal{F}\left( c\right)\left( u\right) \mathcal{F}\left(d\right)\left(u+\pi\Th x\right)\right| du \le \\
\le \int_{\mathbb{R}^{2N}}\frac{C^{\mathcal{F}\left( c\right)}_{3M}}{\left(1 + \left\|u\right\|  \right)^{3M}}\frac{C^{\mathcal{F}\left(d\right)}_{2M}}{\left(1 + \left\|u-\pi\Th x\right\|  \right)^{2M}} du \le \\
\le \sup_{x \in \mathbb{R}^{2N}, ~\left\| x\right\|> \frac{\left\| \Delta\right\|}{2 },~ s \in \mathbb{R}^{2N}}~ \frac{C^{\mathcal{F}\left( c\right)}_{3M}}{\left(1 + \left\|s\right\|  \right)^{M}}\frac{C^{\mathcal{F}\left(d\right)}_{2M}}{\left(1 + \left\|s-\pi\Th x\right\|  \right)^{M}}     
\frac{1}{\left(1 + \left\|u-\pi\Th  x\right\|  \right)^{M}\left(1 + \left\|u\right\|  \right)^{M}}\times\\
\times\int_{\mathbb{R}^{2N}}\frac{1}{\left(1 + \left\|u\right\|  \right)^{M}} du.
\end{split}
\end{equation*}

Since we consider the asymptotic dependence $\left\|\Delta\right\|\to \infty$ only  large values of $\left\|\Delta\right\|$ are interesting, so we can suppose that  $\left\|\Delta\right\| > 2$.
If $\left\|\Delta\right\| > 2$ then from $\left\|x\right\|> \frac{\left\| \Delta\right\|}{2 }$ it follows that $\left\|\pi\Th x\right\| > 1$, and  from \eqref{nt_triangle_eqn} it follows that
\begin{equation*}
\begin{split}
\left(1 + \left\|u\right\|  \right)^{M}\left(1 + \left\|u-\pi\Th x\right\|  \right)^{M} > \left\|\pi\Th x\right\|^M, 
\\
\inf_{x \in \mathbb{R}^{2N}, ~\left\| x\right\|> \frac{\left\| \Delta\right\|}{2 },~ s \in \mathbb{R}^{2N}}~ \left(1 + \left\|s\right\|  \right)^{M}\left(1 + \left\|s-\pi\Th x\right\|  \right)^{M} > \left\|\frac{\pi\Th\Delta}{4}\right\|^M,
\end{split}
\end{equation*}

hence
\begin{equation*}
\begin{split}
\left| \xi_2\left(x \right)   \right|
\le  \frac{C^{\mathcal{F}\left( c\right)}_{3M}C^{\mathcal{F}\left(d\right)}_{2M}}{\left\|\frac{\pi\Delta}{4}\right\|^M \left\|\pi\Th x\right\|^M} \int_{\R^{2N}}\frac{1}{\left(1+\left\|u\right\| \right)^M }du.
\end{split}
\end{equation*}
If $m \ge 1$ and $M = 2N + m$ then the integral $C'\int_{\R^{2N}}\frac{1}{\left(1+\left\|u\right\| \right)^M }du$ is convergent and 
$$
\left| \xi_2\left(x \right)   \right|
\le  \frac{C^{\mathcal{F}\left( c\right)}_{3M}C^{\mathcal{F}\left(d\right)}_{2M}C'}{\left\|\frac{\pi\Delta}{4}\right\|^M \left\|\pi\Th x\right\|^M}.
$$
Taking into account \eqref{nt_simpectic_theta_eqn} and $\th = 2$ one has
\be\label{mp_2x_eqn}
\left\|\Th z\right\|= \left\|2 z\right\|; ~ \forall z \in \R^{2N}.
\ee
It follows that
$$
\left| \xi_1\right|^2_2 \le \int_{x \in \R^{2N} \left\| x\right\| >  \frac{\left\| \Delta\right\| }{2}}\left(\frac{C^{\mathcal{F}\left( c\right)}_{3M}C^{\mathcal{F}\left(d\right)}_{2M}C'}{\left\|\frac{2\pi\Delta}{4}\right\|^M \left\|2\pi x\right\|^M} \right)^2dx .
$$
Since above integral is convergent one has
there is a constant $C_2$ such that
\be\label{mp_xi2_eqn}
\left| \xi_2\right|^2_2 \le \frac{C_2}{\left\|\frac{\pi\Delta}{4}\right\|^{2M}}
\ee
Since $\xi_1 \perp \xi_2$ one has $\left| \xi\right\|_2= \sqrt{\left| \xi_1\right|^2_2+\left| \xi_2\right|^2_2}$ and taking into account \eqref{mp_xi1_eqn}, \eqref{mp_xi2_eqn} it follows that for any $m \in \N$ there is $C_m > 0$ such that
$$
\left\| \xi\right\|_2 = \left\| \mathcal{F}\left(a_{\Delta}\times  b\right)\right\|_2 \le \frac{C^{a,b}_m}{\left( 1+\left\|\Delta \right\|\right)^m  }.
$$
From \eqref{mp_inv_eqn} it turns out
$$
 \left\|a_{\Delta}\times  b\right\|_2=\left\| \mathcal{F}\left(a_{\Delta}\times  b\right)\right\|_2 \le \frac{C^{a,b}_m}{\left( 1+\left\|\Delta \right\|\right)^m  }.
$$

\end{proof}
\begin{prop}\label{mp_factor_prop}\cite{moyal_spectral}
	The algebra $\SS\left(\R^{2N}, \star_\th \right)$  has the (nonunique) factorization property:
	for all $h \in  \SS\left(\R^{2N} \right)$ there exist $f,g \in  \SS\left(\R^{2N} \right)$ that $h = f \star_\th g$.
\end{prop}
\begin{lem}\label{mp_weak_lem}
	Following conditions hold:
	\begin{enumerate}
\item[(i)]
	Let $\left\{a_n \in  C_b\left(\R^{2N}_\th\right)\right\}_{n \in \N}$ be a sequence such that
	\begin{itemize}
		\item $\left\{a_n \right\}$ is weakly-* convergent (cf. Definition \ref{nt_*w_defn}),
		\item If $a = \lim_{n \to \infty} a_n$ in the sense of weak-* convergence then  $a \in C_b\left(\R^{2N}_\th\right)$.
	\end{itemize}
	Then the sequence $\left\{a_n \right\}$ is convergent in sense of weak topology $\left\{a_n \right\}$ (cf. Definition \ref{weak_topology}) and $a$ is limit of $ \left\{a_n \right\}$ with respect to the weak topology. Moreover if $\left\{a_n \right\}$ is increasing or decreasing sequence of self-adjoint elements then $\left\{a_n \right\}$ is convergent in sense of strong topology  
	(cf. Definition \ref{strong_topology})  and $a$ is limit of $ \left\{a_n \right\}$ with respect to the strong topology.
\item[(ii)] If $\left\{a_n \right\}$ is strongly and/or weakly convergent (cf. Definitions \ref{strong_topology}, \ref{weak_topology}) and  $a = \lim_{n \to \infty} a_n$ is strong and/ or weak limit then $\left\{a_n \right\}$ is  weakly-* convergent and $a$ is the limit of $\left\{a_n \right\}$ in the sense of weakly-* convergence.
\end{enumerate}
\end{lem}
\begin{proof}
	(i)
	If $\left\langle \cdot, \cdot \right\rangle: \SS'\left(\R^{2N} \right)\times \SS\left(\R^{2N} \right) \to \C$ is the natural pairing then one has 
	\be\label{mp_an_lim_eqn}
	\lim_{n \to \infty} \left\langle a_n, b \right\rangle = \left\langle a, b \right\rangle; ~~ \forall b \in \SS\left(\R^{2N} \right).
	\ee
	Let $\xi, \eta \in L^2\left(\R^{2N} \right)$ and let $\left\{x_j \in \SS\left(\R^{2N}_\th\right)\right\}_{j \in \N}$, $\left\{y_j \in \SS\left(\R^{2N}_\th\right)\right\}_{ \in \N}$ be  sequences such that there are following  limits
	\be\label{mp_xn_lim_eqn}
	\begin{split}
		\lim_{j \to \infty} x_j = \xi,~~\lim_{j \to \infty} y_j = \eta
	\end{split}
	\ee
	in the topology of the Hilbert space $L^2\left(\R^{2N} \right)$.  If $\left(\cdot, \cdot \right): L^2\left(\R^{2N} \right) \times L^2\left(\R^{2N} \right) \to \C$ is the Hilbert pairing then from \eqref{mp_star_ext_eqn},  \eqref{mp_xn_lim_eqn} it follows that 
	\be\nonumber
	\left(a_n \xi,  \eta \right) = \lim_{j \to \infty} \left\langle a_n x_j ,y_j \right\rangle=\lim_{j \to \infty} \left\langle a_n, x_j \star_\th y_j \right\rangle,
	\ee
	hence, taking into account \eqref{mp_an_lim_eqn} one has
	\be
	\lim_{n \to \infty }\left(a_n \xi,  \eta \right)=\lim_{n \to \infty }\lim_{j \to \infty}\left(a_n x_j,  y_j \right)=\lim_{n \to \infty }\lim_{j \to \infty} \left\langle a_n, x_j \star_\th y_j \right\rangle = \lim_{j \to \infty} \left\langle a, x_j \star_\th y_j \right\rangle=\left(a \xi,  \eta \right),
	\ee
	i.e. $\left\{a_n \right\}$ is weakly convergent to $a$. 
	If $\left\{a_n \right\}$ is an increasing sequence then $a_n < a$ for any $n \in \N$ and from the Lemma \ref{increasing_convergent_w} it turns out that $\left\{a_n \right\}$ is strongly convergent. Clearly the strong limit coincides with the weak one. Similarly one can prove that  is $\left\{a_n \right\}$ is an decreasing then $\left\{a_n \right\}$ is strongly convergent.\\
	(ii) If $b \in \SS\left(\R^{2N} \right)$ then from the Proposition \ref{mp_factor_prop} it follows that $b = x \star_{\th} y$ where $x,y \in \SS\left(\R^{2N} \right)$. The sequence $\left\{a_n \right\}$  is strongly and/or weakly convergent it turns out that
	$$
	\left(x, a_n  \star_{\th} y \right) = \left\langle a_n, x \star_\th y\right\rangle =  \left\langle a_n, b\right\rangle
	$$
	is convergent. Hence $\left\{a_n \right\}$ is weakly-* convergent.
\end{proof}

\paragraph*{}There are elements $f_{mn} \in \sS\left( \mathbb{R}^2\right)$ which have very useful properties.   To present $f_{mn}$ explicitly,  we use polar coordinates $q + ip= \rho e^{i\al}$, where $p, q \in \mathbb{R}^2$ and $\rho = \left(p, q\right)\in \mathbb{R}^{2}$  Note that $\left\|\rho\right\|^2= \left\|p\right\|^2 + \left\|q\right\|^2$.
      \begin{equation*}
      \begin{split}
            f_{mn}= 2\left(-1\right)^n\sqrt{\frac{n!}{m!}}e^{i\al\left(m - n\right)}\left\|\rho\right\|^{m-n}L^{m-n}_n\left(\left\|\rho\right\|^2\right)e^{-\left\|\rho\right\|^2},
   \\
      f_{nn}\left(\rho, \al\right)= 2\left(-1\right)^nL_n\left(\left\|\rho\right\|^2\right)e^{-\left\|\rho\right\|^2/2}
      \end{split}
      \end{equation*}
       where $L^{m-n}_n$, $L^{n}$ are Laguerre functions.
       From this properties it follows that $C_0\left(\R^{2}_\th\right)$ is the $C^*$-norm completion of linear span of $f_{mn}$ (cf. {\rm\cite{varilly_bondia:phobos}}).
       
       \begin{lem}\label{mp_osc_lem} {\rm\cite{varilly_bondia:phobos}}
       	\label{lm:osc-basis}
       	Let $m,n,k,l \in \N$. Then $f_{mn} \star_\th f_{kl} = \delta_{nk}f_{ml}$
       	and $f_{mn}^* = f_{nm}$. Thus $f_{nn}$ is an orthogonal projection and
       	$f_{mn}$ is nilpotent for $m \neq n$. Moreover,
       	$\left\langle f_{mn}, f_{kl}\right\rangle  = 2^N\,\delta_{mk}\,\delta_{nl}$. The family
       	$\set{f_{mn} : m,n\in \N^0} \subset \sS\left( \mathbb{R}^{2}\right) \subset L^2(\R^{2})$ is an
       	orthogonal basis.
       \end{lem}
     
\begin{prop}\label{mp_fmn}\cite{moyal_spectral,varilly_bondia:phobos}
Let $N = 1$. Then $\SS\left(\R^{2N}_\th\right)=\SS\left(\R^{2}_\th\right) $ has a Fr\'echet algebra isomorphism with
the matrix algebra of rapidly decreasing double sequences
$c = (c_{mn})$ such that, for each $k \in \N$,
\begin{equation}\label{mp_matr_norm}
r_k(c) \stackrel{\mathrm{def}}{=} \biggl( \sum_{m,n=0}^\infty
\th^{2k}  \left( m+\half\right)^k \left( n+\half\right)^k |c_{mn}|^2 \biggr)^{1/2}
\end{equation}

is finite, topologized by all the seminorms $(r_k)$; via the
decomposition $f = \sum_{m,n=0}^\infty c_{mn} f_{mn}$ of~$\SS(\R^2)$ in
the $\{f_{mn}\}$ basis.
The twisted product $f \star_\th g$ is
the matrix product $ab$, where
\begin{equation}\label{mp_mult_eqn}
\left( ab\right)_{mn} \stackrel{\mathrm{def}}{=} \sum_{k= 0}^{\infty} a_{mk}b_{kn}.
\end{equation}

For $N > 1$, $\Coo\left(\R^{2N}_\th\right)$ is isomorphic to the (projective) tensor product
of $N$ matrix algebras of this kind, i.e.
\begin{equation}\label{mp_tensor_prod}
\SS\left(\R^{2N}_\th\right) \cong \underbrace{\SS\left(\R^{2}_\th\right)\otimes\dots\otimes\SS\left(\R^{2}_\th\right)}_{N-\mathrm{times}}
\end{equation}
with the projective topology induced by seminorms $r_k$ given by \eqref{mp_matr_norm}.	
\end{prop}
\begin{rem}
	If $A$ is  $C^*$-norm completion of the matrix algebra with the norm \eqref{mp_matr_norm} then $A \approx \mathcal K$, i.e.
	\begin{equation}\label{mp_2_eqn}
	C_0\left(\R^{2}_\th\right) \approx \mathcal K.
	\end{equation}
	Form \eqref{mp_tensor_prod} and \eqref{mp_2_eqn} it follows that
	\begin{equation}\label{mp_2N_eqn}
	C_0\left(\R^{2N}_\th\right) \cong \underbrace{C_0\left(\R^{2}_\th\right)\otimes\dots\otimes C_0\left(\R^{2}_\th\right)}_{N-\mathrm{times}} \approx \underbrace{\mathcal K\otimes\dots\otimes\mathcal K}_{N-\mathrm{times}} \approx \mathcal K
	\end{equation}
	where $\otimes$ means minimal or maximal tensor product ($\mathcal{K}$ is nuclear hence both products coincide).

\end{rem}

   \subsection{Infinite coverings}
\paragraph{} Let us consider a sequence
\begin{equation}\label{nt_long_seq_eqn}
\mathfrak{S}_{C\left( \T^n_\Th\right) } =\left\{ C\left( \T^n_\Th\right)  =C\left( \T^n_{\Th_0}\right) \xrightarrow{\pi^1}  ... \xrightarrow{\pi^j} C\left( \T^n_{\Th_j}\right) \xrightarrow{\pi^{j+1}} ...\right\}.
\end{equation}
of finite coverings of noncommutative tori.
The sequence \eqref{nt_long_seq_eqn} satisfies to the Definition  \ref{comp_defn}, i.e. $\mathfrak{S}_{C\left( \T^n_\Th\right) } \in \mathfrak{FinAlg}$.

\begin{empt}\label{nt_mp_prel_lim}
	Let $\Th = J \th$ where $\th \in \R \backslash \Q$ and
	$$
	J = \begin{pmatrix} 0 & 1_N \\ -1_N & 0 \end{pmatrix}.
	$$
	Denote by $C\left( \T^{2N}_\th\right) \stackrel{\text{def}}{=} C\left( \T^{2N}_\Th\right)$.
	Let $\{p_k \in \mathbb{N}\}_{k \in \mathbb{N}}$ be an infinite sequence of natural numbers such that $p_k > 1$ for any $k$, and let $m_j = \Pi_{k = 1}^{j} p_k$. From the \ref{nt_fin_cov} it follows that there is a  sequence of *-homomorphisms
	\begin{equation}\label{nt_long_seq_spec_eqn}
	\mathfrak{S}_\th= \left\{	C\left(\mathbb{T}^{2N}_\th\right) \to	C\left(\mathbb{T}^{2N}_{\th/m_1^{2}}\right) \to C\left(\mathbb{T}^{2N}_{\th/m_2^{2}}\right) \to... \to C\left(\mathbb{T}^{2N}_{\th/m_j^{2}}\right)\to ...\right\}
	\end{equation}
	such that
	\begin{enumerate}
		\item[(a)] 	For any $j \in \mathbb{N}$ there are generators $u_{j-1,1},..., u_{j-1,2N}\in U\left(C\left(\mathbb{T}^{2N}_{\th/m_{j-1} ^{2}}\right)\right)$ and generators $u_{j,1},..., u_{j,2N}\in U\left(C\left(\mathbb{T}^{2N}_{\th/m_j^{2}}\right)\right)$ such that the *-homomorphism $ C\left(\mathbb{T}^{2N}_{\th/m_{j-1}^{2}}\right)\to C\left(\mathbb{T}^{2N}_{\th/m_j^{2}}\right)$ is given by
		$$
		u_{j-1,k} \mapsto u^{p_j}_{j,k}; ~~ \forall k =1,..., 2N.
		$$
		There are generators $u_1,...,u_{2N} \in U\left( C\left(\mathbb{T}^{2N}_\th\right)\right) $ such that *-homomorphism $C\left(\mathbb{T}^{2N}_\th\right) \to C\left(\mathbb{T}^{2N}_{\th/m_1^{2}}\right)$ is given by
		$$
		u_{j} \mapsto u^{p_1}_{1,j}; ~~ \forall j =1,..., 2N,
		$$
		
		\item[(b)] For any $j \in \N$ the triple $\left(C\left(\mathbb{T}^{2N}_{\th/m_{j - 1}^{2}}, C\left(\mathbb{T}^{2N}_{\th/m_j^{2}}\right), \Z_{p_j}\right) \right)$ is a noncommutative finite-fold covering,
		\item[(c)] There is the sequence of groups and epimorphisms
		\begin{equation*}
		\mathbb{Z}^{2N}_{m_1} \leftarrow\mathbb{Z}^{2N}_{m_2} \leftarrow ...
		\end{equation*}
		which is equivalent to the sequence
		\begin{equation*}
		\begin{split}
		G\left(C\left(\mathbb{T}^{2N}_{\th/m_1^{2}}\right)~|~ C\left(\mathbb{T}^{2N}_{\th}\right)\right)\leftarrow G\left(C\left(\mathbb{T}^{2N}_{\th/m_2^{2}}\right)~|~ C\left(\mathbb{T}^{2N}_{\th}\right)\right)\leftarrow... \\ \leftarrow G\left(C\left(\mathbb{T}^{2N}_{\th/m_j^{2}}\right)~|~ C\left(\mathbb{T}^{2N}_{\th}\right)\right)\leftarrow...\ .
		\end{split}
		\end{equation*}
	\end{enumerate}
	The sequence \eqref{nt_long_seq_spec_eqn}, is a specialization of \eqref{nt_long_seq_eqn}, hence $\mathfrak{S}_\th \in \mathfrak{FinAlg}$. Denote by $\widehat{C\left(\mathbb{T}^{2N}_{\th}\right) } \stackrel{\text{def}}{=} \varinjlim C\left(\mathbb{T}^{2N}_{\th/m_j^{2}}\right)$, $\widehat{G} \stackrel{\text{def}}{=} \varprojlim G\left(C\left(\mathbb{T}^{2N}_{\th/m_j^{2}}\right)~|~ C\left(\mathbb{T}^{2N}_{\th}\right)\right)$ . The group $\widehat{G}$ is Abelian because it is the inverse limit of Ablelian groups. Denote by $0_{\widehat{G}}$ (resp. "+") the neutral element of $\widehat{G}$ (resp. the product operation of $\widehat{G}$).
\end{empt}
\begin{empt}\label{mp_weak_constr}
	For any $\widetilde{a} \in \SS\left( \R^{2N}_\th\right)$ from \eqref{mp_sooth_sum_eqn} it turns out that the series
	$$
	a_j = \sum_{g \in \ker\left( \Z^{2N} \to \Z^{2N}_{m_j}\right) } \widetilde{a} 
	$$
	is point-wise convergent  and $a_j$ satisfies to following conditions:
	\begin{itemize}
		\item $a_j \in \SS'\left(\R^{2N} \right)$,
		\item  $a_j$ is a smooth $m_j$ - periodic function.
	\end{itemize}
	It follows that the above series is  weakly-* convergent (cf. Definition \ref{nt_*w_defn}) and from the Lemma \ref{mp_weak_lem} it turns out that the series is weakly convergent.  From \eqref{fourier_from_r_to_z_eqn} it follows that
	$$
	a_j = \sum_{k \in Z^{2N}} c_k \exp\left(2\pi i \frac{k}{m_j} ~\cdot \right)
	$$
	where $\left\{c_k \in \C\right\}_{k \in \Z^{2N}}$ are rapidly decreasing coefficients given by
	\begin{equation}\label{mp_fourier_torus_eqn}
	c_k = \frac{1}{m_j^{2N} } \int_{\R^{2N}} \widetilde{a}\left(x \right)  \exp\left(2\pi i \frac{k}{m_j}\cdot x\right) dx = \frac{1}{m_j^{2N} } \mathcal{F}\widetilde{a}\left( \frac{k}{m_j}\right)   .
	\end{equation}

	On the other hand
	\begin{equation}\label{mp_weak_lim_eqn}
	\widetilde a = \lim_{j \to \infty} a_j
	\end{equation}
	in sense of weakly-* convergence, and from the Lemma \ref{mp_weak_lem} it follows that \eqref{mp_weak_lim_eqn} is a limit in sense of the weak topology.
\end{empt}
\begin{lem}\label{mp_strong_lem}
	Let $\overline{G}_j = \ker\left( \Z^{2N} \to \Z^{2N}_{m_j}\right)$.
Let $\widetilde{a} \in \SS\left( \R^{2N}_\th\right)$ and let
\be\label{mp_aj_eqn}
a_j = \sum_{g \in\overline{G}_j } g\widetilde{a} 
\ee
where the sum the series means weakly-* convergence. Following conditions hold:
\begin{enumerate}
	\item [(i)] $a_j \in \Coo\left(\R^{2N} \right)$,
	\item[(ii)]  The series \eqref{mp_aj_eqn} is convergent with respect to the strong topology (cf. Definition \ref{strong_topology}),
	\item[(iii)] There is a following strong limit
	\be\label{mp_ta_eqn}
	\widetilde{a} = \lim_{j \to \infty} a_j.
	\ee
\end{enumerate}
 
\end{lem}
\begin{proof}
	
(i) From \ref{mp_fourier_torus_eqn} it turns out that 
$$
a_j = \sum_{k \in \Z^{2N}} c_k U_k
$$
where $\left\{c_k\right\}$ is a rapidly decreasing sequence, hence $a_j \in \Coo\left(\R^{2N} \right)$.\\
(ii)
From the Lemma \ref{mp_weak_lem} it turns out that the series
	$$
	c = \sum_{g \in \Z^{2N}}g\left(\widetilde{a}^*\widetilde{a} \right) 
	$$
	is strongly convergent, and the series \eqref{mp_aj_eqn} is weakly convergent. If $k = \max\left(1, \sqrt{\left\|c \right\|}  \right)$ then for any $\eta \in L^2\left( \R^{2N}\right)$ and any subset $G \subset \Z^{2N}$ following condition holds
	$$
	\left\|\left( \sum_{g \in G}g \widetilde{a}\right)\eta \right\|_2 \le	k\left\|\eta \right\|_2
	$$ 
	where $\left\|\cdot \right\|_2$ is given by \eqref{nt_l2_norm_eqn}.
If $\xi \in L^2\left(\R^{2N} \right)$ then for any $\eps > 0$ there is $\widetilde{b} \in \SS\left(\R^{2N} \right)$ such that 
	\be\label{mp_e_k_eqn}
	\left\|\xi - \widetilde{b}\right\|_2 < \frac{\eps}{2k}
	\ee
	From the Lemma \ref{mp_ab_delta_lem} it follows that for any $m \in \N$ there is a constant $C_m > 0$ such that
	\be\label{mp_est_eqn}
	\left\| \left(g \widetilde{a}\right)b\right\|_2  < \frac{C_m}{\left(1+ \left\|g\right\|\right)^m };~ \forall g \in \Z^{2N}
	\ee
	where $\left\|g\right\|$ is given by \eqref{mp_znorm_eqn}. If $m > 2N$ then there is $M \in \N$ such that if $G_0 = 
	\left\{-M, \dots, M\right\}^{2N} \subset  \Z^{2N}$ such that
\be\label{mp_cm_eqn}
	\sum_{g \in \Z^{2N} \backslash G_0} \frac{C_m}{\left(1+ \left\|g\right\|\right)^m } < \frac{\eps}{2}.
\ee
	It follows that
\be\nonumber
\begin{split}
	\left\|\left(	\sum_{g \in \overline{G}_j } \widetilde{a}-	\sum_{g \in \overline{G}_j\bigcap G_0}\widetilde{a}\right) \widetilde{b} \right\|_2 = \left\|\left(	\sum_{g \in \overline{G}_j\backslash \left(\overline{G}_j\bigcap G_0\right) } \widetilde{a}\right) \widetilde{b} \right\|_2  <  	\sum_{g \in \overline{G}_j \backslash \left(\overline{G}_j\bigcap G_0\right) } \frac{C_m}{\left(1+ \left\|g\right\|\right)^m } < \frac{\eps}{2}.
\end{split}
\ee	
Otherwise from \eqref{mp_e_k_eqn}-\eqref{mp_cm_eqn} one has
\be\nonumber
\begin{split}
	\left\|\left(	\sum_{g \in \overline{G}_j } \widetilde{a}-	\sum_{g \in   \overline{G}_j\bigcap G_0}\widetilde{a}\right) \xi \right\|_2 < 	\left\|\left(	\sum_{g \in \overline{G}_j \backslash \left(\overline{G}_j\bigcap G_0\right) } \widetilde{a}\right) \widetilde{b} \right\|_2 + \left\|\left(	\sum_{g \in \overline{G}_j \backslash G_0 } \widetilde{a}\right) \left(\xi - \widetilde{b} \right)   \right\|_2 <\\<\frac{\eps}{2}+ k \left\|\xi- \widetilde{b}\right\|_2 < \eps.
\end{split}
\ee
Above equation means that the series \eqref{mp_aj_eqn} is strongly convergent.\\
(iii)
If $j \in \N$ is such that $m_j > M$ then
\be\nonumber
	\left\|\left( a_j - \widetilde{a}\right) \xi \right\|_2 = \left\|\left( \sum_{g \in \overline{G}_j}g \widetilde{a}-  \widetilde{a}\right) \xi \right\|_2= \left\|\left( \sum_{g \in \overline{G}_j\backslash \{0\}}g \widetilde{a}  \right) \xi \right\|_2
\ee
where $0$ is the neutral element of $\Z^{2N}$. However from  $m_j > M$ it turns out $G_0 \bigcap \left(  \overline{G}_j\backslash \{0\}\right) =\emptyset$, so from \eqref{mp_e_k_eqn}-\eqref{mp_cm_eqn} one has
\be\nonumber
\begin{split}
\left\|\left( a_j - \widetilde{a}\right) \xi \right\|_2 < \left\|\left( a_j - \widetilde{a}\right) \widetilde{b} \right\|_2  +k \left\| \xi- \widetilde{b} \right\|_2 < \\ < \left\|\left( \sum_{g \in \overline{G}_j}g \widetilde{a}-  \widetilde{a}\right) \widetilde{b} \right\|_2+ \frac{\eps}{2}= \left\|\left(  \sum_{g \in \overline{G}_j\backslash \{0\}}g \widetilde{a}  \right) \widetilde{b} \right\|_2 + \frac{\eps}{2}< \eps.
\end{split}
\ee
Above equation means that there is the strong limit \eqref{mp_ta_eqn}.

\end{proof}
\begin{cor}\label{mp_strong_cor}
Any $\widetilde{a} \in \SS\left( \R^{2N}_\th\right)$ lies in $\widehat{C\left( \T^{2N}_\th\right)}''$. 
\end{cor}
\begin{proof}
There is a strong limit \eqref{mp_aj_eqn}, i.e. $\widetilde{a} = \lim_{j \to \infty} a_j$. For any $j \in \N$ one has $a_j \in \widehat{C\left( \T^{2N}_\th\right)}$ it turns out $\widetilde{a} = \lim_{j \to \infty} a_j \in \widehat{C\left( \T^{2N}_\th\right)}''$.
\end{proof}

\subsubsection{Equivariant representation}

\paragraph*{}

Denote by $\left\{U^{\theta/m_j}_{k}\in U\left( C\left(\T^{2N}_{\th/m_{j}}\right)\right) \right\}_{k \in \Z^{2N}}$ the basis of $C\left(\T^{2N}_{\th/m_{j}}\right)$. Similarly to \eqref{nt_l2r_eqn} there is the representation $\pi_j: C\left(\T^{2N}_{\th/m_{j}}\right) \to B\left(L^2\left( \R^{2N}\right)  \right)$ given by
$$
\pi_j\left(U^{\theta/m_j}_{k} \right) = \exp\left(2\pi i\frac{k}{m_j}~\cdot\right).
$$
There is a following commutative diagram.
\newline
\begin{tikzpicture}
\matrix (m) [matrix of math nodes,row sep=3em,column sep=4em,minimum width=2em]
{
	C\left(\T^{2N}_{\th/m_{j}}\right)	&   & C\left(\T^{2N}_{\th/m_{j+1}}\right) \\
	& B\left(L^2\left(\R^{2N}\right)\right) &  \\};
\path[-stealth]
(m-1-1) edge node [left] {} (m-1-3)
(m-1-1) edge node [right] {$~\pi_j$} (m-2-2)
(m-1-3) edge node [above] {$\pi_{j+1}$}  (m-2-2);

\end{tikzpicture}
\newline
This diagram defines a faithful representation $\widehat{\pi}:\widehat{C\left(\mathbb{T}^{2N}_{\th}\right) } \to  B\left(L^2\left(\R^{2N}\right)\right)$. There is the action of $\Z^{2N} \times \R^{2N} \to \R^{2N}$ given by
$$
\left(k, x \right) \mapsto k + x.
$$
The action naturally induces the action of $\Z^{2N}$ on both $L^2\left(\R^{2N}\right)$ and $B\left(L^2\left(\R^{2N}\right)\right)$. Otherwise the action of $\Z^{2N}$ on  $B\left(L^2\left(\R^{2N}\right)\right)$ induces the action  of $\Z^{2N}$ on $\widehat{C\left(\mathbb{T}^{2N}_{\th}\right) }$. There is the following commutative diagram
\newline
\begin{tikzpicture}
\matrix (m) [matrix of math nodes,row sep=3em,column sep=4em,minimum width=2em]
{
	\Z^{2N}	&   & G\left(\widehat{C\left(\T^{2N}_\th\right)}~|~ C\left(\mathbb{T}^{2N}_{\th}\right)\right) \\
	& G_j = G\left(C\left(\T^{2N}_{\th/m_{j}}\right)~|~ C\left(\mathbb{T}^{2N}_{\th}\right)\right) \approx \Z^{2N}_{m_j}  &  \\};
\path[-stealth]
(m-1-1) edge node [left] {} (m-1-3)
(m-1-1) edge node [right] {} (m-2-2)
(m-1-3) edge node [above] {}  (m-2-2);

\end{tikzpicture}
\newline
From the above diagram it follows that there is the natural homomorphism $\Z^{2N} \hookto \widehat{G}$, and 
$\Z^{2N}$ is a normal subgroup.   Let  $J \subset \widehat{G}$  be a set of representatives of $\widehat{G}/\Z^{2N}$, and suppose that $0_{\widehat{G}} \in J$. 
Any $g \in \widehat{G}$ can be uniquely represented as $g = g_J + g_\Z$ where  $g \in J$, $g_\Z \in \Z^{2N}$.  For any $g_1, g_2 \in \widehat{G}$ denote by $\Phi_J\left(g_1, g_2 \right) \in J$, $\Phi_\Z\left(g_1, g_2 \right) \in \Z^{2N}$, such that
$$
g_1 + g_2 = \Phi_J\left(g_1, g_2 \right) + \Phi_\Z\left(g_1, g_2 \right).
$$
Let us define an action of $\widehat{G}$ on  $\bigoplus_{g \in J}L^2\left(\R^{2N}\right)$ given by
$$
g_1  \left(0,..., \underbrace{ \xi}_{g_2^{\text{th}}-\text{place}},...,0,... \right) = \left(0,..., \underbrace{\Phi_\Z\left(g_1, g_2 \right)  \xi}_{\Phi_J\left(g_1 + g_2 \right) ^{\text{th}}-\text{place}},..., 0, ... \right).
$$
Let $X \subset \bigoplus_{g \in J}L^2\left(\R^{2N}\right)$ be given by
$$
X = \left\{ \eta  \in \bigoplus_{g \in J}L^2\left(\R^{2N}\right)~|~   \eta=\left(0,..., \underbrace{ \xi}_{0_{\widehat{G}}^{\text{th}}-\text{place}},...,0,... \right) \right\}.
$$
Taking into account that $X \approx L^2\left(\R^{2N}\right)$, we will write $L^2\left(\R^{2N}\right) \subset \bigoplus_{g \in J}L^2\left(\R^{2N}\right)$ instead of $X \subset\bigoplus_{g \in J}L^2\left(\R^{2N}\right)$.
This inclusion and the action of $\widehat{G}$ on $\bigoplus_{g \in J}L^2\left(\R^{2N}\right)$ enable us write $\bigoplus_{g \in J}gL^2\left(\R^{2N}\right)$ instead of $\bigoplus_{g \in J}L^2\left(\R^{2N}\right)$. If $\widehat{\pi}^\oplus: \widehat{C\left(\mathbb{T}^{2N}_{\th}\right) } \to B\left( \bigoplus_{g \in J}gL^2\left(\R^{2N}\right)\right)   $ is given by
$$
\widehat{\pi}^\oplus\left( a\right)\left(g\xi \right) =  g\left( \widehat{\pi}\left(g^{-1}a \right) \xi\right); ~ \forall a \in  \widehat{C\left(\mathbb{T}^{2N}_{\th}\right) }, ~ \forall g \in J,~ \forall \xi \in L^2\left(\R^{2N} \right) 
$$
then $\widehat{\pi}^\oplus$ is an equivariant representation.

\subsubsection{Inverse noncommutative limit}\label{nt_inv_lim_sec}
\paragraph{}
If $\widetilde{a} \in \SS\left( \R^{2N}_\th\right)$ then from the Corollary \ref{mp_strong_cor} it turns out
$\widetilde a \in \widehat{C\left( \T^{2N}_\th\right) }''$. Since $\widehat{\pi}^\oplus$ is a faithful representation of $\widehat{C\left( \T^{2N}_\th\right) }$, one has  an injective homomorphism  $\SS\left( \R^{2N}_\th\right) \hookto \widehat{\pi}^\oplus\left( \widehat{C\left( \T^{2N}_\th\right) }\right) ''$ of involutive algebras. 

For any $\widetilde{a} \in \SS\left( \R^{2N}_\th\right)$ following condition holds
$$
\sum_{g \in \ker\left(\widehat{G} \to G_j\right)}g\widehat{\pi}^\oplus\left(\widetilde{a} \right) = \sum_{g' \in J}g'\left( \sum_{ g''\in \ker\left(\Z^{2N} \to G_j\right)}  g''\widehat{\pi}\left(\widetilde{a} \right)\right)= \sum_{g \in J}gP. 
$$
where 
$$
P = \sum_{ g\in \ker\left(\Z^{2N} \to G_j\right)}  g\widehat{\pi}\left(\widetilde{a} \right).
$$
 
If $J \subset \widehat{G}$ is a set of  representatives of $\widehat{G}/\Z^{2N}$ and $g',g'' \in J$ are such that $g' \neq g''$ then operators $g'P$, $g''P$ act on mutually orthogonal Hilbert subspaces $g'L^2\left( 
R^{2N}\right)$, $g''L^2\left( 
R^{2N}\right)$ of the direct sum $\bigoplus_{g \in J} gL^2\left( 
R^{2N}\right)$, and 
 taking into account $\left\|P\right\|=\left\|gP\right\|$  one has
\begin{equation}\label{nt_norm_equ_eqn}
\left\|\sum_{g \in \ker\left(\widehat{G} \to G_j\right)}\widehat{\pi}^\oplus\left( \widetilde{a} \right)\right\|= \left\|\sum_{g \in J}gP\right\|=\left\|P\right\|=\left\|\sum_{g \in \ker\left(\Z^{2N} \to G_j\right)}\widehat{\pi}\left( \widetilde{a} \right)\right\|.
\end{equation}

\begin{lem}\label{nt_long_delta_lem}
	
	Let $a \in \sS\left(\mathbb{R}^{2N}_\th\right)$, and let  $a_\Delta\in \sS\left(\mathbb{R}^{2N}_\th\right)$ be given by
	\begin{equation}\label{nt_a_delta_eqn}
	a_\Delta\left(x\right)= a(x + \Delta);~\forall x \in \mathbb{R}^{2N}
	\end{equation}
	where $\Delta \in \mathbb{R}^{2N}$.
	For any $m \in \mathbb{N}$ there is  a dependent on $a$ real constant $C_m > 0$  such  that for any $j \in\mathbb{N}$ following condition holds
	$$
	\left\|\sum_{g \in \ker\left(\widehat{G} \to G_j\right)}\widehat{\pi}^\oplus\left(a_{\Delta} a\right)  \right\|  \le \frac{C_m}{\left\| \Delta\right\|^m}.
	$$
	
\end{lem}

\begin{proof}
	From \eqref{nt_norm_equ_eqn} it follows that
	$$
	\left\|\sum_{g \in \ker\left(\widehat{G} \to G_j\right)}\widehat{\pi}^\oplus\left( a_{\Delta} a\right)\right\|=\left\|\sum_{g \in \ker\left(\Z^{2N} \to G_j\right)}\widehat{\pi}\left( a_{\Delta} a\right)\right\|.
	$$	
 From \eqref{nt_c_f_m_eqn} it follows that for any $f \in \sS\left( \mathbb{R}^{2N} \right)$ and any $m \in \mathbb{N}$ there is $C^f_m$ such that
	\begin{equation*}
	\left|f \left(u\right)\right|<\frac{C^f_m}{\left( 1 + \left\|u\right\|\right)^m }.
	\end{equation*}
	Let $M = 2N + 1 + m$. From \eqref{nt_norm_estimation} and \eqref{mp_fourier_torus_eqn} it follows that
	$$
	\left\|\sum_{g \in \ker\left(\Z^{2N} \to \Z^{2N}_{m_j} \right)} g\widehat{\pi}\left(a_\Delta a\right) \right\| \le\frac{1}{m_j^{2N}}~~\sum_{l \in \mathbb{Z}^{2N}} \left| \mathcal{F}\left(a_{\Delta}  a\right)\left(\frac{l}{m_j} \right) \right|.
	$$
	Otherwise from \eqref{mp_fourier_eqn} it follows that
	$$
	\mathcal{F}\left(a_{\Delta}  a\right)\left(x \right)  =    	\int_{\R^{2N}}\mathcal{F}a_\Delta\left(x-y \right) \mathcal{F}a\left(y\right)e^{ \pi i  y \cdot \Th x }~dy.
	$$
	From the above equations it turns out
	\begin{equation*}
	\begin{split}
	\frac{1}{m_j^{2N}}~~\sum_{l \in \mathbb{Z}^{2N}} \left| \mathcal{F}\left(a_{\Delta}  a\right)\left(\frac{l}{m_j} \right) \right| = 
	\\
	=\frac{1}{m_j^{2N}}\sum_{l \in \mathbb{Z}^{2N}} \left|\int \mathcal{F}a\left(\frac{l}{m_j} + \Delta - t  \right)  \mathcal{F}a\left(t \right)e^{\frac{\pi i l}{m_j} \cdot \Th t}  dt\right|\le
	\\ \le \frac{1}{m_j^{2N}}\sum_{l \in \mathbb{Z}^{2N}} \int \left| b\left( t - \Delta - \frac{l}{m_j} \right)  c\left(t \right)e^{\frac{\pi i l}{m_j} \cdot \Th t}  \right|dt=\\
	=\frac{1}{m_j^{2N}}~~\sum_{l \in \mathbb{Z}^{2N}, ~\left\| \frac{l}{m_j}\right\|\le \frac{\left\| \Delta\right\|}{2 }}~ \int \left| b\left( t - \Delta - \frac{l}{m_j} \right)  c\left(t \right)e^{\frac{\pi i l}{m_j} \cdot \Th t}  \right|dt +
	\\
	+
	\frac{1}{m_j^{2N}}~~\sum_{l \in \mathbb{Z}^{2N}, ~\left\| \frac{l}{m_j}\right\|> \frac{\left\| \Delta\right\|}{2 }}~ \int \left| b\left( t - \Delta - \frac{l}{m_j} \right)  c\left(t \right)e^{\frac{\pi i l}{m_j} \cdot \Th t}  \right|dt
	\end{split}
	\end{equation*}
	
	where $b\left( u\right) = \mathcal{F}a\left(-u \right)$, $c\left( u\right) = \mathcal{F}a\left(u \right)$.  From \eqref{nt_c_f_m_eqn} it turns out
	\begin{equation*}
	\begin{split}
	\frac{1}{m_j^{2N}}~~\sum_{l \in \mathbb{Z}^{2N}, ~\left\| \frac{l}{m_j}\right\|\le \frac{\left\| \Delta\right\|}{2 }}~ \int \left| b\left( t - \Delta - \frac{l}{m_j} \right)  c\left(t \right)e^{\frac{\pi i l}{m_j} \cdot \Th t}  \right|dt \le
	\\
	\le \frac{1}{m_j^{2N}}~~\sum_{l \in \mathbb{Z}^{2N}, ~\left\| \frac{l}{m_j}\right\|\le \frac{\left\| \Delta\right\|}{2 }}~ \int_{\mathbb{R}^{2N}}\frac{C^{b}_{M}}{\left(1 + \left\|t- \Delta - \frac{l}{m_j}\right\|  \right)^{M}}~\frac{C^{c}_{2M}}{\left(1 + \left\|t\right\|  \right)^{2M}} dt  =
	\\
	= \frac{1}{m_j^{2N}}~~\sum_{l \in \mathbb{Z}^{2N}, ~\left\| \frac{l}{m_j}\right\|\le \frac{\left\| \Delta\right\|}{2 }}~ \int_{\mathbb{R}^{2N}}\frac{C^{b}_{M}}{\left(1 + \left\|t - \Delta - \frac{l}{m_j}\right\|  \right)^{M}\left(1 + \left\|t\right\|  \right)^{M} }~\frac{C^{c}_{2M}}{\left(1 + \left\|t\right\|  \right)^{M}} dt  \le
	\\
	\le \frac{N^\Delta_{m_j}}{m_j^{2N}}~~\sup_{l \in \mathbb{Z}^{2N}, ~\left\| \frac{l}{m_j}\right\|\le \frac{\left\| \Delta\right\|}{2 },~s\in \mathbb{R}^{2N}}~ \frac{C^{b}_{M}C^{c}_{2M}}{\left(1 + \left\|s - \Delta - \frac{l}{m_j}\right\|  \right)^{M} \left(1 + \left\|s\right\|  \right)^{M}}~\times
	\\
	\times  \int_{\mathbb{R}^{2N}}\frac{1}{\left(1 + \left\|t\right\|  \right)^{M}} dt.
	\end{split}
	\end{equation*}
	where $N^\Delta_{m_j} = \left|\left\{l \in \mathbb{Z}^{2N}~| ~\left\| \frac{l}{m_j}\right\|\le \frac{\left\| \Delta\right\|}{2 }\right\}\right|$. The number $N^\Delta_{m_j}$ can be estimated as a number of points with integer coordinates inside $2N$-dimensional cube
	$$
	N^\Delta_{m_j} < \left\|m_j \Delta\right\|^{2N}.
	$$  
	
	From $M > 2N $ it turns out the integral $ \int_{\mathbb{R}^{2N}}\frac{1}{\left(1 + \left\|t\right\|  \right)^{M}} dt$ is convergent, hence
	
	$$
	\frac{1}{m_j^{2N}}~\sum_{l \in \mathbb{Z}^{2N}, ~\left\| \frac{l}{m_j}\right\|\le \frac{\left\| \Delta\right\|}{2 }}~~\left|\int b\left( t - \Delta - \frac{l}{m_j} \right)  c\left(t \right)e^{\frac{\pi i l}{m_j} \cdot \Th t}  dt\right| \le
	$$
	$$
	\le C_1'  \sup_{l \in \mathbb{Z}^{2N}, ~\left\| \frac{l}{m_j}\right\|\le \frac{\left\| \Delta\right\|}{2 },~s\in \mathbb{R}^{2N}}~  \frac{ \left\|\Delta\right\|^{2N} }{\left(1 + \left\|s - \Delta - \frac{l}{m_j}\right\|  \right)^{M} \left(1 + \left\|s\right\|  \right)^{M}}
	$$
	where 
	$$
	C_1' =  C^{b}_MC^c_{2M}\int_{\mathbb{R}^{2N}}\frac{1}{\left(1 + \left\|t\right\|  \right)^{M}} dt.
	$$

	It turns out from the \eqref{nt_triangle_eqn} that
	$$
	\mathrm{inf}_{l \in \mathbb{Z}^{2N}, ~\left\| \frac{l}{m_j}\right\|\le \frac{\left\| \Delta\right\|}{2 },~s\in \mathbb{R}^{2N}} \left(1 + \left\|s - \Delta - \frac{l}{m_j}\right\|  \right)^{M} \left(1 + \left\|s\right\|  \right)^{M} >\left\|\frac{\Delta}{4}\right\|^M.
	$$
	From $M = 2N+1 + m$    it turns out
	$$
	\frac{1}{m_j^{2N}}~\sum_{l \in \mathbb{Z}^{2N}, ~\left\| \frac{l}{m_j}\right\|\le \frac{\left\| \Delta\right\|}{2 }}~~\left|\int b\left( t - \Delta - \frac{l}{m_j} \right)  c\left(t \right)e^{\frac{\pi i l}{m_j} \cdot \Th t}  dt\right| \le \frac{C_1}{\left\| \Delta\right\|^{m}} 
	$$  
	where $C_1=C'_1/4^{m}$. Clearly 
	\begin{equation*}
	\begin{split}
	\frac{1}{m_j^{2N}}~\sum_{l \in \mathbb{Z}^{2N}, ~\left\| \frac{l}{m_j}\right\|> \frac{\left\| \Delta\right\|}{2 }}~~\left|\int b\left( t - \Delta - \frac{ l}{m_j} \right)  c\left(t \right)e^{\frac{\pi i l}{m_j} \cdot \Th t}  dt\right|=
	\\=
	\frac{1}{m_j^{2N}}~~\sum_{l \in \mathbb{Z}^{2N}, ~\left\| \frac{ l}{m_j}\right\|> \frac{\left\| \Delta\right\|}{2 }}~ \left| \left(  b\left( \bullet- \Delta-\frac{ l}{m_j}\right) ,~ c\left(\bullet\right)e^{\frac{\pi i l}{m_j}\cdot \Th\bullet} \right)  \right|
	\end{split}
	\end{equation*}
	where $\left( \cdot, \cdot \right)$ means the given by \eqref{fourier_scalar_product_eqn} scalar product.
	From  the $\mathcal{F}$-invariance of $\left( \cdot, \cdot \right) $  it follows that 
	\begin{equation*}
	\begin{split}
	\frac{1}{m_j^{2N}}~~\sum_{l \in \mathbb{Z}^{2N}, ~\left\| \frac{l}{m_j}\right\|> \frac{\left\| \Delta\right\|}{2 }}~ \left|  \left(  b\left( \bullet- \Delta-\frac{ l}{m_j}\right) ,~ c\left(\bullet\right)e^{\frac{\pi i l}{m_j}\cdot \Th\bullet} \right)  \right|=
	\\=
	\frac{1}{m_j^{2N}}~~\sum_{l \in \mathbb{Z}^{2N}, ~\left\| \frac{l}{m_j}\right\|> \frac{\left\| \Delta\right\|}{2 }}~ \left|\left(  \mathcal{F}\left(  b\left( \bullet+ \Delta-\frac{ l}{m_j}\right)\right) ,~ \mathcal{F}\left( c\left(\bullet\right)e^{\frac{\pi i l}{m_j}\cdot \Th\bullet} \right)  \right)  \right|=
	\\
	= \frac{1}{m_j^{2N}}~~\sum_{l \in \mathbb{Z}^{2N}, ~\left\| \frac{l}{m_j}\right\|> \frac{\left\| \Delta\right\|}{2 }}~ \left|\int_{\mathbb{R}^{2N}} 
	\mathcal{F}\left( b\right)\left( \bullet- \Delta -\frac{ l}{m_j}\right)\left(u\right)\mathcal{F}\left(c\left(\bullet\right)e^{\frac{\pi i l}{m_j}\cdot \Th\bullet} \right)\left(u\right)  du \right|\le
	\\
		\end{split}
	\end{equation*}
	\begin{equation*}
\begin{split}	
	\le \frac{1}{m_j^{2N}}~~\sum_{l \in \mathbb{Z}^{2N}, ~\left\| \frac{l}{m_j}\right\|> \frac{\left\| \Delta\right\|}{2 }}~ \int_{\mathbb{R}^{2N}}\left| e^{-i\left(\Delta - \frac{l}{m_j}\right) \cdot u}\mathcal{F}\left( b\right)\left( u\right) \mathcal{F}\left(c\right)\left(u+\Th\frac{\pi l}{m_j}\right)\right| du \le
	\\
	\le \frac{1}{m_j^{2N}}~~\sum_{l \in \mathbb{Z}^{2N}, ~\left\| \frac{l}{m_j}\right\|> \frac{\left\| \Delta\right\|}{2 }}~ \int_{\mathbb{R}^{2N}}\frac{C^{\mathcal{F}\left( b\right)}_{3M}}{\left(1 + \left\|u\right\|  \right)^{3M}}\frac{C^{\mathcal{F}\left(c\right)}_{2M}}{\left(1 + \left\|u-\Th\frac{\pi l}{m_j}\right\|  \right)^{2M}} du  \le
	\\
	\le \frac{1}{m_j^{2N}}~~\sup_{l \in \mathbb{Z}^{2N}, ~\left\| \frac{l}{m_j}\right\|> \frac{\left\| \Delta\right\|}{2 },~ s \in \mathbb{R}^{2N}}~ \frac{C^{\mathcal{F}\left( b\right)}_{3M}}{\left(1 + \left\|s\right\|  \right)^{M}}\frac{C^{\mathcal{F}\left(c\right)}_{2M}}{\left(1 + \left\|s-\Th\frac{\pi l}{m_j}\right\|  \right)^{M}}\frac{1}{\left(1 + \left\|u-\Th\frac{\pi l}{m_j}\right\|  \right)^{M}\left(1 + \left\|u\right\|  \right)^{M}}
	\times\\
	\times \sum_{l \in \mathbb{Z}^{2N}, ~\left\| \frac{l}{m_j}\right\|> \frac{\left\| \Delta\right\|}{2 }}     
	\int_{\mathbb{R}^{2N}}\frac{1}{\left(1 + \left\|u\right\|  \right)^{M}} du.
	\end{split}
	\end{equation*}
	Since we consider the asymptotic dependence $\left\|\Delta\right\|\to \infty$ only  large values of $\left\|\Delta\right\|$ are interesting, so we can suppose that  $\left\|\Delta\right\| > 2$.
	If $\left\|\Delta\right\| > 2$ then from $\left\|\frac{l}{m_j}\right\|> \frac{\left\| \Delta\right\|}{2 }$ it follows that $\left\|\Th\frac{\pi l}{m_j}\right\| > 1$, and  from \eqref{nt_triangle_eqn} it follows that
	\begin{equation*}
	\begin{split}
	\left(1 + \left\|u\right\|  \right)^{M}\left(1 + \left\|u-\Th\frac{\pi l}{m_j}\right\|  \right)^{M} > \left\|\Th\frac{\pi l}{m_j}\right\|^M, 
	\\
	\inf_{l \in \mathbb{Z}^{2N}, ~\left\| \frac{l}{m_j}\right\|> \frac{\left\| \Delta\right\|}{2 },~ s \in \mathbb{R}^{2N}}~ \left(1 + \left\|s\right\|  \right)^{M}\left(1 + \left\|s-\Th\frac{\pi l}{m_j}\right\|  \right)^{M} > \left\|\Th\frac{\Delta}{4}\right\|^M,
	\end{split}
	\end{equation*}
	
	hence, taking into account \eqref{mp_2x_eqn}, one has
	\begin{equation*}
	\begin{split}
	\frac{1}{m_j^{2N}}~~\sum_{l \in \mathbb{Z}^{2N}, ~\left\| \frac{l}{m_j}\right\|> \frac{\left\| \Delta\right\|}{2 }}~ \left|  \left(  b\left( \bullet+ \Delta-\frac{l}{m_j}\right) ,~ c\left(\bullet\right)e^{\frac{\pi i l}{m_j}\cdot \Th\bullet} \right)  \right|\le
	\\
	\le \frac{1}{m_j^{2N}} \frac{C_2'}{\left\| \Delta\right\|^M} \sum_{l \in \mathbb{Z}^{2N}, ~\left\| \frac{l}{m_j}\right\|> 1}\int_{\R^{2N}}\frac{1}{\left\|\frac{2\pi l}{m_j}\right\|^M}\frac{1}{\left(1+\left\|u\right\| \right)^M }=
	\\
	\frac{C_2'}{\left\| \Delta\right\|^M}  \frac{1}{m_j^{2N}} \left(  \sum_{l \in \mathbb{Z}^{2N}, ~\left\| \frac{l}{m_j}\right\|> 1}~ \frac{1}{\left\|\frac{2\pi l}{m_j}\right\|^M}  \right) \left(  \int_{\R^{2N}}\frac{1}{\left(1+\left\|u\right\| \right)^M }du\right),
	\end{split}
	\end{equation*}
	where $C'_2=C^{\mathcal{F}\left( b\right)}_{3m}C^{\mathcal{F}\left(c\right)}_{2m}$.
Since $M \ge 2N + 1$ the integral  $\int_{\mathbb{R}^{2N}}\frac{1}{\left(1 + \left\|u\right\|  \right)^{M}} du$ is convergent. The infinite sum in the above equation can be represented as an integral of step function, in particular  following condition holds
	$$
	\frac{1}{m_j^{2N}} \sum_{l \in \mathbb{Z}^{2N},\left\| \frac{l}{m_j}\right\|> 1}~ \frac{1}{\left\|\frac{\pi \Th l}{m_j}\right\|^M} = \int_{\R^{2N} - \left\{x \in \R^{2N}~|~\left\|x\right\|> 1 \right\}} f_{m_j}\left( x\right) dx
	$$
	where $f_{m_j}$ is a multidimensional step function such that
	$$
	f_{m_j}\left(\frac{2\pi  l}{m_j} \right) = \frac{1}{\left\|\frac{2\pi  l}{m_j}\right\|^M}
	$$
	From 
	$$
	f_{m_j}\left(x\right)  < \frac{2}{\left\|\frac{2\pi h x}{m_j}\right\|^M}
	$$ 
	it follows that
	$$
	\int_{\R^{2N} - \left\{x \in \R^{2N}~|~\left\|x\right\|> 1 \right\}} f_{m_j}\left( x \right)dx < \int_{\R^{2N} - \left\{x \in \R^{2N}~|~\left\|x\right\|> 1 \right\}} \frac{2}{\left\|2\pi x\right\|^M}dx.
	$$
	From $m > 2N+1$ it turns out the integral $$\int_{\R^{2N} - \left\{x \in \R^{2N}~|~\left\|x\right\|> 1 \right\}} \int_{\R^{2N} - \left\{x \in \R^{2N}~|~\left\|x\right\|> 1 \right\}} \frac{2}{\left\|2\pi  x\right\|^m}dx$$ is convergent, hence
	$$
	\frac{1}{m_j^{2N}} \sum_{l \in \mathbb{Z}^{2N},\left\| \frac{l}{m_j}\right\|> 1}~ \frac{1}{\left\|J\frac{l}{m_j}\right\|^m} < C''_2= \int_{\R^{2N} - \left\{x \in \R^{2N}~|~\left\|x\right\|> 1 \right\}} \frac{2}{\left\|2\pi x\right\|^M}dx.
	$$
	From above equations it follows that
	$$
	\frac{1}{m_j^{2N}}~~\sum_{l \in \mathbb{Z}^{2N}, ~\left\| \frac{l}{m_j}\right\|> \frac{\left\| \Delta\right\|}{2 }}~ \left|\int_{\mathbb{R}^{2N}} b\left( w-\frac{l}{m_j}\right)c\left(w+\Delta\right)e^{i\frac{l}{m_j}\cdot Jw} dt \right| \le \frac{C_2}{\left\| \Delta\right\|^m} 
	$$  
	where $M = 2N+1 + m$ and $C_2 = C'_2C''_2\int_{\R^{2N}}\frac{1}{\left(1+\left\|u\right\| \right)^M }du$. In result for any $m > 0$ there is $C_m \in \mathbb{R}$ such that 
	
	$$
	\left\|\sum_{g \in \ker\left(\widehat{G} \to G_j\right)}\widehat{\pi}^\oplus\left( a_{\Delta} a\right)\right\| <\frac{1}{m_j^{2N}}\sum_{l \in \mathbb{Z}^{2N}} \int \left| b\left( t + \Delta - \frac{l}{m_j} \right)  c\left(t \right)e^{\frac{\pi i l}{m_j} \cdot \Th t}  \right|dt < \frac{C_m}{\left\| \Delta\right\|^m}.
	$$
\end{proof}

\begin{lem}\label{nt_w_spec_lem_p} 
	If $\overline{a}$ in $ \sS\left(\mathbb{R}^{2N}_\th \right) $ is positive then following conditions hold:
		\begin{enumerate}	
		\item[(i)]
		For any $j \in \mathbb{N}^0$ the following  series 
		\begin{equation*}
		\begin{split}
		a_j = \sum_{g \in \ker\left( \widehat{G} \to  G_j\right) )} g \overline{a},\\
		b_j = \sum_{g \in \ker\left( \widehat{G} \to  G_j\right) } g  \overline{a}^2
		\end{split}
		\end{equation*}
		are strongly convergent and the sums lie in   $\Coo\left( \T^{2N}_{\th/m_j^2}\right) $, i.e. $a_j, b_j \in \Coo\left( \T^{2N}_{\th/m_j^2}\right) $; 
		\item[(ii)] For any $\eps > 0$ there is $N \in \N$ such that for any $j \ge N$ the following condition holds
		\begin{equation*}
		\begin{split}
		\left\| a_j^2 - b_j\right\| < \eps.
		\end{split}
		\end{equation*}
	\end{enumerate}

\end{lem}
\begin{proof}
	(i) Follows from the Lemmas \ref{mp_weak_lem} and/or \ref{mp_strong_lem}.
	\newline
	(ii)
	Denote by $J_j = \ker\left( \Z^{2N} \to G_j\right)= m_j\Z^{2N}$. If
	\begin{equation*}
	\begin{split}
	a_j =  \sum_{g \in J_j  }g \overline{a},\\ b_j =  \sum_{g \in J_j  }  g \overline{a}^2
	\end{split}
	\end{equation*}
	then 
	\begin{equation}\label{nt_an_bn_eqn}
	a^2_j - b_j = \sum_{g \in J_j  }g\overline{a} ~ \left(  \sum_{g' \in J_j  \backslash \{g\} }g'' \overline{a}\right) .
	\end{equation}
	From \eqref{nt_a_delta_eqn}  it follows that $g \overline{a}= \overline{a}_{g }$ where  $\overline{a}_{g}\left(x \right)  = \overline{a}\left( x +g\right)$ for any $x \in \R^{2N}$ and $g \in \Z^{2N}$. Hence the equation \eqref{nt_an_bn_eqn} is equivalent to
	\begin{equation*}
	\begin{split}
	a^2_j - b_j = \sum_{g \in J_j  }\overline{a}_{g }  \sum_{g' \in J_j  \backslash \{g\} } \overline{a}_{g' } = \sum_{g \in \mathbb{Z}^{2N}  }\overline{a}_{m_jg }  \sum_{g' \in \mathbb{Z}^{2N}  \backslash \{g\} } \overline{a}_{m_jg' }=
	\\
	=\sum_{g' \in J_j}  g'\left( \overline{a}  \sum_{g \in \mathbb{Z}^{2N}  \backslash \{0\} } \overline{a}_{m_jg } \right). 
	\end{split}
	\end{equation*}
	Let $m > 1$ and $M = 2N+1 + m$.
	From the Lemma \ref{nt_long_delta_lem} it follows that there is $C \in \mathbb{R}$ such that
	$$
	\left\|\sum_{g \in J_j }g\left( aa_\Delta\right) \right\| < \frac{C}{\left\| \Delta\right\|^M}.
	$$
	From the triangle inequality it follows that
	\begin{equation*}
	\begin{split}	\left\|a^2_j - b_j\right\|=\left\|\sum_{g \in J_j }g\left( \overline{a}  \sum_{g' \in \mathbb{Z}^{2N}  \backslash \{0\} } \overline{a}_{m_jg' } \right)\right\| \le
	\\
	\le \sum_{g' \in \mathbb{Z}^{2N}  \backslash \{0\} } \left\|\sum_{g \in J_j }g\left( \overline{a}  ~ \overline{a}_{m_jg' } \right)\right\|   \le \sum_{g \in \mathbb{Z}^{2N}  \backslash \{0\} } \frac{C}{\left\| m_jg' \right\|^M }.
	\end{split}
	\end{equation*}
	
	From  $M> 2N$ it turns out that the series
	$$
	C' = \sum_{g' \in \mathbb{Z}^{2N}  \backslash \{0\} }  \frac{C}{\left\| g' \right\|^M }
	$$
	is convergent and 
	$$
	\sum_{g \in \mathbb{Z}^{2N}  \backslash \{0\} } \frac{C}{\left\| m_jg \right\|^M } = \frac{C'}{m_j^M }.
	$$
	If $\varepsilon > 0$ is a small number and $N\in \mathbb{N}$ is such $m_N > \sqrt[M]{\frac{C'}{\varepsilon}}$ then from above equations it follows that for any $j \ge N$ the following condition holds
	$$
	\left\|a^2_j - b_j\right\| < \eps.
	$$

\end{proof}

\begin{lem}\label{nt_w_spec_lem}
	Let  us consider  a dense inclusion
	$$
	\underbrace{\SS\left(\R^{2}_\th\right)\otimes\dots\otimes\SS\left(\R^{2}_\th\right)}_{N-\mathrm{times}} \subset \SS\left(\R^{2N}_\th\right)
	$$
	of algebraic tensor product which follows from  \eqref{mp_tensor_prod}.
If $\overline{a} \in \SS\left(\R^{2N}_\th\right) $ is a positive such that
	\begin{itemize}
		\item 	$$
		\overline{a} = \sum_{\substack{j=0\\k=0}}^{M_1} c^1_{jk}f_{jk}\otimes \dots \otimes \sum_{\substack{j=0\\k=0}}^{M_N} c^N_{jk}f_{jk}
		$$
		where $c^l_{jk} \in \C$ and $f_{jk}$ are given by the Lemma \ref{lm:osc-basis},
		\item For any $l = 1,\dots, N$ the sum
		$
		\sum_{\substack{j=0\\k=0}}^{M_l} c^l_{jk}f_{jk}
		$
	is a rank-one operator.
	\end{itemize}
	 then $\overline{a}$ is special. 
\end{lem}
\begin{proof}
	Clearly $\overline{a}$ is a rank-one operator.
	If $\overline{a} \in \SS\left(\R^{2N}_\th\right) $ then from the  Lemmas \ref{mp_weak_lem} and/or \ref{mp_strong_lem} it turns out that $\overline{a}$ satisfies to (a) of the Definition \ref{special_el_defn}.
	If $z \in C\left(\T^{2N}_\th \right)$ then $z$ then from the injective *-homomorphism $C\left( \T^{2N}_{\th}\right) \hookto C\left( \T^{2N}_{\th/m_j^2}\right)$ it follows that $z$ can be regarded as element of $C\left( \T^{2N}_{\th/m_j^2}\right)$, i.e. $z\in C\left( \T^{2N}_{\th/m_j^2}\right)$. Denote by
			\begin{equation*}
	\begin{split}
	b_j=	\sum_{g \in \ker\left(\widehat{G} \to G_j \right)} g\left( 	z\overline{a}z^*\right)= z\left( \sum_{g \in \ker\left(\widehat{G} \to G_j \right)} g\overline{a}\right) z^*,\\
	c_j =  \sum_{g \in \ker\left(\widehat{G} \to G_j \right)} g\left( 	z\overline{a}z^*\right)^2= z\left( \sum_{g \in \ker\left(\widehat{G} \to G_j \right)} g\left( \overline{a}z^*z\overline{a}\right) \right) z^*,\\
	d_j =  \sum_{g \in \ker\left(\widehat{G} \to G_j \right)} gf_\eps\left( 	z\overline{a}z^*\right)
	\end{split}
	\end{equation*}
	where $f_\eps$ is given by \eqref{f_eps_eqn}. From $\overline{   a}$ in $\SS\left(\R^{2N}_\th \right)$ it turns out
	$
	a_j=\sum_{g \in \ker\left(\widehat{G} \to G_j \right)} g\overline{a}\in C\left( \T^{2N}_{\th/m_j^2}\right) 
	$, hence $b_j = za_jz^*\in C\left( \T^{2N}_{\th/m_j^2}\right)$. If $\xi\in \H$ is eigenvector of $\overline{   a}$ such that  $\overline{   a}\xi=\left\|\overline{   a}\right\|\xi$ then $\eta=z\xi$ is an is eigenvector of $\eta=z\overline{   a}z^*$ such that  $z\overline{   a}z^*\eta=\left\|z\overline{   a}z^*\right\|\eta$. It follows that $\left( z\overline{a}z^*\right)^2  = kz\overline{a}z^*$ where $k \in \R_+$ is given by
	$$
	k = \frac{\left\|z\overline{   a}z^*\right\|^2}{\left\|z\overline{   a}z^*\right\|}.
	$$
	Hence $c_j = kb_j$ and $c_j \in \Coo\left( \T^{2N}_{\th/m_j^2}\right)$. Similarly $f_\eps\left( 	z\overline{a}z^*\right) = k'\left(  z\overline{a}z^*\right)$ where
	$$
k' = \frac{\max\left( 0, \left\|z\overline{   a}z^*\right\|-\eps\right) }{\left\|z\overline{   a}z^*\right\|}.
$$
Hence $d_j = k'b_j$ and $d_j \in \Coo\left( \T^{2N}_{\th/m_j^2}\right)$, it follows that $\overline{a}$ satisfies to the condition (b) of the Definition \ref{special_el_defn}.
Let $\eps >0$, and let $\delta > 0$ be such that
\begin{equation*}
\begin{split}
\delta^4 \left\|\sum_{g \in \widehat{G}} \overline{a}\right\|^2 + 2 \delta^2 \left\|\sum_{g \in \widehat{G}} \overline{a}\right\|\left\|\sum_{g \in \widehat{G}} z\overline{a}z^*\right\| <\frac{\eps}{4},\\
\left\|\sum_{g \in \widehat{G}} g\left( 	\overline{a}z^*z\overline{a}\right)  \right\| \delta^2   <\frac{\eps}{4},\\
\left( \left\|z\right\|+\delta\right)^2\left( \delta^2 + 2 \delta\left\|z\right\| \right) \left\|  \sum_{g \in \widehat{G}} g 	\overline{a}^2\right\|<\frac{\eps}{4}.
\end{split}
\end{equation*}
The algebra $\Coo\left( \T^{2N}_{\th}\right) $ is a dense subalgebra of $C\left( \T^{2N}_{\th}\right)$, so there is $y \in \Coo\left( \T^{2N}_{\th}\right)$ such that $\left\|z - y\right\|< \delta$.
From
			\begin{equation*}
\begin{split}
\left\| b_j-	y\left( \sum_{g \in \ker\left(\widehat{G} \to G_j \right)} g 	\overline{a}\right) y^*\right\| \le\left\| \left(z-y \right) \left( \sum_{g \in \widehat{G} } g 	\overline{a}\right) (z-y)^*\right\|<\delta^2\left\|\sum_{g \in \widehat{G}} \overline{a}\right\| 
\end{split}
\end{equation*}
and taking into account $\delta^4 \left\|\sum_{g \in \widehat{G}g} \overline{a}\right\|^2 + 2 \delta^2 \left\|\sum_{g \in \widehat{G}g} \overline{a}\right\|\left\|\sum_{g \in \widehat{G}g} z\overline{a}z^*\right\|<\frac{\eps}{4}$ one has
\begin{equation}\label{nt_eps_4_1}
\begin{split}
\left\| b^2_j-	\left( y\left( \sum_{g \in \ker\left(\widehat{G} \to G_j \right)} g 	\overline{a}\right) y^*\right)^2 \right\|<\frac{\eps}{4}.
\end{split}
\end{equation}
From
\begin{equation*}
\begin{split}
\left\| c_j -  y\left( \sum_{g \in \ker\left(\widehat{G} \to G_j \right)} g\left( 	\overline{a}z^*z\overline{a}\right)\right) y^*\right\|\le\left\| \left(z-y \right) \left( \sum_{g\in \widehat{G}} g\left( 	\overline{a}z^*z\overline{a}\right)\right) \left(z-y \right)^*\right\|<\\ < \delta^2\left\|\sum_{g\in \widehat{G}} g\left( 	\overline{a}z^*z\overline{a}\right)  \right\| 
\end{split}
\end{equation*}
and taking into account $\left\|\sum_{g \in \widehat{G}} g\left( 	\overline{a}z^*z\overline{a}\right)  \right\| \delta^2   <\frac{\eps}{4}$ one has
\begin{equation}\label{nt_eps_4_2}
\begin{split}
\left\| c_j -  y\left( \sum_{g \in \ker\left(\widehat{G} \to G_j \right)} g\left( 	\overline{a}z^*z\overline{a}\right)\right) y^*\right\|<\frac{\eps}{4}.
\end{split}
\end{equation}

From $\left\|y\right\| < \left\|z\right\|+ \delta$ it turns out

			\begin{equation*}
\begin{split}
\left\| y\left( \sum_{g \in \ker\left(\widehat{G} \to G_j \right)} g\left( 	\overline{a}z^*z\overline{a}\right)\right) y^*-y\left( \sum_{g \in \ker\left(\widehat{G} \to G_j \right)} g\left( 	\overline{a}y^*y\overline{a}\right)\right) y^*\right\|\le \\
\le\left( \left\|z\right\|+\delta\right)^2\left( \delta^2 + 2 \delta\left\|z\right\| \right) \left\|  \sum_{g \in \widehat{G}} g 	\overline{a}^2\right\| 
\end{split}
\end{equation*}
and taking into account $\left( \left\|z\right\|+\delta\right)^2\left( \delta^2 + 2 \delta\left\|z\right\| \right) \left\|  \sum_{g \in \widehat{G}} g 	\overline{a}^2\right\|<\frac{\eps}{4}$ one has
\begin{equation}\label{nt_eps_4_3}
\begin{split}
\left\| y\left( \sum_{g \in \ker\left(\widehat{G} \to G_j \right)} g\left( 	\overline{a}z^*z\overline{a}\right)\right) y^*-y\left( \sum_{g \in \ker\left(\widehat{G} \to G_j \right)} g\left( 	\overline{a}y^*y\overline{a}\right)\right) y^*\right\|<\frac{\eps}{4}.
\end{split}
\end{equation}
From $y \in \Coo\left(\T^{2N}_\th \right)$ and $\overline{a} \in \SS\left(\R^{2N}_\th \right)$ it follows that $y\overline{a}y^*\in  \SS\left(\R^{2N}_\th \right)$, hence from the Lemma \ref{nt_w_spec_lem_p} it turns out the existence of $N \in \N$ such that for any $j \ge \N$ following condition holds 
\begin{equation}\label{nt_eps_4_4}
\begin{split}
\left\| \left( \sum_{g \in \ker\left(\widehat{G} \to G_j \right)} g\left( 	y\overline{a}y^*\right)\right)^2- \sum_{g \in \ker\left(\widehat{G} \to G_j \right)}\left( g\left( 	y\overline{a}y^*\right)\right)^2\right\|<\frac{\eps}{4}.
\end{split}
\end{equation}
From \eqref{nt_eps_4_1}-\eqref{nt_eps_4_4} it follows than for any $j \ge \N$ following condition holds
		\begin{equation*}
\begin{split}
\left\| b_j^2 - c_j\right\| < \eps,
\end{split}
\end{equation*}	
i.e. $\overline{   a}$ satisfies to the condition (c) of the Definition \ref{special_el_defn}.

\end{proof}
\begin{corollary}\label{nt_norm_compl}
	If $\overline{A}_{\widehat{\pi}^\oplus}$ is the disconnected inverse noncommutative limit of $\mathfrak{S}_\th$ with respect to $\widehat{\pi}^\oplus$ then 
	$$
 \bigoplus_{g \in J} g	C_0\left( \R^{2N}_\th\right) \subset \overline{A}_{\widehat{\pi}^\oplus}
	$$
\end{corollary}
\begin{proof}
	From the Lemma \ref{nt_w_spec_lem} it turns out that $\overline{A}_{\widehat{\pi}^\oplus}$ contains all elements 
	\be\label{mp_fjkel_eqn}
f_{j_1k_1}\otimes \dots \otimes f_{j_Nk_N} \in 	\underbrace{\SS\left(\R^{2}_\th\right)\otimes\dots\otimes\SS\left(\R^{2}_\th\right)}_{N-\mathrm{times}} \subset \SS\left(\R^{2N}_\th\right)
	\ee
where $f_{j_lk_l}$ ($l = 1,\dots, N$) are  given by the Lemma \ref{lm:osc-basis}. However the linear span of given by \eqref{mp_fjkel_eqn} elements is dense in  $C_0\left( \R^{2N}_\th\right)$, hence  $C_0\left( \R^{2N}_\th\right) \subset \overline{A}_{\widehat{\pi}^\oplus}$. From the Corollary \ref{disconnect_group_action_cor} it turns out
		$$
\bigoplus_{g \in J} g	C_0\left( \R^{2N}_\th\right)\subset \overline{A}_{\widehat{\pi}^\oplus}	.	
		$$
\end{proof}
\begin{empt}
	From the Lemma \ref{nt_l_2_est_lem} it turns out that $L^2 \left(\R^{2N}_\th\right) \subset B\left( L^2\left(\R^{2N}\right)\right)$ is a Hilbert space with the norm $\left\|\cdot\right\|_2$ given by \eqref{nt_l2_norm_eqn}. One can construct the Hilbert direct sum 
	\begin{equation*}
	\begin{split}
	X=	\bigoplus_{g \in J} g L^2 \left(\R^{2N}_\th\right) \subset \prod_{g \in J} B\left(g L^2\left(\R^{2N}\right)\right), \\
	X = \left\{\overline{x} \in \prod_{g \in J} B\left( g L^2\left(\R^{2N}\right)\right) ~|~ \left\|\left(...,x_{g_k},...\right)  \right\|_2= \sqrt{\sum_{g \in J}\left\|x_g\right\|^2_2} < \infty\right\}.
	\end{split}
	\end{equation*}
	If $\overline{a} \in X$ is a special element and $b = \sum_{g \in \widehat{G}}g \overline{a}^2 \in C\left(\T^{2N}_\th \right) $
	then 
	$$
	\tau\left(b\right) = \int_{\R^{2N}_g} \overline{a}^2 dx = \left\|\overline{a}\right\|^2_2
	$$ 
	where  $\tau$ is given by \eqref{nt_state_eqn}, or \eqref{nt_varphi_inf_eqn}. On the other hand  $\left|\tau\left(b \right) \right| < \infty$ for any $b \in C\left(\T^{2N}_\th \right)$ it follows that
	$\left\|\overline{a}\right\|^2_2 < \infty$ for a special element $\overline{a}$.
	In result we have the following lemma.
\end{empt}

\begin{lem}\label{nt_l2_spec_lem}
	The special element $\overline{a}\in \varinjlim C\left(\mathbb{T}^{2N}_{\th/m_j^2}\right) $ lies in  $X=\bigoplus_{g \in J} L^2\left( \R^{2N}_\th\right)$.  Moreover if $b = \sum_{g \in \widehat G}g\left( \overline{a}^2\right) \in C\left(\mathbb{T}^{2N}_\th\right)$ then
	$$
	\left\| \overline{a}\right\|^2_2 = \tau\left(b\right) < \infty
	$$
	where $\tau$ is the tracial state on $C\left(\mathbb{T}^{2N}_{\th}\right)$ given by \eqref{nt_state_eqn}, \eqref{nt_state_integ_eqn} and $\left\|\cdot\right\|_2$ is given by \eqref{nt_l2_norm_eqn}.
\end{lem}
\begin{rem}\label{nt_sup_norm}
	From $L^2 \left(\R^{2N}_\th\right) \subset C_0 \left(\R^{2N}_\th\right)$ it follows that any special element in $B\left( L^2 \left(\R^{2N}_\th\right)\right)$ lies in  $C_0 \left(\R^{2N}_\th\right)$. 
\end{rem}



\begin{empt}\label{nt_c_0}
	Let $\overline{A}_{\widehat{\pi}^\oplus}$ be the disconnected inverse noncommutative limit of $\mathfrak{S}_\th$ with respect to $\widehat{\pi}^\oplus$ of $\mathfrak{S}_\th$. From the Corollary \ref{nt_norm_compl} it follows that
	$$
	C_0\left( \R^{2N}_\th\right) \subset \overline{A}_{\widehat{\pi}^\oplus} \bigcap B\left(L^2\left( R^{2N}\right)  \right).
	$$
	 
	From the Remark \ref{nt_sup_norm} it follows that	$$
	\overline{A}_{\widehat{\pi}^\oplus} \bigcap B\left(L^2\left( R^{2N}\right)  \right) \subset C_0\left( \R^{2N}_\th\right).
	$$
		In result we have
	\begin{equation}\label{nt_c_0_eqn}
	\overline{A}_{\widehat{\pi}^\oplus} \bigcap B\left(L^2\left( R^{2N}\right)  \right) = C_0\left( \R^{2N}_\th\right).
	\end{equation}
	Similarly for any $g \in J$ on has
	$$
	\overline{A}_{\widehat{\pi}^\oplus} \bigcap B\left(gL^2\left( R^{2N}\right)  \right) = gC_0\left( \R^{2N}_\th\right).
	$$	
	The algebra $ C_0\left( \R^{2N}_\th\right)$ is irreducible. Clearly $ C_0\left( \R^{2N}_\th\right) \subset \overline{A}_{\widehat{\pi}^\oplus}$ is a maximal irreducible subalgebra.
\end{empt}

\begin{thm}\label{nt_inf_cov_thm} Following conditions hold:
	\begin{enumerate}
		\item[(i)] The representation $\widehat{\pi}^\oplus$ is good,
		\item[(ii)] 
		\begin{equation*}	\begin{split}
		\varprojlim_{\widehat{\pi}^\oplus} \downarrow \mathfrak{S}_\th = C_0\left(\R^{2N}_\th\right); \\
		G\left(\varprojlim_{\widehat{\pi}^\oplus} \downarrow \mathfrak{S}_\th~|~ C\left(\mathbb{T}^{2N}_\th \right)\right)  = \Z^{2N},
		\end{split}
		\end{equation*}
		\item[(iii)] The triple $\left(C\left(\mathbb{T}^{2N}_\th \right), C_0\left(\R^{2N}_\th\right), \Z^{2N} \right)$ is an  infinite noncommutative covering of $\mathfrak{S}_\th$ with respect to $\widehat{\pi}^\oplus$.
	\end{enumerate}

\end{thm}

\begin{proof} (i)
	There is the natural inclusion $\overline{A}_{\widehat{\pi}^\oplus} \hookto \prod_{g \in J} B\left(g L^2\left(\R^{2N}\right)\right)$ where $\prod$ means the Cartesian product of algebras. This inclusion induces the decomposition
	$$
	\overline{A}_{\widehat{\pi}^\oplus} \hookto  \prod_{g \in J}\left( 	\overline{A}_{\widehat{\pi}^\oplus}\bigcap B\left(g L^2\left(\R^{2N}\right)\right) \right). 
	$$
	From \eqref{nt_c_0_eqn} it turns out $\overline{A}_{\widehat{\pi}^\oplus}\bigcap B\left(g L^2\left(\R^{2N}\right)\right) = gC_0\left( \R^{2N}_\th\right)$, hence there is the inclusion
	$$
	\overline{A}_{\widehat{\pi}^\oplus} \hookto  \prod_{g \in J}	 gC_0\left( \R^{2N}_\th\right).
	$$
	
	From the above equation it follows that $C_0\left( \R^{2N}_\th\right)\subset \overline{A}_{\widehat{\pi}^\oplus                                                    }$ is a maximal irreducible subalgebra.
	From the Lemma \ref{nt_l2_spec_lem} it turns out that  algebraic direct sum $\bigoplus_{g \in J}	 gC_0\left( \R^{2N}_\th\right)$ is a dense subalgebra of $\overline{A}_{\widehat{\pi}^\oplus}$,
	i.e. the condition (b) of the Definition \ref{good_seq_defn} holds.
	Clearly the map $\widehat{C\left(\T^{2N}_\th \right)} \to M\left(C_0\left( \R^{2N}_\th\right) \right)$ is injective, i.e. the condition (a) of the Definition \ref{good_seq_defn} holds. If $G \subset \widehat{G}$ is the maximal group such that $GC_0\left( \R^{2N}_\th\right) = C_0\left( \R^{2N}_\th\right)$  then $G = \Z^{2N}$. The homomorphism $\Z^{2N} \to \Z^{2N}_{m_j}$ is surjective, it turns out that the condition (c) of the Definition \ref{good_seq_defn} holds. 
	\newline
	(ii) and (iii) Follows from the proof of (i).
\end{proof}

	\section{Isospectral deformations and their coverings}

\paragraph*{}A very general construction of isospectral
deformations
of noncommutative geometries is described in \cite{connes_landi:isospectral}. The construction
implies in particular that any
  compact spin-manifold $M$ whose isometry group has rank
$\geq 2$ admits a
natural one-parameter isospectral deformation to noncommutative geometries
$M_\theta$.
We let $(\Coo\left(M \right)  , \H = L^2\left(M,S \right)  , \slashed D)$ be the canonical spectral triple associated with a
compact spin-manifold $M$. We recall that $\mathcal{A} = C^\infty(M)$ is
the algebra of smooth
functions on $M$, $S$ is the spinor bundle and $\slashed D$
is the Dirac operator.
Let us assume that the group $\mathrm{Isom}(M)$ of isometries of $M$ has rank
$r\geq2$.
Then, we have an inclusion
\begin{equation*}
\mathbb{T}^2 \subset \mathrm{Isom}(M) \, ,
\end{equation*}
with $\mathbb{T}^2 = \mathbb{R}^2 / 2 \pi \mathbb{Z}^2$ the usual torus, and we let $U(s) , s \in
\mathbb{T}^2$, be
the corresponding unitary operators in $\H = L^2(M,S)$ so that by construction
\begin{equation*}
U(s) \, \slashed D = \slashed D \, U(s).
\end{equation*}
Also,
\begin{equation}\label{isospectral_sym_eqn}
U(s) \, a \, U(s)^{-1} = \alpha_s(a) \, , \, \, \, \forall \, a \in \mathcal{A} \, ,
\end{equation}
where $\alpha_s \in \mathrm{Aut}(\mathcal{A})$ is the action by isometries on the
algebra of functions on
$M$.

\noindent
We let $p = (p_1, p_2)$ be the generator of the two-parameters group $U(s)$
so that
\begin{equation*}
U(s) = \exp(i(s_1 p_1 + s_2 p_2)) \, .
\end{equation*}
The operators $p_1$ and $p_2$ commute with $D$.
Both $p_1$ and $p_2$
have integral spectrum,
\begin{equation*}
\mathrm{Spec}(p_j) \subset \mathbb{Z} \, , \, \, j = 1, 2 \, .
\end{equation*}

\noindent
One defines a bigrading of the algebra of bounded operators in $\H$ with the
operator $T$ declared to be of bidegree
$(n_1,n_2)$ when,
\begin{equation*}
\alpha_s(T) = \exp(i(s_1 n_1 + s_2 n_2)) \, T \, , \, \, \, \forall \, s \in
\mathbb{T}^2 \, ,
\end{equation*}
where $\alpha_s(T) = U(s) \, T \, U(s)^{-1}$ as in \eqref{isospectral_sym_eqn}.
\paragraph{}
Any operator $T$ of class $C^\infty$ relative to $\alpha_s$ (i. e. such that
the map $s \rightarrow \alpha_s(T) $ is of class $C^\infty$ for the
norm topology) can be uniquely
written as a doubly infinite
norm convergent sum of homogeneous elements,
\begin{equation*}
T = \sum_{n_1,n_2} \, \widehat{T}_{n_1,n_2} \, ,
\end{equation*}
with $\widehat{T}_{n_1,n_2}$ of bidegree $(n_1,n_2)$ and where the sequence
of norms $||
\widehat{T}_{n_1,n_2} ||$ is of
rapid decay in $(n_1,n_2)$.
Let $\lambda = \exp(2 \pi i \theta)$. For any operator $T$ in $\H$ of
class $C^\infty$ we define
its left twist $l(T)$ by
\begin{equation}\label{l_defn}
l(T) = \sum_{n_1,n_2} \, \widehat{T}_{n_1,n_2} \, \lambda^{n_2 p_1} \, ,
\end{equation}
and its right twist $r(T)$ by
\begin{equation*}
r(T) = \sum_{n_1,n_2} \, \widehat{T}_{n_1,n_2} \, \lambda^{n_1 p_2} \, ,
\end{equation*}
Since $|\lambda | = 1$ and $p_1$, $p_2$ are self-adjoint, both series
converge in norm. Denote by $\Coo\left(M \right)_{n_1, n_2} \subset \Coo\left(M \right) $ the $\C$-linear subspace of elements of bidegree $\left( n_1, n_2\right) $. \\
One has,
\begin{lem}\label{conn_landi_iso_lem}\cite{connes_landi:isospectral}
\begin{itemize}
\item[{\rm a)}] Let $x$ be a homogeneous operator of bidegree $(n_1,n_2)$
and $y$ be
a homogeneous operator of  bidegree $(n'_1,n'_2)$. Then,
\begin{equation}
l(x) \, r(y) \, - \,  r(y) \, l(x) = (x \, y \, - y \, x) \,
\lambda^{n'_1 n_2} \lambda^{n_2 p_1 + n'_1 p_2}
\end{equation}
In particular, $[l(x), r(y)] = 0$ if $[x, y] = 0$.
\item[{\rm b)}] Let $x$ and $y$ be homogeneous operators as before and
define
\begin{equation}
x * y = \lambda^{n'_1 n_2} \, x y \, ; \label{star}
\end{equation}
then $l(x) l(y) = l(x * y)$.
\end{itemize}
\end{lem}

\noindent
The product $*$ defined in (\ref{star}) extends by linearity
to an associative product on the linear space of smooth operators and could
be called a $*$-product.
One could also define a deformed `right product'. If $x$ is homogeneous of
bidegree
$(n_1,n_2)$ and $y$ is homogeneous of bidegree $(n'_1,n'_2)$ the product is
defined by
\begin{equation*}
x *_{r} y = \lambda^{n_1 n'_2} \, x y \, .
\end{equation*}
Then, along the lines of the previous lemma one shows that $r(x) r(y) = r(x
*_{r} y)$.

We can now define a new spectral triple where both $\H$ and the operator
$D$ are unchanged while the
algebra $\Coo\left(M \right)$  is modified to $l(\Coo\left(M \right))$ . By
Lemma~{\ref{conn_landi_iso_lem}}~b) one checks that  $l\left( \Coo\left(M \right)\right) $ is still an algebra. Since $D$ is of bidegree $(0,0)$ one has,
\begin{equation*}
[D, \, l(a) ] = l([D, \, a]) \label{bound}
\end{equation*}
which is enough to check that $[D, x]$ is bounded for any $x \in l(\mathcal{A})$. There is a spectral triple $\left(l\left( \Coo\left(M \right)\right) , \H, D\right)$.
\paragraph{} Denote by $C\left(M_\th \right)$ the operator norm completion (equivalently $C^*$-norm completion) of $l\left(\Coo\left( M\right)  \right)  $, and denote by $\rho: C\left(M\right) \to L^2\left( M, S\right) $ (resp. $\pi_\th: C\left(M_\th\right) \to B\left( L^2\left( M, S\right)\right) $ ) natural representations.

\subsection{Finite-fold coverings}\label{isosectral_fin_cov}
\paragraph{}
Let $M$ be a spin - manifold with the smooth action of $\T^2$.
Let $\pi:\widetilde{M} \to M$ be a finite-fold covering. Let $\widetilde{x}_0 \in \widetilde{M}$ and $x_0=\pi\left(\widetilde{x}_0 \right)$. Denote by $\varphi: \R^2 \to \R^2 / \Z^2 = \T^2$ the natural covering. There are two closed paths $\om_1, \om_2: \left[0,1 \right]\to M$ given by
\begin{equation*}
\begin{split}
\om_1\left(t \right) = \varphi\left(t, 0 \right) x_0,~
\om_2\left(t \right) = \varphi\left(0, t \right) x_0.
\end{split}
\end{equation*} 
There are  lifts of these paths, i.e. maps $\widetilde{\om}_1 , \widetilde{\om}_2: \left[0,1 \right] \to\widetilde{M}$ such that
\begin{equation*}
\begin{split}
\widetilde{\om}_1\left(0 \right)= \widetilde{\om}_2\left(0 \right)=\widetilde{x}_0,~ \\
\pi\left( \widetilde{\om}_1\left(t \right)\right)  = \om_1\left(t\right),\\
\pi\left( \widetilde{\om}_2\left(t \right)\right)  = \om_2\left(t\right).
\end{split}
\end{equation*}
Since $\pi$ is a finite-fold covering there are $N_1, N_2 \in \N$ such that if 
$$
\gamma_1\left(t \right) = \varphi\left(N_1t, 0 \right) x_0,~
\gamma_2\left(t \right) = \varphi\left(0, N_2t \right) x_0.
$$
and $\widetilde{\gamma}_1$ (resp. $\widetilde{\gamma}_2$) is the lift of $\gamma_1$ (resp. $\gamma_2$) then both $\widetilde{\gamma}_1$, $\widetilde{\gamma}_2$ are closed. Let us select minimal positive values of $N_1, N_2$. If $\text{pr}_n: S^1 \to S^1$ is an $n$ listed covering and $\text{pr}_{N_1, N_2}$ the covering given by
$$
\widetilde{\T}^2 = S^1 \times S^1 \xrightarrow{\text{pr}_{N_1}\times\text{pr}_{N_2}} \to S^1 \times S^1 = \T^2
$$
then there is the action $\widetilde{\T}^2 \times \widetilde{M} \to  \widetilde{M}$ such that

  	  	\begin{tikzpicture}
  	  	\matrix (m) [matrix of math nodes,row sep=3em,column sep=4em,minimum width=2em]
  	  	{
\widetilde{\T}^2 \times \widetilde{M}  	&  	 & \widetilde{M} \\
\T^2. \times M 	 & 	 & M   \\};
  	  	\path[-stealth]
  	  	(m-1-1) edge node [above] {$ \ $} (m-1-3)
  	  	(m-1-1) edge node [right] {$\mathrm{pr}_{N_1N_2} \times \pi $} (m-2-1)
  	  	(m-1-3) edge node [right] {$\pi$} (m-2-3)
    	(m-2-1) edge node [above] {$ \ $} (m-2-3);
    	  
  	  	\end{tikzpicture}

 where  $\widetilde{\T}^2 \approx \T^2$.
Let $\widetilde{p} = \left( \widetilde{p}_1, \widetilde{p}_2\right) $ be the generator of the associated with $\widetilde{\T}^2$ two-parameters group $\widetilde{U}\left(s \right) $
 so that
 \begin{equation*}
\widetilde{U}\left(s \right) = \exp\left( i\left( s_1 \widetilde{p}_1 + s_2 \widetilde{p}_2\right)\right).
 \end{equation*}	
 The covering $\widetilde{M} \to M$ induces an involutive injective homomorphism
 \begin{equation*}
 \varphi:\Coo\left(M \right)\hookto\Coo\left( \widetilde{M} \right).
  \end{equation*}
 
 Since $\widetilde{M} \to M$ is a covering $\Coo\left( \widetilde{M} \right)$ is a finitely generated projective $\Coo\left( M \right)$-module, i.e. there is the following direct sum of $\Coo\left( \widetilde{M} \right)$-modules
  \begin{equation}\label{isosectral_proj_eqn}
\Coo\left( \widetilde{M} \right) \bigoplus P = \Coo\left( M \right)^n
 \end{equation}
 such that
  \begin{equation*}
 \varphi\left(\Coo\left(M \right) \right)_{n_1,n_2} \subset \Coo\left( \widetilde{M} \right)_{n_1N_1,~ n_2N_2}.
  \end{equation*}

Let $\th, \widetilde{\th} \in \R$ be such that
$$
\widetilde{\th}= \frac{\th +  n}{N_1N_2}, \text{ where }n \in \Z.
$$
If $\lambda= e^{2\pi i \th}$, $\widetilde{\lambda}= e^{2\pi i \widetilde{\th}}$ then
$
 \lambda = \widetilde{\lambda}^{N_1N_2}.
$
There are isospectral deformations $\Coo\left(M_\th \right), \Coo\left( \widetilde{M}_{\widetilde{\th}} \right)$ and $\C$-linear isomorphisms
$l:\Coo\left(M \right) \to \Coo\left(M_\th \right)$, $\widetilde{l}:\Coo\left( \widetilde{M} \right) \to \Coo\left( \widetilde{M}_{\widetilde{\th}} \right)$.
These isomorphisms and the inclusion $\varphi$ induce the inclusion
\begin{equation*}
\begin{split}
\varphi_\th:\Coo\left(M_\th \right)\to\Coo\left( \widetilde{M}_{\widetilde{\th}} \right),
\\
 \varphi_{\widetilde{\th}}\left(\Coo\left(M_\th \right) \right)_{n_1,n_2} \subset \Coo\left( \widetilde{M}_{\widetilde{\th}} \right)_{n_1N_1,~ n_2N_2}.
 \end{split}
  \end{equation*}
 
From \eqref{isosectral_proj_eqn} it follows that
\begin{equation*}
\begin{split}
\widetilde{l}\left( \Coo\left( \widetilde{M} \right)\right)  \bigoplus l\left( P \right) = l\left( \Coo\left( M \right)\right) ^n, \\
\text{or equivalently }\Coo\left( \widetilde{M}_{\widetilde{\th}} \right) \bigoplus l\left(P \right) = \Coo\left(M_\th \right)^n,
\end{split}
\end{equation*}
i.e. $\Coo\left( \widetilde{M}_{\widetilde{\th}} \right)$ is a finitely generated projective $\Coo\left(M_\th \right)$ module. There is a projection $p \in \mathbb{M}_n\left(\Coo\left(M_\th \right)\right)$ such that
$$
\Coo\left( \widetilde{M}_{\widetilde{\th}} \right) = p \Coo\left(M_\th \right)^n
$$ 
If $C\left(  \widetilde{M}_{\widetilde{\th}} \right)$ (resp. $C\left(M_\th \right)$ ) is the operator norm completion of  $\Coo\left( \widetilde{M}_{\widetilde{\th}} \right)$ (resp. $\Coo\left(M_\th \right)$ ) then
$$
C\left( \widetilde{M}_{\widetilde{\th}} \right) = p C\left(M_\th \right)^n,
$$ 
i.e. $C\left(  \widetilde{M}_{\widetilde{\th}} \right)$ is a finitely generated projective $C\left( M_\th \right)$ module.
Denote by $G = G\left(\widetilde{M}~|~M \right)$ the group of covering transformations.   Since $\widetilde{l}$ is a $\C$-linear isomorphism the action of $G$ on $\Coo\left( \widetilde{M} \right)$ induces a $\C$-linear action  $G \times \Coo\left( \widetilde{M}_{ \widetilde{\th}} \right)  \to \Coo\left( \widetilde{M}_{ \widetilde{\th}} \right)$. According to the definition of the action of $\widetilde{\T}^2$ on $\widetilde{M}$ it follows that the action of $G$ commutes with the action of $\widetilde{\T}^2$.
It turns out
$$
g \Coo\left( \widetilde{M} \right)_{n_1,n_2} = \Coo\left( \widetilde{M} \right)_{n_1,n_2}
$$
for any $n_1, n_2 \in \Z$ and $g \in G$.
If $\widetilde{a} \in \Coo\left( \widetilde{M} \right)_{n_1,n_2}$, $\widetilde{b} \in \Coo\left( \widetilde{M} \right)_{n'_1,n'_2}$ then  $g\left( \widetilde{a}\widetilde{b}\right)= \left(g\widetilde{a} \right) \left(g\widetilde{b} \right)\in \Coo\left( \widetilde{M} \right)_{n_1+n'_1,n_2+n'_2} $. One has
\begin{equation*}
\begin{split}
\widetilde{l}\left(\widetilde{a}\right)\widetilde{l}\left(\widetilde{b}\right)= \widetilde{\la}^{n'_1n_2}\widetilde{l}\left(\widetilde{a}\widetilde{b}\right), \\
\widetilde{\la}^{n_2\widetilde{p}_1}l\left( \widetilde{b}\right) = \widetilde{\la}^{n'_1n_2}l\left( \widetilde{b}\right) \widetilde{\la}^{n_2\widetilde{p}_1},\\
\widetilde{l}\left(g \widetilde{a}\right)\widetilde{l}\left(g \widetilde{b}\right)= g \widetilde{a}\widetilde{\la}^{n_2\widetilde{p}_1}g \widetilde{b}\widetilde{\la}^{n'_2\widetilde{p}_1}= \widetilde{\la}^{n'_1n_2} g\left(\widetilde{a}\widetilde{b} \right) \widetilde{\la}^{\left( n_2+n_2'\right) \widetilde{p}_1}.
\end{split}
\end{equation*}
On the other hand
\begin{equation*}
\begin{split}
g\left( \widetilde{l}\left(\widetilde{a}\right)\widetilde{l}\left(\widetilde{b}\right)\right) = g\left( \widetilde{\la}^{n'_1n_2}\widetilde{l}\left(\widetilde{a}\widetilde{b}\right)\right)= \widetilde{\la}^{n'_1n_2} g\left(\widetilde{a}\widetilde{b} \right) \widetilde{\la}^{\left( n_2+n_2'\right) \widetilde{p}_1}. 
\end{split}
\end{equation*}
From above equations it turns out
$$
\widetilde{l}\left(g \widetilde{a}\right)\widetilde{l}\left(g \widetilde{b}\right) = g\left( \widetilde{l}\left(\widetilde{a}\right)\widetilde{l}\left(\widetilde{b}\right)\right),
$$
i.e. $g$ corresponds to automorphism of $\Coo\left( \widetilde{M}_{ \widetilde{\th}}\right)$. It turns out that $G$ is the group of automorphisms of $\Coo\left( \widetilde{M}_{ \widetilde{\th}}\right)$. Clearly form $\widetilde{a} \in \Coo\left( \widetilde{M}_{ \widetilde{\th}}\right)_{n_1,n_2}$ it follows that $\widetilde{a}^* \in \Coo\left( \widetilde{M}_{ \widetilde{\th}}\right)_{-n_1,-n_2}$. One has
$$
g\left(\left( \widetilde{l}\left(\widetilde{a}\right)\right)^* \right) =  g\left( \widetilde{\la}^{-n_2\widetilde{p_1}}\widetilde{a}^*\right) =
g \left(\widetilde{\la}^{n_1 n_2} \widetilde{a}^*\widetilde{\la}^{-n_2\widetilde{p_1}}\right) = \widetilde{\la}^{n_1 n_2} g\left(\widetilde{l}\left(\widetilde{a}^* \right)  \right). 
$$
On the other hand
$$
\left(g \widetilde{l}\left(\widetilde{a}\right) \right)^*= \left(\left( g \widetilde{a}\right)\widetilde{\la}^{n_2\widetilde{p_1}}  \right)^*=\widetilde{\la}^{-n_2\widetilde{p_1}}\left(ga^* \right) = \widetilde{\la}^{n_1 n_2}\left(ga^* \widetilde{\la}^{-n_2\widetilde{p_1}}\right)= \widetilde{\la}^{n_1 n_2} g\left(\widetilde{l}\left(\widetilde{a}^* \right)  \right),
$$
i.e. $g\left(\left( \widetilde{l}\left(\widetilde{a}\right)\right)^* \right) = \left(g \widetilde{l}\left(\widetilde{a}\right) \right)^*$.
It follows that $g$ corresponds to the involutive automorphism of $\Coo\left( \widetilde{M}_{ \widetilde{\th}}\right)$. Since  $\Coo\left( \widetilde{M}_{ \widetilde{\th}}\right)$ is dense in $C\left( \widetilde{M}_{ \widetilde{\th}}\right)$ there is the unique involutive action $G \times C\left( \widetilde{M}_{ \widetilde{\th}}\right) \to C\left( \widetilde{M}_{ \widetilde{\th}}\right)$. From the above construction it turns out the following theorem.
\begin{thm}\label{isospectral_fin_thm}
	The triple $\left( C\left(M_\th\right), C\left( \widetilde{M}_{ \widetilde{\th}}\right), G\left(\widetilde{M}~|~ M \right)\right)   $ is an unital noncommutative finite-fold  covering.
	\end{thm}

\subsection{Infinite coverings}
\paragraph{} Let $\mathfrak{S}_M =\left\{M = M^0 \leftarrow M^1 \leftarrow ... \leftarrow M^n \leftarrow ... \right\} \in \mathfrak{FinTop}$ be an infinite sequence of spin  - manifolds and regular finite-fold covering. Suppose that there is the action $\T^2 \times M \to M$ given by \eqref{isospectral_sym_eqn}. From the Theorem \ref{isospectral_fin_thm} it follows that there is the  algebraical  finite covering sequence
\begin{equation*}\label{isospectral_sequence_eqn}
\mathfrak{S}_{C\left(M_\th \right) }  = \left\{C\left(M_\th \right)\to ... \to C\left(M^n_{\th_n} \right)\to ...\right\}.
\end{equation*}

So one can calculate a finite noncommutative limit of the above sequence. This article does not contain detailed properties of this noncommutative limit, because it is not known yet by the author of this article.

\section*{Acknowledgment} 
\paragraph*{}Author would like to acknowledge members of the Moscow State University Seminar ``Algebras in analysis''
leaded by professor A. Ya. Helemskii and others
for a discussion of this work.

\end{document}